\documentclass[11pt]{amsart}
\usepackage{latexsym}
\usepackage{amssymb}
\usepackage{amscd}
\usepackage{mathrsfs}
\usepackage{rotating}
\renewcommand\a{\alpha}

\renewcommand\d{\delta}
\newcommand\la{\lambda}
\newcommand\z{\zeta}
\newcommand\e{\eta}
\renewcommand\th{\theta}

\newcommand\s{\sigma}
\newcommand\x{\chi}
\newcommand\f{\phi}
\newcommand\vf{\varphi}
\newcommand\p{\psi}
\renewcommand\t{\tau}
\renewcommand\r{\rho}

\newcommand\Om{\Omega}
\newcommand\w{\omega}

\newcommand\vL{\varLambda}

\newcommand{\vT}{\varTheta}

\newcommand\ve{\varepsilon}

\newcommand\Fq{{\mathbf F}_q}

\newcommand\Ql{\bar{\mathbf Q}_l}

\newcommand\BQ{\mathbf Q}
\newcommand\BF{\mathbf F}

\newcommand\BZ{\mathbf Z}

\newcommand\BH{\mathbf H}

\newcommand\BV{\mathbf V}

\newcommand\Ba{\mathbf a}

\newcommand\Bm{\mathbf m}

\newcommand\Bv{\mathbf v}
\newcommand\Bk{\mathbf k}

\newcommand\Bla{\boldsymbol\lambda}

\newcommand\Bmu{\boldsymbol\mu}
\newcommand\Bnu{\boldsymbol\nu}

\newcommand\CE{\mathcal{E}}

\newcommand\CP{\mathcal{P}}

\newcommand\CX{ \mathcal{X}}
\newcommand\CY{ \mathcal{Y}}

\newcommand\SB{\mathscr{B}}
\newcommand\SC{\mathscr{C}}
\newcommand\SD{\mathscr{D}}
\newcommand\SE{\mathscr{E}}

\newcommand\SL{\mathscr{L}}
\newcommand\SM{\mathscr{M}}
 
\newcommand\SO{\mathscr{O}}
\newcommand\SP{\mathscr{P}}
\newcommand\SQ{\mathscr{Q}}

\newcommand\SH{\mathscr{H}}

\newcommand\SX{\mathscr{X}}
\newcommand\SY{\mathscr{Y}}

\newcommand\FS{\mathfrak S}

\newcommand\iv{^{-1}}
\newcommand\wh{\widehat}
\newcommand\wt{\widetilde}
\newcommand\wg{^{\wedge}}

\newcommand\ol{\overline}

\newcommand\lra{\leftrightarrow}

\newcommand\vl{\prec}
\newcommand\vg{\succ}
\newcommand\vle{\preceq}

\newcommand\IC{\operatorname{IC}}

\newcommand\Hom{\operatorname{Hom}}
\newcommand\End{\operatorname{End}}
\newcommand\Aut{\operatorname{Aut}}

\newcommand\Ind{\operatorname{Ind}}

\newcommand\Tr{\operatorname{Tr}\,}
\newcommand\ch{\operatorname{ch}}

\newcommand\ad{\operatorname{ad}}

\newcommand\reg{_{\operatorname{reg}}}

\newcommand\unip{\operatorname{uni}}
\newcommand\uni{_{\operatorname{uni}}}

\newcommand\id{\operatorname{id}}

\newcommand\lp{\operatorname{\!\langle\!}}
\newcommand\rp{\operatorname{\!\rangle\!}}
\renewcommand\Im{\operatorname{Im}}

\newcommand\Diag{\operatorname{Diag}}

\newcommand\dw{\dot w}

\newcommand{\isom}{\,\raise2pt\hbox{$\underrightarrow{\sim}$}\,}
\numberwithin{equation}{section}

\newtheorem{thm}{Theorem}[section]
\newtheorem{lem}[thm]{Lemma}
\newtheorem{cor}[thm]{Corollary}
\newtheorem{prop}[thm]{Proposition}

\def \para#1{\par\medskip\textbf{#1}
              \addtocounter{thm}{1}}

\def \remark#1{\par\medskip\noindent
                \textbf{Remark #1}
                \addtocounter{thm}{1}}

\def \remarks#1{\par\medskip\noindent
                \textbf{Remarks #1}
                \addtocounter{thm}{1}}

\begin{document}
\setlength{\baselineskip}{4.9mm}
\setlength{\abovedisplayskip}{4.5mm}
\setlength{\belowdisplayskip}{4.5mm}
%%%
%%%
\renewcommand{\theenumi}{\roman{enumi}}
\renewcommand{\labelenumi}{(\theenumi)}
\renewcommand{\thefootnote}{\fnsymbol{footnote}}
%%%
\renewcommand{\thefootnote}{\fnsymbol{footnote}}
%%%
\allowdisplaybreaks[2]
%\NoBlackBoxes
\parindent=20pt
%%%%%%%%%%%%%%%%%%%%
%%%%%%%%%%%%%%%%%%%%%%%%%%%%%%%%%%%
\medskip
\begin{center}
 {\bf Enhanced variety of higher level and  \\
     Kostka functions associated to complex reflection groups}

\par\medskip
\par
\vspace{1cm}
Toshiaki Shoji

%\vspace{0.8cm}

%{\it To the memory of J. A. Green}

%\vspace{0.7cm}
\title{}
\end{center}

\begin{abstract}
Let $V$ be an $n$-dimensional vector space over an algebraic closure of a finite field 
$\BF_q$, and 
$G = GL(V)$.  A variety $\SX = G \times V^{r-1}$ is called an enhanced variety 
of level $r$. Let $\SX\uni = G\uni \times V^{r-1}$ be the unipotent variety of 
$\SX$. We have a partition 
$\SX\uni = \coprod_{\Bla}X_{\Bla}$ indexed by $r$-partitions $\Bla$ of $n$.  
In the case where $r = 1$ or 2, $X_{\Bla}$ is a single $G$-orbit, but if $r \ge 3$, 
$X_{\Bla}$ is, in general, a union of infinitely many $G$-orbits.  
In this paper, we prove certain orthogonality relations 
for the characteristic functions (over $\BF_q$) of the intersection cohomology 
$\IC(\ol X_{\Bla},\Ql)$, and show some results, 
which suggest a close relationship between those characteristic functions and 
Kostka functions associated to the complex reflection group $S_n\ltimes (\BZ/r\BZ)^n$.   
\end{abstract}

\maketitle
\markboth{SHOJI}{KOSTKA POLYNOMIALS}
\pagestyle{myheadings}

\begin{center}
{\sc Introduction}
\end{center}
Let $V$ be an $n$-dimensional vector space over an algebraic closure of a finite field 
$\BF_q$, 
and $G = GL(V) \simeq GL_n$.  In 1981, Lusztig showed in [L1] that Kostka polynomials 
$K_{\la,\mu}(t)$ have a geometric interpretation in terms of the intersection cohomology associated to
the closure of unipotent classes in $G$ in the following sense.   
Let $C_{\la}$ be the unipotent class corresponding to a partition $\la$ of $n$, 
and $K = \IC(\ol C_{\la}, \Ql)$ be the intersection cohomology complex on the closure 
$\ol C_{\la}$ of $C_{\la}$.
He proved that $\SH^iK = 0$ for odd $i$, and that for partitions $\la, \mu$ of $n$,  
\begin{equation*}
\tag{*}
t^{n(\mu)}K_{\la,\mu}(t\iv) = t^{n(\la)}\sum_{i \ge 0}(\dim \SH^{2i}_xK)t^i, 
\end{equation*}
where $x \in C_{\mu} \subset \ol C_{\la}$, and $n(\la)$ is the usual $n$-function. 
\par
Kostka polynomials are polynomials indexed by a pair of partitions.  In [S1], [S2],
as a generalization of Kostka polynomials, 
Kostka functions $K_{\Bla, \Bmu}(t)$ associated to the complex reflection group 
$S_n \ltimes (\BZ/r\BZ)^n$ were introduced, which are a-priori rational functions 
in $\BQ(t)$ indexed by $r$-partitions $\Bla, \Bmu$ of $n$ (see 3.10 for the definition of 
$r$-partitions of $n$).  
It is known by [S2] that $K_{\Bla, \Bmu}(t)$ are actually polynomials if $r = 2$. 
 Although those Kostka functions 
are defined in a purely combinatorial way, and have no geometric background, recently
various generalizations of Lusztig's result for those Kostka functions were found. 
Under the notation above, consider a variety $\SX = G \times V$, which  is called the enhanced variety, 
and its subvariety $\SX\uni = G\uni \times V$ is isomorphic to the enhanced nilpotent cone introduced by 
Achar-Henderson [AH] (here $G\uni$ is the unipotent variety of $G$).  
The set of $G$-orbits under the diagonal action of $G$ on $\SX\uni$ 
is parametrized 
by double partitions of $n$ ([AH], [T]).  In [AH], they proved that Kostka polynomials indexed by 
double partitions have a geometric interpretation as in (*) in terms of the intersection cohomology 
associated to the closure of $G$-orbits in $\SX\uni$.   
\par
On the other hand, let $V$ be a $2n$-dimensional symplectic 
vector space over an algebraic closure of $\BF_q$ with $\ch \BF_q \ne 2$, 
and consider $G = GL(V) \supset H = Sp(V)$.  The variety 
$\SX = G/H \times V$ is called the exotic symmetric space, and its 
``unipotent variety'' $\SX\uni$ is isomorphic to the exotic nilpotent 
cone introduced by Kato [K1]. $H$ acts diagonally on $\SX\uni$, and 
the set of $H$-orbits on $\SX\uni$ is parametrized by double partitions of $n$ ([K1]).
As in the enhanced case, it is proved by [K2], and [SS1], [SS2], independently, 
that Kostka polynomials indexed by double partitions have a geometric interpretation 
in terms of the intersection cohomology associated to the closure of $H$-orbits in 
$\SX\uni$.   
\par
As a generalization of the enhanced variety $G \times V$ or the exotic symmetric 
space $G/H \times V$, we consider $G \times V^{r-1}$ or  $G/H \times V^{r-1}$
for any $r \ge 1$. 
$G$ acts diagonally on $G \times V^{r-1}$, and $H$ acts diagonally on $G/H \times V^{r-1}$.
$\SX = G \times V^{r-1}$ is called the enhanced variety of level $r$, and 
a certain $H$-stable subvariety $\SX$ of $G/H \times V^{r-1}$ 
is called the exotic symmetric space of level $r$.
For those varieties $\SX$, one can consider $G$-stable subvariety $\SX\uni$ (unipotent variety). 
The crucial difference for the general case is that the number of $G$-orbits (or $H$-orbits)
on $\SX\uni$ is no longer finite if $r \ge 3$.   
Nevertheless, it was shown in [S3], for the exotic case or the enhanced case, 
 that one can construct subvarieties $X_{\Bla}$ of $\SX\uni$ indexed by $r$-partitions $\Bla$ 
of $n$, and the intersection cohomologies $\IC(\ol X_{\Bla}, \Ql)$ enjoy similar properties 
as in the case $r = 1$ or 2, more precisely, a generalization of the Springer correspondence 
holds for $\SX\uni$.   
\par
So it is natural to expect that those intersection cohomologies will 
have a close relation with Kostka functions indexed by $r$-partitions of $n$. 
In this paper, we consider this problem  in the case where $\SX$ is the enhanced 
variety of level $r$.  In this case, we have a partition $\SX\uni = \coprod_{\Bla}X_{\Bla}$ 
parametrized by $r$-partitions $\Bla$ of $n$.  In the case where $r = 1$ or 2, 
$X_{\Bla}$ is a single $G$-orbit, but if $r \ge 3$, $X_{\Bla}$ is, in general, a union of
infinitely many $G$-orbits.  
By applying the strategy employed in the theory 
of character sheaves in [L2], [L3] (and in [SS2]),  
we show (Theorem 5.5) that the characteristic functions of the Frobenius trace over $\BF_q$ 
associated to 
$\IC(\ol X_{\Bla}, \Ql)$ satisfy certain orthogonality relations, which are quite similar 
to the case of $r = 1$ or 2.  In fact, in the case where $r = 1$ or 2, the geometric interpretation 
is deduced from this kind of orthogonality relations. 
However, in the case where $r \ge 3$, 
these orthogonality relations are not
enough to obtain the required formula.  
We show some partial results which suggest an interesting  relationship between those 
characteristic functions and Kostka functions.   
\par
The author is very grateful to Jean Michel for making the GAP program for computing 
Kostka functions.  The examples in 7.3 and 7.4 were computed by his program.

\par\bigskip
\section{Complexes on the enhanced variety}

\para{1.1.}
Let $\Bk$ be an algebraic closure of a finite field $\Fq$. Let $V$ be an $n$-dimensional
vector space over $\Bk$, and put $G = GL(V)$. 
For a fixed integer $r \ge 2$, we consider the variety $G \times V^{r-1}$, on which 
$G$ acts diagonally. We call $G \times V^{r-1}$ the enhanced variety of level $r$. 
Put $\SQ_{n,r} = \{ \Bm = (m_1, \dots, m_r) \in \BZ^r_{\ge 0} \mid \sum_im_i = n\}$.
For each $\Bm \in \SQ_{n,r}$, we define integers $p_i = p_i(\Bm)$ by 
$p_i = m_1 + \cdots + m_i$ for $i =1, \dots, r$.  
Let $B = TU$ be a Borel subgroup of $G$, $T$ a maximal torus of $B$, and $U$ the unipotent
radical of $B$. Let $(M_i)_{1 \le i \le n}$ be the total flag in $V$ such that the stabilizer 
of $(M_i)$ in $G$ coincides with $B$.  
We define varieties 
\begin{align*}
\wt\SX_{\Bm} &= \{ (x, \Bv, gB) \in G \times V^{r-1} \times G/B \mid g\iv xg \in B, 
                      g\iv \Bv \in \prod_{i=1}^{r-1}M_{p_i} \}, \\
\SX_{\Bm} &=  \bigcup_{g \in G}g(B \times \prod_{i=1}^{r-1}M_{p_i}),
\end{align*}
and the map $\pi_{\Bm} : \wt \SX_{\Bm} \to \SX_{\Bm}$ by $(x,\Bv, gB) \mapsto (x,\Bv)$.
Then $\wt\SX_{\Bm}$ is smooth and irreducible, and $\pi_{\Bm}$ is a proper surjective map.
In particular, $\SX_{\Bm}$ is a closed subvariety of $G \times V^{r-1}$.
In the case where $\Bm = (n, 0, \dots, 0)$, we denote $\wt\SX_{\Bm}, \SX_{\Bm}$ by 
$\wt\SX, \SX$. Hence $\SX = G\times V^{r-1}$. 
\par
Let $G\uni$ be the set of unipotent elements in $G$.  Similarly to the above, we define 
\begin{align*}
\wt\SX_{\Bm, \unip} &= \{ (x, \Bv, gB) \in G \times V^{r-1} \times G/B \mid g\iv xg \in U, 
                      g\iv \Bv \in \prod_{i=1}^{r-1}M_{p_i} \}, \\
\SX_{\Bm,\unip} &=  \bigcup_{g \in G}g(U \times \prod_{i=1}^{r-1}M_{p_i}),
\end{align*}
and the map 
$\pi_{\Bm,1}: \wt\SX_{\Bm,\unip} \to \SX_{\Bm,\unip}$ by $(x, \Bv, gB) \mapsto (x, \Bv)$.
Note that $\wt\SX_{\Bm,\unip} = \pi_{\Bm}\iv(\SX_{\Bm,\unip})$, and $\pi_{\Bm,1}$
is the restriction of $\pi_{\Bm}$ on $\wt\SX_{\Bm,\unip}$. 
Hence $\pi_{\Bm,1}$ is proper, and $\SX_{\Bm, \unip}$ is a closed subvariety of 
$G\uni \times V^{r-1}$.  
In the case where $\Bm = (n, 0, \dots,0)$, we denote $\wt\SX_{\Bm,\unip}, \SX_{\Bm, \unip}$
by $\wt\SX\uni, \SX\uni$. 
\par
Let $T\reg$ be the set of regular semisimple elements in $T$, and put 
$G\reg = \bigcup_{g \in G}gT\reg g\iv$, $B\reg = G\reg \cap B$. 
We define varieties

\begin{align*}
\wt\SY_{\Bm} &= \{ (x,\Bv, gB) \in G\reg \times V^{r-1} \times G/B
                    \mid g\iv xg \in B\reg, g\iv\Bv \in \prod_i M_{p_i} \}, \\
   \SY_{\Bm} &= \bigcup_{g \in G}g(B\reg \times \prod_iM_{p_i})
                  = \bigcup_{g \in G}g(T\reg \times \prod_iM_{p_i}).  
\end{align*} 
Then $\wt\SY_{\Bm} = \pi_{\Bm}\iv(\SY_{\Bm})$ and we define 
the map $\psi_{\Bm} : \wt\SY_{\Bm} \to \SY_{\Bm}$ by the restriction of $\pi_{\Bm}$ 
on $\wt\SY_{\Bm}$. 
It is known that 

\begin{lem}[{[S3, Lemma 4.2]}]  %%%%  Lemma 1.2
\begin{enumerate}
\item
$\SY_{\Bm}$ is open dense in $\SX_{\Bm}$ and $\wt\SY_{\Bm}$ is open dense in $\wt\SX_{\Bm}$. 
\item
$\dim \SX_{\Bm} = \dim \wt\SX_{\Bm} = n^2 + \sum_{i=1}^r(r-i)m_i$.
\end{enumerate}
\end{lem}

\para{1.3.}
We fix a basis $e_1, \dots, e_n$ of $V$ such that $e_i$ are weight vectors of $T$ 
and that $M_i$ is spanned by $e_1, \dots, e_i$. Let $W = N_G(T)/T$ be the Weyl group of 
$G$, which is isomorphic to the permutation group $S_n$ of the basis $\{e_1, \dots, e_n\}$.
We denote by $W_{\Bm}$ the subgroup of $W$ which permutes the basis $\{e_j\}$ of $M_{p_i}$
for each $i$.  Hence $W_{\Bm}$ is isomorphic to the Young subgroup 
$S_{\Bm} = S_{m_1} \times \cdots \times S_{m_r}$ of $S_n$. 
Let $M_{p_i}^0$ be the set of $v = \sum_ja_je_j \in M_{p_i}$ such that $a_j \ne 0$ for 
$p_{i-1} + 1 \le j \le p_i$. We define a variety $\wt\SY_{\Bm}^0$ by

\begin{equation*}
\wt\SY_{\Bm}^0 = G \times^T(T\reg \times \prod_iM_{p_i}^0). 
\end{equation*} 
Since $\wt\SY_{\Bm} \simeq G \times^T(T\reg \times \prod_iM_{p_i})$, $\wt\SY_{\Bm}^0$ 
is identified with the open dense subset of $\wt\SY_{\Bm}$. 
Then $\SY_{\Bm}^0 = \psi_{\Bm}(\wt\SY^0_{\Bm})$ is an open dense smooth subset of $\SY_{\Bm}$. 
The map $\psi_{\Bm}^0 : \wt\SY_{\Bm}^0 \to \SY^0_{\Bm}$ 
obtained from the restriction of 
$\psi_{\Bm}$ turns out to be a finite Galois covering with group $W_{\Bm}$
(apply the discussion in [S3, 1.3] to the enhanced case).
\par
We consider the diagram

\begin{equation*}
\begin{CD}
T  @<\a <<  \wt\CX_{\Bm} @>\pi_{\Bm}>>  \CX_{\Bm},   
\end{CD}
\end{equation*} 
where $\a$ is the map defined by $(x,\Bv, gB) \mapsto (p_T(g\iv xg))$ 
($p_T : B \to T$ is the natural projection).  Let $\SE$ be a tame local system on $T$.
We denote by $W_{\Bm, \SE}$ the stabilizer of $\SE$ in $W_{\Bm}$.
Let $\a_0$ be the restriction of $\a$ on $\wt\SY_{\Bm}$.  We also denote by $\a_0$ 
the restriction of $\a$ on $\wt\SY_{\Bm}^0$. Since $\psi_{\Bm}^0$ is a finite Galois 
covering, $(\psi^0_{\Bm})_!\a_0^*\SE$ is a local system  on  $\SY^0_{\Bm}$ 
equipped with $W_{\Bm,\SE}$-action, 
and is decomposed as 

\begin{equation*}
\tag{1.3.1}
(\psi^0_{\Bm})_!\a_0^*\SE \simeq \bigoplus_{\r \in W_{\Bm,\SE}\wg} \r \otimes \SL_{\r},
\end{equation*}
where $\SL_{\r} = \Hom (\r, (\psi^0_{\Bm})_!\a_0^*\SE)$ is 
the simple local system  on $\CY^0_{\Bm}$. 
The following results were proved in Proposition 4.3 and Theorem 4.5 in [S3].

\begin{thm}[{[S3]}] %%%%  Theorem  1.4.
Take $\Bm \in \SQ_{n,r}$, and put $d_{\Bm} = \dim \SX_{\Bm}$. 
\begin{enumerate}
\item
$(\psi_{\Bm})_!\a_0^*\SE[d_{\Bm}]$ is a semisimple perverse sheaf on $\SY_{\Bm}$
equipped with $W_{\Bm,\SE}$-action, and is decomposed as

\begin{equation*}
\tag{1.4.1}
(\psi_{\Bm})_!\a_0^*\SE[d_{\Bm}] \simeq \bigoplus_{\r \in W_{\Bm,\SE}\wg}
                    \r \otimes \IC(\SY_{\Bm}, \SL_{\r})[d_{\Bm}].
\end{equation*}
\item
$(\pi_{\Bm})_!\a^*\SE[d_{\Bm}]$ is a semisimple perverse sheaf on $\SX_{\Bm}$ 
equipped with $W_{\Bm,\SE}$-action, and is decomposed as 

\begin{equation*}
\tag{1.4.2}
(\pi_{\Bm})_!\a^*\SE[d_{\Bm}] \simeq \bigoplus_{\r \in W\wg_{\Bm,\SE}}
             \r \otimes \IC(\SX_{\Bm}, \SL_{\r})[d_{\Bm}]. 
\end{equation*}
\end{enumerate}
\end{thm}

\para{1.5.}
Let $P$ be the stabilizer of the partial flag $(M_{p_i})$ in $G$, which
is a parabolic subgroup of $G$ containing $B$. 
Let $L$ be the Levi subgroup of $P$ containing $T$, and $U_P$ the unipotent 
radical of $P$.  
We consider the varieties
\begin{align*}
\SX_{\Bm}^P &= \bigcup_{g \in P}g(B \times \prod_{i=1}^{r-1} 
           M_{p_i}) = P \times \prod_{i=1}^{r-1} M_{p_i}, \\ 
\wh \SX^P_{\Bm} &= G \times^P\SX_{\Bm}^P = G \times^P(P \times \prod_{i=1}^{r-1}M_{p_i}), \\ 
\wt \SX^P_{\Bm} &= P \times^{B}(B \times \prod_{i=1}^{r-1}M_{p_i}).  
\end{align*}
We define $\pi': \wt\SX_{\Bm} \to  \wh\SX^P_{\Bm}$ as the map induced from 
the inclusion map $G \times (B \times \prod M_{p_i}) \to G \times (P \times \prod M_{p_i})$
under the identification $\wt\SX_{\Bm} \simeq G \times^B(B \times \prod M_{p_i})$, 
and define $\pi'': \wh\SX_{\Bm}^P \to \SX_{\Bm}$ by 
$g*(x,\Bv) \mapsto (gxg\iv, g\Bv)$.  (Here we denote by $g*(x,\Bv)$ the image of 
$(g, (x,\Bv)) \in G \times \SX^P_{\Bm}$ on $\wh\SX^P_{\Bm}$.)  
Thus we have $\pi_{\Bm} = \pi''\circ \pi'$.
Since $\pi_{\Bm}$ is proper, $\pi'$ is proper.  $\pi''$ is also proper. 
\par
Let $B_L = B \cap L$ be the Borel subgroup of $L$ containing $T$, and put
$\ol M_{p_i} = M_{p_i}/M_{p_{i-1}}$ under the convention $M_{p_0} = 0$. Then $L$ acts 
naturally on $\ol M_{p_i}$, and by  
applying the definition of 
$\pi_{\Bm} : \wt\SX_{\Bm} \to \SX_{\Bm}$ to $L$, we can define 
\begin{align*}
\wt\SX^L_{\Bm} &= L \times^{B_L}(B_L \times \prod_{i=1}^{r-1}\ol M_{p_i}), \\
\SX^L_{\Bm} &= \bigcup_{g \in L}g(B_L \times \prod_{i = 1}^{r-1}\ol M_{p_i}) 
= L \times \prod_{i=1}^{r-1} \ol M_{p_i}
\end{align*}
and the map $\pi^L_{\Bm} : \wt\SX^L_{\Bm} \to \SX^L_{\Bm}$ similarly. 
We have the following commutative diagram

\begin{equation*}
\tag{1.5.1}
\begin{CD}
\wt\SX_{\Bm} @<\wt p<<  G \times \wt\SX^P_{\Bm} @>\wt q>>  \wt\SX^L_{\Bm}  \\
     @V\pi'VV                 @VV r V                          @VV\pi^L_{\Bm} V    \\
\wh \SX^P_{\Bm}  @<p <<  G \times \SX_{\Bm}^P  @>q>>  \SX^L_{\Bm}  \\
      @V\pi''VV                                                       \\
   \SX_{\Bm}, 
\end{CD}
\end{equation*}
where 
the map $q$ is defined by $(g, x, \Bv) \mapsto (\ol x, \ol \Bv)$, 
with $x \mapsto \ol x, \Bv \mapsto \ol \Bv$ natural maps 
$P \to L, \prod M_{p_i} \to \prod \ol M_{p_i}$.
$\wt q$ is defined  as the composite of the projection 
$G \times \wt\SX_{\Bm}^P \to \wt\SX^P_{\Bm}$ 
and the map $\wt\SX_{\Bm}^P \to \wt\SX^L_{\Bm}$ 
induced from the projection 
$P \times (B \times \prod M_{p_i}) \to L \times (B_L \times \ol M_{p_i})$.
The maps $p, \wt p$ are the quotients by $P$, under the identification 
$\wt\SX_{\Bm} \simeq G \times^P\wt\SX^P_{\Bm}$. 
$r = \id \times r'$, where $r'$  is the natural map 
$\wt\SX^P_{\Bm} \to \SX^P_{\Bm}$, $g*(x, \Bv) \mapsto (gxg\iv, g\Bv)$.
\par
Here both squares in the diagram are cartesian.  Moreover, we have
\par\medskip\noindent
(i)  \ $p$ is a principal $P$-bundle.
\\
(ii) \ $q$ is a locally trivial fibration with fibre isomorphic to 
$G \times U_P \times \prod_{i=1}^{r-2}M_{p_i}$.
\par\medskip
Put $a = \dim P$, $b = \dim G + \dim U_P + \dim \prod_{i=1}^{r-2}M_{p_i}$. 
By (i) and (ii), the following property holds.
\par\medskip\noindent
(1.5.2) \ For any $L$-equivariant simple perverse sheaf $A_1$ on $\SX^L_{\Bm}$, 
$q^*A_1[b]$ is a $G \times P$-equivariant simple perverse sheaf on $G \times \SX^P_{\Bm}$, 
and there exists a unique $G$-equivriant simple perverse sheaf $A_2$ on $\wh\SX^P_{\Bm}$ 
(up to isomorphism) 
such that
\begin{equation*}
p^*A_2[a] \simeq q^*A_1[b].
\end{equation*}
\par
We define a perverse sheaf $K_{\Bm, T,\SE}$ on $\SX_{\Bm}$ by the right hand side
of the formula (1.4.2).  Thus 
$K_{\Bm, T,\SE} \simeq (\pi_{\Bm})_!\a^*\SE[d_{\Bm}]$.   
We consider the perverse sheaf $K^L_{\Bm, T, \SE}$ on $\SX^L_{\Bm}$ defined 
similarly to $K_{\Bm, T,\SE}$. 
Put $\SL = \a^*\SE$.  Since $\dim \wh\SX_{\Bm}^P = d_{\Bm}$, 
we see by (1.5.2) that 
$\pi'_!\SL[d_{\Bm}]$ is a perverse sheaf on $\wh\SX^P_{\Bm}$ satisfying the property
\begin{equation*}
\tag{1.5.3}
p^*\pi'_!\SL[d_{\Bm} + a]  \simeq q^*K^L_{\Bm, T,\SE}[b].  
\end{equation*}
Since $K^L_{\Bm, T,\SE}$ is decomposed as 

\begin{equation*}
\tag{1.5.4}
K^L_{\Bm,T,\SE} = \bigoplus_{\r \in W_{\Bm,\SE}\wg}
           \r\otimes\IC(\SX_{\Bm}^L, \SL^L_{\r})[d^L_{\Bm}], 
\end{equation*}
where $\SL^L_{\r}, d^L_{\Bm}$ are defined similarly to the theorem,  
again by (1.5.2), $\pi'_!\SL[d_{\Bm}]$ is a semisimple perverse sheaf, equipped with 
$W_{\Bm,\SE}$-action, and is decomposed as 

\begin{equation*}
\tag{1.5.5}
\pi'_!\SL[d_{\Bm}] \simeq \bigoplus_{\r \in W_{\Bm,\SE}\wg}\r \otimes A_{\r},
\end{equation*}
where $A_{\r}$ is a simple perverse sheaf on $\wh\SX^P_{\Bm}$ such that 
$p^*A_{\r}[a] \simeq q^*\IC(\SX^L_{\Bm}, \SL^L_{\r})[d_{\Bm}^L + b]$.  
By applying $\pi''_!$ on both sides, we have a decomposition

\begin{equation*}
\tag{1.5.6}
K_{\Bm, T,\SE} \simeq \bigoplus_{\r \in W_{\Bm,\SE}\wg}\r \otimes \pi''_!A_{\r}.
\end{equation*}
Comparing this with the decomposition in (1.4.2), we have
\begin{prop}  %%%%   Prop. 1.6.
$\pi''_!A_{\r} \simeq \IC(\SX_{\Bm}, \SL_{\r})[d_{\Bm}]$.
\end{prop}

\begin{proof}
By replacing $\wt\SX_{\Bm}$, etc. by $\wt\SY^0_{\Bm}$, etc., we have a similar diagram 
as (1.5.1), 

\begin{equation*}
\tag{1.6.1}
\begin{CD}
\wt\SY^0_{\Bm} @<\wt p_0<<  G \times \wt\SY^{P,0}_{\Bm} @>\wt q_0>>  \wt\SY^{L,0}_{\Bm}  \\
     @V\psi'VV                 @VV r_0 V                  @VV\psi^{L,0}_{\Bm} V    \\
\wh \SY^{P,0}_{\Bm}  @<p_0 <<  G \times \SY_{\Bm}^{P,0}  @>q_0>>  \SY^{L,0}_{\Bm}  \\
      @V\psi''VV                                                       \\
   \SY^0_{\Bm}, 
\end{CD}
\end{equation*}
where 
\begin{align*}
\SY_{\Bm}^{P,0} &= \bigcup_{g \in P}g(T\reg \times \prod_i M^0_{p_i}), \\ 
\wh \SY^{P,0}_{\Bm} &= G \times^P\SY_{\Bm}^{P,0}, \\
\wt \SY^{P,0}_{\Bm} &= P \times^{T}(T\reg \times \prod_iM^0_{p_i}).  
\end{align*}
and $G \times \SY^{P,0}_{\Bm} = q\iv(\SY^{L,0}_{\Bm}) = p\iv(\wt\SY^{P,0}_{\Bm})$, 
$p_0, q_0$ are restrictions of $p,q$, respectively. 
Similar properties as in (1.5.1) hold also for (1.6.1).
Note that $\psi^{L,0}_{\Bm}$ is a finite Galois covering with group 
$W_{\Bm,\SE}$, and so $\psi'$ is also a finite Galois covering with $W_{\Bm,\SE}$.  
Since $\psi^0_{\Bm} = \psi''\circ\psi'$ and $\psi^0_{\Bm}$ is a finite Galois covering 
with group $W_{\Bm,\SE}$, we see that $\psi''$ is an isomorphism. 
Now $\wh\SY^{0,P}_{\Bm}$ is open dense in $\wh\SX^P_{\Bm}$, and the $W_{\Bm,\SE}$-module
structure on $\pi'_!\SL$ is determined from the corresponding structure on 
$\psi'_!\SL$ obtained from the Galois covering $\psi'$.  
On the other hand, the $W_{\Bm,\SE}$-module structure on $(\pi_{\Bm})_!\SL$ is determined from 
the corresponding structure on $(\psi^0_{\Bm})_!\SL$ obtained from the Galois covering 
$\psi^0_{\Bm}$.  Since $\psi''$ is an isomorphism, this shows that the operation $\pi''_!$ 
is compatible with the $W_{\Bm, \SE}$-module structures of $\pi'_!\SL$ and of $(\pi_{\Bm})_!\SL$.  
The proposition is proved.
\end{proof}

\par\bigskip\bigskip
\section{Green functions}

\para{2.1.}
We now assume that $G$ and $V$ are defined over $\BF_q$, and let 
$F: G \to G, F: V \to V$ be the corresponding Frobenius maps. 
We fix an $F$-stable Borel subgroup $B_0$ and an $F$-stable maximal 
torus $T_0$ contained in $B_0$.  We define $W_0$ as $W_0 = N_G(T_0)/T_0$.
Let $(M_{0,i})$ be the total flag corresponding to $B_0$.  Thus $M_{0,i}$ are
$F$-stable subspaces. For $\Bm \in \SQ_{n,r}$, let 
$P_{\Bm}$ be the parabolic subgroup of 
$G$ containing $B_0$   
which is the stabilizer of the partial flag $(M_{0,p_i})$.
Then the Weyl subgroup of $W_0$ corresponding to $P_{\Bm}$ is given 
by $(W_0)_{\Bm}$.    
\par
Let $T$ be an $F$-stable maximal torus of $G$, and $B \supset T$ 
a not necessarily 
$F$-stable Borel subgroup of $G$.  Let $(M_i)$ be the total flag of $G$ whose 
stabilizer is $B$.  We assume that $M_{p_i}$ is $F$-stable for each $i$. 
Let us construct $\wt\SX_{\Bm}, \wt\SY_{\Bm}$, etc. as in 
Section 1 by using these $T$ and $B$. 
There exists $h \in G$ such that $B = hB_0h\iv, T = hT_0h\iv$, and that
$h\iv F(h) = \dw$,
where $\dw$ is a representative of $w \in (W_0)_{\Bm}$ in $N_G(T_0)$.
We fix an $F$-stable basis $e_1, \dots, e_n$ of $V$ which are weight vectors for 
$T_0$.  Then $he_1, \dots, he_{p_i}$ are basis of $M_{p_i}$ consisting of 
weight vectors for $T$.  
If we define $M_{p_i}^0$ as in 1.3 by using this basis,   
then $M_{p_i}^0$ is $F$-stable for each $i$. 
Since $\wt\SY^0_{\Bm} \simeq G \times^T(T\reg \times \prod_iM^0_{p_i})$, 
$\wt\SY^0_{\Bm}$ has a natural $\Fq$-structure.  
$\SY^0_{\Bm}$ is $F$-stable, 
and the maps $\psi^0_{\Bm} : \wt\SY^0_{\Bm} \to \SY^0_{\Bm}$ and 
$\a_0: \wt\SY^0_{\Bm} \to T$ are $F$-equivariant. 
Let 
$\SE$ be a tame local system on $T$ such that $F^*\SE \simeq \SE$.  
We fix an isomorphism 
$\vf_0 : F^*\SE \isom \SE$.  Then $\vf_0$ induces an isomorphism 
$\wt\vf_0 : F^*\SL^{\bullet} \isom \SL^{\bullet}$, where 
$\SL^{\bullet}$ is the local system  $(\psi^0_{\Bm})_!\a_0^*\SE$ on $\SY^0_{\Bm}$.
By (1.3.1), we have 
  $\SL^{\bullet} \simeq \bigoplus_{\r \in W_{\Bm,\SE}\wg}\r \otimes \SL_{\r}$.
As in 1.5, we define a complex $K_{\Bm, T,\SE}$ on 
$\SX_{\Bm}$ by 

\begin{equation*}
\tag{2.1.2}
K_{\Bm, T,\SE} = \IC(\SX_{\Bm}, \SL^{\bullet})[d_{\Bm}] \simeq \bigoplus_{\r \in W_{\Bm,\SE}\wg}
                     \r \otimes \IC(\SX_{\Bm}, \SL_{\r})[d_{\Bm}].
\end{equation*}
$\wt\vf_0$ can be extended to a unique isomorphism 
$\vf: F^*K_{\Bm, T,\SE} \isom K_{\Bm, T,\SE}$. 
Note that by Theorem 1.4 (ii), $(\pi_{\Bm})_!\a^*\SE[d_{\Bm}]$ is isomorphic to
$K_{\Bm, T,\SE}$.  But the $\Fq$-structure of $(\pi_{\Bm})_!\a^*\SE[d_{\Bm}]$ is not defined directly 
from the construction. 

\para{2.2.}
Let $\SD X = \SD^b_c(X)$ be the bounded derived category of $\Ql$-constructible sheaves on 
a variety $X$ over $\Bk$.  
Assume that $X$ is defined over $\Fq$, and let $F:X \to X$ be the corresponding 
Frobenius map.  Recall that for a given $K \in \SD X$ with an isomorphism 
$\f : F^*K \isom K$, the characteristic function $\x_{K,\f} : X^F \to \Ql$ is defined by

\begin{equation*}
\x_{K,\f}(x) = \sum_i(-1)^i\Tr(\f^*, \SH^i_xK), \quad (x \in X^F),
\end{equation*}    
where $\f^*$ is the induced isomorphism on $\SH^i_xK$. 
\par
Returning to the original setting, we consider a tame local system $\SE$ on $T$ 
such that $F^*\SE \simeq \SE$. Since the isomorphism $F^*\SE \isom \SE$ is unique 
up to scalar, we fix $\vf_0: F^*\SE \isom \SE$ so that it induces 
the identity map on the stalk $\SE_e$ at the identity element $e \in T$.   
We consider the characteristic function of $K_{\Bm, T,\SE}$ with respect to the map
$\vf$ induced from this $\vf_0$, and denote it by $\x_{\Bm, T,\SE}$.  
Since $K_{\Bm,T,\SE}$ is a $G$-equivariant perverse sheaf, $\x_{\Bm, T,\SE}$ is 
a $G^F$-invariant function on $\SX_{\Bm}^F$. 
\par
The following result is an analogue of Lustig's result ([L2, (8.3.2)]) for character 
sheaves. The gap of the proof in [L2] was corrected in [L4] in a more general setting of 
character sheaves on disconnected reductive groups. 
The analogous statement in the case of exotic symmetric 
space (of level 2) was proved in [SS2, Prop. 1.6] based on the argument in [L4].  
The proof for the present case is quite similar to that of [SS2], so we omit the proof here.

\begin{prop}  %%%%  Prop. 2.3.
The restriction of $\x_{\Bm, T,\SE}$ on $\SX_{\Bm,\unip}$ is independent of the
choice of $\SE$.  
\end{prop}

We define a function $Q_{\Bm, T} = Q^G_{\Bm, T}$ as the restriction of $\x_{\Bm, T,\SE}$ on 
$\SX^F_{\Bm,\unip}$, and call it the Green function on $\SX_{\Bm,\unip}$. 

\para{2.4.}
Let $T = T_w$ be an $F$-stable maximal torus in $G$ as in 2.1, namely 
$T = hT_0h\iv$ with $ h\in G$ such that $h\iv F(h) = \dw$ for $w \in (W_0)_{\Bm}$.
We consider the isomorphism $\vf = \vf_T : F^*K_{\Bm,T,\SE} \isom K_{\Bm,T,\SE}$ as in 2.2, 
defined from the specific choice of $\vf_0$. 
Let $\SE_0$ be the tame local system on $T_0$ defined by $\SE_0 = (\ad h)^*\SE$.
Then we have an isomorphism $\vf_{T_0}: F^*K_{\Bm, T_0, \SE_0} \isom K_{\Bm, T_0, \SE_0}$. 
For later use, we shall describe the relationship between $\vf_T$ and $\vf_{T_0}$. 
We write the varieties and maps $\wt\SY_{\Bm}, \wt\SY^0_{\Bm}, \a_0$, etc. as 
$\wt\SY_{\Bm, T}, \wt\SY^0_{\Bm, T}, \a_{0,T}$, etc. to indicate the dependence on $T$. 
$\wt\SY^0_{\Bm, T_0}$ has a natural Frobenius action 
$F: (x, \Bv, gT_0) \mapsto (F(x), F(\Bv), F(g)T_0)$, and similarly for $\wt\SY^0_{\Bm, T}$.  
The map 
$(x, \Bv, gT) \mapsto (x, \Bv, ghT_0)$ gives a morphism 
$\d : \wt\SY^0_{\Bm, T} \to \wt\SY^0_{\Bm, T_0}$ commuting with the projection to 
$\SY^0_{\Bm}$ (note that $\SY^0_{\Bm}$ is independent of the choice of $T$).
We define a map $a_w : \wt\SY^0_{\Bm, T_0} \to \wt\SY^0_{\Bm, T_0}$ by 
$(x, \Bv, gT_0) \mapsto (x, \Bv, g\dw\iv T_0)$.  Then we have a commutative diagram

\begin{equation*}
\tag{2.4.1}
\begin{CD}
\wt\SY^0_{\Bm, T} @>\d >>  \wt\SY^0_{\Bm, T_0}  \\
     @VF VV                   @VVa_{w}F V             \\
\wt\SY^0_{\Bm, T}  @>\d >>  \wt\SY^0_{\Bm, T_0}
\end{CD}
\end{equation*}
Let $\SL^{\bullet}_0 = (\psi^0_{\Bm, T_0})_!(\a_{0,T_0})^*\SE_0$ be the local system on $\SY^0_{\Bm}$. 
We know $\End \SL^{\bullet}_0  \simeq \Ql[(W_0)_{\Bm, \SE_0}]$.  This isomorphism 
is given as follows; 
the map $w \mapsto a_w$ gives a homomorphism $(W_0)_{\Bm} \to \Aut(\wt\SY^0_{\Bm, T_0})$.    
If $w \in (W_0)_{\Bm,\SE}$, $a_w$ induces an isomorphism $\wt a_w$ 
on $\SL^{\bullet}_0$, and the map $w \mapsto \wt a_w$ gives 
the isomorphism  $\Ql[(W_0)_{\Bm, \SE_0}] \to \End \SL^{\bullet}_0$. 
By the property of the intermediate extensions, we have an isomorphism on 
$K_{\Bm, T_0, \SE_0}$ induced from $\wt a_w$, which we denote by $\th_w$. 
Also $\d$ induces an isomorphism  
$\SL^{\bullet} = (\psi^0_{\Bm,T})_!(\a_{0,T})^*\SE \to \SL^{\bullet}_0$, and 
so induces an isomorphism $K_{\Bm, T, \SE} \to K_{\Bm, T_0, \SE_0}$, which 
we denote by $\wt\d$.  
Then the diagram (2.4.1) implies the following commutative diagram.

\begin{equation*}
\tag{2.4.2}
\begin{CD}
F^*K_{\Bm, T,\CE} @>F^*(\wt\d)>>  F^*K_{\Bm, T_0, \CE_0} \\
   @V\vf_T VV                      @VV \vf_{T_0}\circ F^*(\th_w)V     \\
   K_{\Bm, T, \SE} @>\wt\d>>  K_{\Bm, T_0, \SE_0}.
\end{CD}
\end{equation*}
Note that $F^*(\th_w) = \th_{F(w)}$. Since $F$ acts trivially on $W_0$, we have
$F^*(\th_w) = \th_w$. 

\para{2.5.}
For a semisimple element $s \in G^F$, we consider $Z_G(s) \times V^{r-1}$.
Then $V$ is decomposed as $V = V_1 \oplus \cdots \oplus V_t$, where $V_i$ 
is an eigenspace of $s$ with $\dim V_i = n_i$, and $F$ permutes the eigenspaces.
$Z_G(s) \simeq G_1 \times \cdots \times G_t$ with $G_i = GL(V_i)$. 
Hence

\begin{equation*}
\tag{2.5.1}
Z_G(s) \times V^{r-1} \simeq \prod_{i=1}^t (G_i \times V_i^{r-1}), 
\end{equation*} 
and the diagonal action of $Z_G(s)$ on the left hand side is compatible with 
the diagonal action of $G_i$ on $G_i\times V_i^{r-1}$ under the isomorphism 
$Z_G(s) \simeq G_1 \times \cdots \times G_t$. 
The definition of $\wt\SX_{\Bm}, \SX_{\Bm},\wt\SY_{\Bm}, \SY_{\Bm}$, etc. make sense if we replace 
$G$ by $Z_G(s)$, hence one can define the complex $K_{\Bm, T,\SE}$ associated to
$Z_G(s)$.  By (2.5.1), $\wt\SX_{\Bm}, \SX_{\Bm}$ with respect to $Z_G(s)$ are a direct product 
of similar varieties appeared in 1.1 with respect to $G_i$, hence Theorem 1.4 holds also for 
$Z_G(s)$ under a suitable adjustment. 
Concerning the $\Fq$-structure, Proposition 2.3 holds also for $Z_G(s)$.  In fact, this is 
easily reduced to the case where $F$ acts transitively on 
$Z_G(s) \simeq G_1 \times \cdots \times G_t$. In that case, $G^F \simeq G_1^{F^t}$, and 
we may assume that $s \in T_0^{F^t}$.  
Thus Proposition 2.3 holds in this case, hence holds for the general case. In particular, 
one can define a Green function $Q^{Z_G(s)}_{\Bm, T}$ on $\SX_{\Bm, \unip}^{Z_G(s)}$ 
for $T \subset Z_G(s)$. 
\par
For an $F$-stable torus $S$, put $(S^F)\wg = \Hom (S^F, \Ql^*)$.
As in [SS2, 2.3], the set of $F$-stable tame local systems on $S$ is 
in bijection with the set $(S^F)\wg$ in such a way that 
the characteristic function $\x_{\SE, \vf_0}$ on $S^F$ gives an element of 
$(S^F)\wg$ (under the specific choice of $\vf_0: F^*\SE \isom \SE$ as in 2.2).  
We denote by $\SE_{\th}$ the $F$-stable tame local system on $S$ corresponding 
to $\th \in (S^F)\wg$. 
\par
The following theorem is an analogue of Lusztig's character formula for 
character sheaves [L2, Theorem 8.5].  A similar formula for the case of exotic 
symmetric space (of level 2) was proved in [SS2, Theorem 2.4].  The proof of the 
theorem is quite similar to the proof given in [SS2, Theorem 2.4], so we omit 
the proof here. 

\begin{thm}[Character formula]  %%%%  Theorem 2.6.
Let $s, u \in G^F$ be such that $su = us$, where $s$ is semisimple and 
$u$ is unipotent.  Then 

\begin{equation*}
\x_{\Bm, T, \SE}(su, \Bv) = |Z_G(s)^F|\iv \sum_{\substack{x \in G^F \\
                         x\iv sx \in T^F}}Q^{Z_G(s)}_{\Bm, xTx\iv}(u,\Bv)\th(x\iv sx),
\end{equation*} 
where $\th \in (T^F)\wg$ is such that $\SE = \SE_{\th}$. 
\end{thm}

\para{2.7.}
For later use, we shall introduce another type of Green functions.
We follow the notation in 1.5.
By using the basis $\{ e_1, \dots, e_n\}$ of $V$, we identify $\ol M_{p_i}$ with the subspace 
$M_{p_i}^+$ of $V$ so that 
\begin{equation*}
V = M^+_{p_1} \oplus \cdots \oplus M^+_{p_r}.
\end{equation*}
Hence $B_L$ stabilizes each $M_{p_i}^+$. 
As an analogue of the construction of $\wt\SX_{\Bm}$, we consider the diagram

\begin{equation*}
\begin{CD}
T @<\a^+ <<  \wt\SX^+_{\Bm} @>\pi^+_{\Bm}>>  \SX_{\Bm},   
\end{CD}
\end{equation*}
where 

\begin{align*}
\wt\SX^+_{\Bm} &= \{ (x, \Bv, gB_L) \in G \times V^{r-1} \times G/B_L
            \mid g\iv xg \in B, g\iv \Bv \in \prod_{i=1}^{r-1} M^+_{p_i} \}  \\
               &\simeq G \times^{B_L}(B \times \prod_i M^+_{p_i}),
\end{align*}
and $\a^+ : (x, \Bv, gB_L) \mapsto p_T(g\iv xg), 
      \pi^+_{\Bm} : (x, \Bv, gB_L) \mapsto (x, \Bv)$. 
Put 
\begin{equation*}
\tag{2.7.1}
\SX^+_{\Bm} = \bigcup_{g \in G}g(B \times \prod_iM^+_{p_i}).
\end{equation*}
Then $\wt\SX^+_{\Bm}$ is smooth, irreducible, and 
$\Im \pi^+_{\Bm} = \SX^+_{\Bm}$, 
but $\pi^+_{\Bm}$ is not proper, and $\SX^+_{\Bm}$ is not necessarily 
a locally closed subset of $\SX_{\Bm}$. 
We consider a variety
\begin{equation*}
\tag{2.7.2}
\wh\SX^L_{\Bm} = G \times^L(P \times \prod_iM^+_{p_i})
\end{equation*}
and put $d^+_{\Bm} = \dim \wh\SX^L_{\Bm}$.  
For a tame local system $\SE$ on $T$, we consider the complex 
$(\pi^+_{\Bm})_!(\a^+)^*\SE[d^+_{\Bm}]$ on $\SX_{\Bm}$, 
which we denote by $K^+_{\Bm, T,\SE}$.

\para{2.8.}
Let $\wh\SX^L_{\Bm}$ be as in 2.7.  We further define 

\begin{align*}
\SX_{\Bm}^{P,+} &= \bigcup_{g \in L}g(B \times \prod_i M^+_{p_i}) = P \times \prod_i M^+_{p_i}, \\ 
\wt \SX^{L,+}_{\Bm} &= L \times^{B_L}(B \times \prod_i M^+_{p_i}). 
\end{align*}
We define maps $\pi'_+: \wt\SX^+_{\Bm} \to  \wh\SX^L_{\Bm}$, 
$\pi''_+: \wh\SX_{\Bm}^L \to \SX_{\Bm}$ in a similar way as in 1.5. 
Thus we have $\pi^+_{\Bm} = \pi''_+\circ \pi'_+$. 
We consider a commutative diagram

\begin{equation*}
\tag{2.8.1}
\begin{CD}
\wt\SX^+_{\Bm} @<\wt p_+<<  G \times \wt\SX^{L,+}_{\Bm} @>\wt q_+>>  \wt\SX^L_{\Bm}  \\
     @V\pi'_+VV                 @VV r_+ V                                           @VV\pi^L_{\Bm} V    \\
\wh \SX^L_{\Bm}  @<p_+ <<  G \times \SX_{\Bm}^{P,+}  @>q_+>>  \SX^L_{\Bm}  \\
      @V\pi''_+VV                                                       \\
   \SX_{\Bm}, 
\end{CD}
\end{equation*}
where $q_+$ is the composite of the projection $G \times \SX^{P,+}_{\Bm} \to \SX^{P,+}_{\Bm}$ and the 
map $\SX^{P,+}_{\Bm} \to\SX^L_{\Bm},  (x, \Bv) \mapsto (\ol x, \Bv)$ under the identification 
$M^+_{p_i} = \ol M_{p_i}$ ($\ol x$ is as in 1.5).   
The map $\wt q_+$ is the composite of the projection 
$G \times \wt\SX_{\Bm}^{L,+} \to \wt\SX_{\Bm}^{L,+}$
and the natural map $\wt\SX^{L,+}_{\Bm} \to \wt\SX^L_{\Bm}$ induced from the map
$L \times (B \times \prod_iM^+_{p_i}) \mapsto L \times (B_L \times \prod_i\ol M_{p_i})$. 
$r_+ = \id \times r_+'$, where $r_+'$ is the natural map $\wt\SX^{L,+}_{\Bm} \to \SX^{P,+}_{\Bm}$.
Both squares are cartesian, and we have  
\par\medskip\noindent
(i) \ $p_+$ is a principal $L$-bundle.
\par\noindent
(ii) \ $q_+$ is a locally trivial fibration with fibre isomorphic to 
$G \times U_P$.  
\par\medskip\noindent
Then (i) and (ii) imply a similar statement as (1.5.2).
We consider the perverse sheaf $K^L_{\Bm, T, \SE}$ on $\SX^L_{\Bm}$ as in 1.5.
Put $\SL^+ = (\a^+)^*\SE$. 
By a similar discussion as in 1.5, 
$(\pi'_+)_!\SL^+[d^+_{\Bm}]$ is a perverse sheaf on $\wh\SX^L_{\Bm}$ satisfying the property
\begin{equation*}
\tag{2.8.2}
p_+^*(\pi'_+)!\SL^+[d^+_{\Bm} + a']  \simeq q_+^*K^L_{\Bm, T,\SE}[b'],  
\end{equation*}
where $a' = \dim L$ and $b' = \dim G + \dim U_P$.
Since $K^L_{\Bm, T,\SE}$ is decomposed as in (1.5.4), it follows that 
$(\pi'_+)_!\SL^+[d^+_{\Bm}]$ is a semisimple perverse sheaf on $\wh\SX^L_{\Bm}$, 
equipped with 
$W_{\Bm,\SE}$-action, and is decomposed as 
\begin{equation*}
\tag{2.8.3}
(\pi'_+)_!\SL^+[d^+_{\Bm}] \simeq \bigoplus_{\r \in W_{\Bm,\SE}\wg}\r \otimes B_{\r},
\end{equation*}
where $B_{\r}$ is a simple perverse sheaf on $\wh\SX^L_{\Bm}$.
\par
By applying $(\pi''_+)_!$ on both sides of (2.8.3), we have a similar formula 
as (1.5.6),

\begin{equation*}
\tag{2.8.4}
K^+_{\Bm, T,\SE} \simeq \bigoplus_{\r \in W_{\Bm,\SE}\wg}
              \r \otimes (\pi''_+)_!B_{\r}.
\end{equation*}
But note that $(\pi''_+)_!B_{\r}$ is not necessarily a perverse 
sheaf.

\para{2.9.}
We now consider the $\Fq$-structure, and assume that $T, B$ are both $F$-stable, 
hence $P$ and $L$ are also $F$-stable.   
Then all the varieties involved in the diagram (2.8.1) have natural $\Fq$-structures, and all the 
maps are $F$-equivariant.  This is also true for the diagram (1.5.1). 
Recall that $K^L_{\Bm, T,\SE}$ has the natural $\Fq$-structure 
$\vf^L: F^*K^L_{\Bm, T,\SE} \isom K^L_{\Bm, T,\SE}$. 
Then the decomposition (1.5.4) determines the isomorphism 
$\vf^L_{\r} : F^*\IC(\SX^L_{\Bm}, \SL^L_{\r}) \isom \IC(\SX^L_{\Bm}, \SL^L_{\r})$  for 
each $\r$ 
such that $\vf^L = \sum_{\r}\s_{\r} \otimes \vf^L_{\r}$, where $\s_{\r}$ is the identity 
map on the representation space $\r$.
By (2.8.1), the map $\vf^L_{\r}$ induces an isomorphism $h_{\r} :F^*B_{\r} \isom B_{\r}$.
We also obtain an isomorphism $f_{\r} : F^*A_{\r} \isom A_{\r}$ by a similar discussion 
applied to the diagram (1.5.1).
\par     
We have the following lemma.

\begin{lem}   %%%%  Lemma 2.10
For any $z \in (\SX_{\Bm}^+)^F$ and $\r \in W_{\Bm,\SE}\wg$, we have 
\begin{equation*}
\Tr\bigl( h^*_{\r}, \BH_c^i((\pi''_+)\iv(z), B_{\r})\bigr) = 
       \Tr\bigl(f^*_{\r}, \BH_c^{i - \dim U_P}((\pi'')\iv(z), A_{\r})\bigr)q^{\dim U_P},
\end{equation*}
where $f^*_{\r}, h^*_{\r}$ are the linear maps on the cohomologies induced from 
$f_{\r}, h_{\r}$. 
\end{lem}

\begin{proof}
We consider the following commutative diagram
\begin{equation*}
\tag{2.10.1}
\begin{CD}
\wh\SX^L_{\Bm} @<p_+<< G \times (P \times \prod M_{p_i}^+) @>q_+>>  \SX^L_{\Bm} \\
   @V\e VV                 @V\wt\e VV                 @VV\id V     \\
G \times^L(P \times \prod M_{p_i}) @<p_{\bullet}<<  G \times (P \times \prod M_{p_i}) 
              @>q>>  \SX^L_{\Bm}    \\
    @V\xi VV                @VV\id V          \\
\wh\SX^P_{\Bm}  @<p<<  G \times (P \times \prod M_{p_i}),
\end{CD}
\end{equation*} 
where $\wt\e$ is the inclusion map, and $\e$ is the map induced from $\wt\e$ by 
taking the quotients by $L$. 
$\xi$ is the natural map from the quotient by $L$ to the quotient by $P$, 
and $p_{\bullet}$ is the quotient by $L$. 
Note that $\e$ is injective, and $\xi$ is a locally trivial fibration with fibre 
isomorphic to $U_P$. 
\par
By a similar construction of $A_{\r}$ and $B_{\r}$, one can define 
a simple perverse sheaf $\wt B_{\r}$ on $G \times^L(P \times \prod M_{p_i})$. 
By (2.10.1), we have

\begin{equation*}
\tag{2.10.2}
\xi^*A_{\r}[d] \simeq \wt B_{\r}, \qquad 
\e^*\wt B_{\r} \simeq B_{\r},
\end{equation*}
where $d = \dim U_P$. 
(The latter formula follows from the fact that the upper left square 
in (2.10.1) is cartesian.)
We now define $\pi_{\bullet}'' : G \times^L(P \times \prod M_{p_i}) \to \SX_{\Bm}$ by 
$(g*(x,\Bv)) \mapsto (gxg\iv, g\Bv)$. 
Then $(\pi_{\bullet}'')\iv(\SX_{\Bm}^+) = \wh\SX^L_{\Bm}$ and the restriction of 
$\pi''_{\bullet}$ on $\wh\SX^L_{\Bm}$ coincides with $\pi''_+$.
It follows, for $z \in (\SX^+_{\Bm})^F$, that 

\begin{align*}
\tag{2.10.3}
\BH^i_c((\pi''_+)\iv(z), B_{\r}) &\simeq \BH^i_c((\pi''_{\bullet})\iv(z), \wt B_{\r})  \\
                                &= \BH^i_c((\pi''_{\bullet})\iv(z), \xi^* A_{\r}[d]) \\
                                &\simeq \BH^i_c((\pi'')\iv(z), \xi_!\xi^*A_{\r}[d]) \\
                                &\simeq \BH_c^{i - d}((\pi'')\iv(z), A_{\r})(d)
\end{align*}
since $(\pi_{\bullet}'')\iv(z) \to (\pi'')\iv(z)$ is a locally trivial fibration with fibre 
isomorphic to $U_P$, and so $\xi_!\xi^*A_{\r} \simeq A_{\r}[-2d](d)$,
where $(d)$ denotes the Tate twist.  
Note that the Frobenius actions $h^*_{\r}, f^*_{\r}$ on these cohomologies come from the Frobenius
action $\vf^L_{\r}$ on $\IC(\SX^L_{\Bm}, \SL^L_{\r})$.  
Hence the isomorphism in (2.10.3) is compatible with the maps $h^*_{\r}, f^*_{\r}$.
This proves the lemma.  
\end{proof}

\para{2.11.}
We now consider the general case.  
We assume that $T, L$ and $P$ are $F$-stable, but $B$ is not necessarily 
$F$-stable. Since $P$ is $F$-stable, in particular all the subspaces 
$M_{p_i}$ are $F$-stable. 
We consider the diagram (2.8.1).  Except the top row, 
all the objects involved in (2.8.1) are $F$-equivariant.  
We consider the isomorphism $\vf^L : F^*K^L_{\Bm, T,\SE} \isom K^L_{\Bm, T,\SE}$.
Since $p_+, q_+$ are 
$F$-equivariant, $\vf^L$ induces an isomorphism 
$\vf': F^*((\pi'_+)_!\SL^+) \isom (\pi'_+)_!\SL^+$.  
Since $\pi''_+$ is $F$-equivariant, $\vf'$ induces an isomorphism 
$\vf^+ : F^*K^+_{\Bm, T, \SE} \isom K^+_{\Bm, T,\SE}$. 
We define a function $\x^+_{\Bm, T,\SE}$ by 
\begin{equation*}
\tag{2.11.1}
\x^+_{\Bm, T,\SE} = (-q)^{- \dim U_P}\x_{K^+_{\Bm, T,\SE}, \vf^+}.
\end{equation*} 

\para{2.12.}
Let $T_0 \subset B_0$ be as in 2.1.  We also consider the $F$-stable parabolic subgroup
$P_0$ containing $B_0$ which is conjugate to $P$ under $G^F$.  Let $L_0$ be the Levi
subgroup of $P_0$ containing $T_0$.  Since $T$ is $G^F$-conjugate to an $F$-stable maximal 
torus in $L_0$, in considering the $\Fq$-structure of $K_{\Bm, T,\SE}$, 
we may assume that $T = T_w \subset L_0$ for $w \in (W_0)_{\Bm}$. We shall consider two complexes
$K_{\Bm, T,\SE}$ and $K_{\Bm, T_0, \SE_0}$ as discussed in 2.4. 
We follow the notation in 2.4.
In particular, let $\vf_T$ (resp. $\vf_{T_0}$) be the isomoprhism $\vf$ with respect to
$K_{\Bm, T,\SE}$ (resp. $K_{\Bm, T_0, \SE_0}$).  
By the decomposition of 
$K_{\Bm, T_0, \SE_0}$ in (1.4.2), $\vf_{T_0}$ determines an isomorphism 
$\vf_{\r} : F^*\IC(\SX_{\Bm}, \SL_{\r})[d_{\Bm}] \isom \IC(\SX_{\Bm}, \SL_{\r})[d_{\Bm}]$ such that
$\vf_{T_0} = \sum_{\r \in (W_0)_{\Bm,\SE_0}\wg}\s_{\r}\otimes \vf_{\r}$, where $\s_{\r}$ 
is the identity map on the representation space $\r$. 
By (2.4.2), under the isomorphism $\wt\d : K_{\Bm, T, \SE} \isom K_{\Bm, T_0, \SE_0}$, the map
$\vf_T$ can be described by the map $\vf_{T_0}$ and $\th_w$, where $\th_w$ corresponds to the 
action of $w$ on each $(W_0)_{\Bm, \SE_0}$-module $\r$.  It follows that 
\begin{equation*}
\tag{2.12.1}
\vf_{T} \simeq \sum_{\r \in (W_0)_{\Bm, \SE_0}\wg}w|_{\r}\otimes \vf_{\r},
\end{equation*} 
where $w|_{\r}$ denotes the action of $w$ on $\r$. 
\par
Next we consider two complexes $K^+_{\Bm, T, \SE}$ 
and $K^+_{\Bm, T_0, \SE_0}$.  Let $\vf^+_T$ (resp. $\vf^+_{T_0}$) be 
the isomorphism with respect to $K^+_{\Bm, T, \SE}$ (resp. $K^+_{\Bm, T_0, \SE_0}$).
The isomorphism $\vf^L_{T}$ and $\vf^L_{T_0}$ are defined similarly with respect to 
$K^L_{\Bm, T, \SE}$
and $K^L_{\Bm, T_0, \SE_0}$.  Then a similar formula as (2.12.1) holds for 
$\vf^L_T$ and $\vf^L_{T_0}$.   
We denote by $K^+_T$ (resp. $K^+_{T_0}$) the complex $(\pi'_+)_!\SL^+[d^+_{\Bm}]$ 
defined with respect to $T$ (resp. $T_0$).  By the diagram (2.8.1), the isomorphism 
$\vf^L_T$ induces an isomorphism $h_T: F^*K^+_T \isom K^+_T$, and 
$h_{T_0}: F^*K^+_{T_0} \isom K^+_{T_0}$.  Then again by (2.8.1), we have
\begin{equation*}
\tag{2.12.2}
h_T \simeq \sum_{\r \in (W_0)_{\Bm, \SE_0}\wg}w|_{\r} \otimes h_{\r},
\end{equation*}  
where $h_{\r}$ is the isomorphism in 2.9 defined with respect to $T_0$.  
By applying $(\pi''_+)_!$, we see that 

\begin{equation*}
\tag{2.12.3}
\vf^+_T \simeq \sum_{\r \in (W_0)_{\Bm,\SE_0}\wg}w|_{\r} \otimes \vf^+_{\r},
\end{equation*} 
where $\vf^+_{\r} : F^*(\pi''_+)_!B_{\r} 
        \isom (\pi''_+)_!B_{\r}$ is an isomorphism 
induced from $h_{\r}$.

\begin{prop}  %%%%% Prop. 2.13.
Under the setting in 2.11, we have
\begin{equation*}
\x^+_{\Bm, T,\SE}(z) = 
             \begin{cases}            
             \x_{\Bm, T,\SE}(z)   &\quad\text{ if }
            z \in (\SX^+_{\Bm})^F,    \\
              0  &\quad\text{otherwise.}
              \end{cases}
\end{equation*}
\end{prop}

\begin{proof}
Since $\Im \pi^+_{\Bm} = \SX^+_{\Bm}$, it is clear 
that $\x^+_{\Bm, T,\SE}(z) = 0$ unless $z \in (\SX^+_{\Bm})^F$. 
Take $z \in (\SX^+_{\Bm})^F$. By (2.12.1), we have 
\begin{equation*}
\x_{\Bm, T,\SE}(z) = \sum_{\r \in (W_0)_{\Bm, \SE_0}\wg}\Tr(w, \r)\x_{I_{\r}, \vf_{\r}},
\end{equation*}
where $I_{\r} = \IC(\SX_{\Bm}, \SL_{\r})[d_{\Bm}]$. Here 
$\SH^i_z(I_{\r}) \simeq \BH^i_c((\pi'')\iv(z), A_{\r})$ by Proposition 1.6,
and the isomorphism on $\SH^i_z(I_{\r})$ induced from $\vf_{\r}$ coincides with 
the isomorphism $f^*_{\r}$ on $\BH^i_c((\pi'')\iv(z), A_{\r})$. 
On the other hand, by (2.12.3) we have
\begin{equation*}
\x_{K^+_{\Bm, T,\SE}, \vf^+}(z) 
   = \sum_{\r \in (W_0)_{\Bm, \SE_0}\wg}\Tr(w, \r)\x_{J_{\r}, \vf^+_{\r}},
\end{equation*}
where $J_{\r} = (\pi''_+)_!B_{\r}$.  
Here $\SH^i_z(J_{\r}) \simeq \BH^i_c((\pi''_+)\iv(z), B_{\r})$, 
and the isomorphism 
on $\SH^i_z(J_{\r})$ induced from $\vf^+_{\r}$ coincides with $h^*_{\r}$ on  
$\BH^i_c((\pi''_+)\iv(z), B_{\r})$. 
The proposition now follows from 
Lemma 2.10 and (2.11.1). 
\end{proof}

\para{2.14.}
As in the case of Green functions $Q_{\Bm, T}$, 
we define Green functions $Q^+_{\Bm, T}$ as the 
restriction of $\x^+_{\Bm, T, \SE}$ on $\SX_{\Bm,\unip}$. 
In view of Proposition 2.13 and Proposition 2.3, $Q^+_{\Bm, T}$ 
does not depend on the choice of $\SE$. As in the case of Green functions, 
the definition of $Q^+_{\Bm, T}$ can be generalized to the case where $G$ 
is replaced by $Z_G(s)$, in which case we denote it as $Q^{+, Z_G(s)}_{\Bm, T}$.
Now take $(su, \Bv) \in (\SX^+_{\Bm})^F$ under the setting in Theorem 2.6.
Then by Proposition 2.13 and Theorem 2.6, 
the value $\x^+_{\Bm,T,\SE}(su, \Bv) = \x_{\Bm, T,\SE}(su,\Bv)$
 can be described 
as a linear combination of various Green functions $Q^{Z_G(s)}_{\Bm, xTx\iv}(u,\Bv)$
such that $s \in xTx\iv$.
One can check that if $(su, \Bv) \in \SX^+_{\Bm}$, then 
$(u,\Bv) \in \SX^{+,Z_G(s)}_{\Bm, \unip}$. 
It follows, by Proposition 2.13, that 
$Q^{Z_G(s)}_{\Bm, xTx\iv}(u,\Bv) = Q^{+,Z_G(s)}_{\Bm,xTx\iv}(u,\Bv)$.
Thus as a corollary to Theorem 2.6, we have

\begin{cor} [Character formula for $\x^+_{\Bm, T,\SE}$] %%%% Corollary 2.14
Under the assumption in Theorem 2.6, we have
\begin{equation*}
\x^+_{\Bm, T,\SE}(su,\Bv) = |Z_G(s)^F|\iv 
         \sum_{\substack{x \in G^F \\
                             x\iv sx \in T^F}}Q^{+,Z_G(s)}_{\Bm, xTx\iv}(u,\Bv)\th(x\iv sx). 
\end{equation*}
\end{cor}

\par\bigskip\bigskip
\section{Orthogonality relations}

\para{3.1.}
For a fixed $\Bm \in \SQ_{n,r}$, 
we have defined in the previous sections, 
$K_{\Bm, T,\SE}$, $\x_{\Bm, T,\SE}$, $Q_{\Bm, T}$, etc. and 
$K^+_{\Bm, T, \SE}, \x^+_{\Bm, T, \SE}, Q^+_{\Bm, T}$, etc..
From this section by changing the notation, we denote 
$K_{\Bm, T,\SE}, \x_{\Bm, T,\SE}, Q_{\Bm, T}$, etc. by 
$K^-_{\Bm, T,\SE}, \x^-_{\Bm, T,\SE}, Q^-_{\Bm, T}$, etc. by 
attaching the sign ``$-$''.    
In this section, we shall prove the orthogonality relations for the functions 
$\x^{\pm}_{\Bm, T,\SE}$ and $Q^{\pm}_{\Bm,T}$.
Before stating the results, we 
prepare the following.

%%%%
\begin{prop} %%%  Proposition 3.2.
Let $K = K^{\ve}_{\Bm, T, \CE}$ $($resp. $K' = K^{\ve'}_{\Bm', T', \CE'}$$)$
be a complex on $\SX$ associated to $\Bm$ and 
$(M_{p_i})$ $($ resp. $\Bm'$ and  $(M'_{p'_i}))$, 
where $\ve, \ve' \in \{ +, - \}$.  
Assume that $\SE'$ is the constant sheaf on $T'$, and $\SE$ is a
non-constant sheaf on $T$.  Then we have
\begin{equation*}
\BH_c^i(\SX, K \otimes K') = 0 \quad\text{ for all } i.
\end{equation*}
\end{prop}

\begin{proof}
We prove the proposition by a similar argument as in the proof 
of Proposition 7.2 in [L2].  In the discussion below, we consider 
the case where $\ve = -, \ve' = +$.  The other cases are dealt
similarly.
Let $B'$ be the Borel subgroup of $G$ containing $T'$, which is the 
stabilizer of the total flag $(M'_i)$, 
and $P'$ the parabolic subgroup of $G$ containing $B'$ 
which is  the stabilizer of the partial flag $(M'_{p'_i})$.  
Let $L'$ be the Levi subgroup of $P'$ containing $T'$, and $U_{P'}$
the unipotent radical of $P'$.  Put 
$B'_{L'} = B' \cap L'$. 
We consider the fibre product 
$Z = \wt\SX_{\Bm} \times_{\SX} \wt\SX^+_{\Bm'}$, where
$Z$ can be written as 
\begin{equation*}
\begin{split}
Z = \{ (gB, &g'B'_{L'}, x, \Bv) \in G/B \times G/B'_{L'} 
         \times G \times  V^{r-1}  \\
&\mid g\iv xg \in B, g'^{-1}x{g'} \in B', g\iv \Bv \in \prod_i M_{p_i}, 
      g'^{-1} \Bv \in \prod_i M'^+_{p'_i} \}.
\end{split}
\end{equation*}
Let $\SL = \a^*\SE$ and $\SL' = (\a^+)^*\SE'$.  
Since $K = (\pi_{\Bm})_!\SL$ and $K' = (\pi^+_{\Bm})_!\SL'$, up to
 shift, by the K\"unneth 
formula, we have
\begin{equation*}
\tag{3.2.1}
\BH^i_c(\SX, K \otimes K')
\simeq H_c^i(Z, \SL \boxtimes \SL')
\end{equation*}
up to the degree shift.  
Hence in order to prove the proposition, it is enough to show
that the right hand side of (3.2.1) is equal to zero for each $i$.
For each $G$-orbit $\SO$ of $G/B \times G/B'$, put
\begin{equation*}
Z_{\SO} = \{ (gB, g'B'_{L'}, (x, \Bv)) \in Z \mid (gB, g'B') \in \SO\}.
\end{equation*}
Then $Z = \coprod_{\SO} Z_{\SO}$ is a finite partition, and $Z_{\SO}$
is a locally closed subvariety of $Z$. Hence it is enough to show that
$H_c^i(Z_{\SO}, \SL\boxtimes\SL') = 0$ for any $i$.  We consider the
morphism $\vf_{\SO} : Z_{\SO} \to \SO, 
       (gB, g'B'_{L'}, (x, \Bv)) \mapsto (gB, g'B')$.
Then by the Leray spectral sequence, we have
\begin{equation*}
H_c^i(\SO, R^j(\vf_{\SO})_!(\SL\boxtimes\SL')) 
    \Rightarrow H_c^{i+j}(Z_{\SO},\SL\boxtimes \SL'). 
\end{equation*}
Thus it is enough to show that 
$R^j(\vf_{\SO})_!(\SL\boxtimes \SL') = 0$ for any $j$, which is 
equivalent to the statement that 
$H_c^j(\vf_{\SO}\iv(\xi), \SL\boxtimes\SL') = 0$ for any $j$ and 
any $\xi \in \SO$.
Since $\SL\boxtimes\SL'$ is a $G$-equivariant local system, it is
enough show this for a single element $\xi \in \SO$.
Thus, we may choose $\xi = (B, nB') \in \SO$, where 
$n \in G$ is such that $nT'n\iv = T$.
Then $\vf_{\SO}\iv(\xi)$ is given as 
\begin{equation*}
\begin{split}
\vf_{\SO}\iv(\xi) = \{ (B, &nuB'_{L'}, x, \Bv) \mid x \in B, n\iv xn \in B', \\
       &\Bv \in \prod_iM_{p_i}, u\iv n\iv \Bv \in \prod_i M'^+_{p'_i},  u \in U_{P'} \}.  
\end{split}
\end{equation*}
Thus $Y = \vf_{\SO}\iv(\xi)$ is isomorphic to $(B \cap nB'n\iv) \times Y_1$, where
\begin{equation*}
Y_1 = \{(u, \Bv) \in U_{P'} \times V^{r-1} \mid \Bv \in 
            \prod_iM_{p_i} \cap \prod_inu( M'^+_{p'_i}) \}. 
\end{equation*}
Let $h : Y \to U_{P'}$ be the map obtained from the projection 
$h_1 : Y_1 \to U_{P'}$. Again by using the Leray spectral sequence associated to 
the map $h$, in order to show $H^i_c(Y, \SL\boxtimes\SL') = 0$, it is enough to see 
that $H^i_c(h\iv(u), \SL \boxtimes\SL') = 0$ for any $u \in U_{P'}$.
For each $u \in U_{P'}$, the fibre $h_1\iv(u)$ has a structure of an affine space. 
Hence $h\iv(u)$ is isomorphic to the direct product of $(B \cap nB'n\iv)$ with 
an affine space. Since $B \cap nB'n\iv  \simeq T \times (U \cap n U' n\iv)$, 
where $U'$ is the unipotent radical of $B'$, $h\iv(u)$ can be written as 
$h\iv(u) \simeq T \times Y_2$ with an affine space $Y_2$.  
If we denote by $p : h\iv(u) \to T$ the projection on $T$,   
the restriction of $\SL \boxtimes \SL'$ on $h\iv(u)$ coincides with 
$p^*(\SE \otimes f^*\SE') = p^*\SE \simeq \SE \boxtimes \Ql$ since $\SE'$ is 
the constant sheaf, where 
$f : T \to T' = n\iv Tn$. 
Hence we have only to show that $H^i_c(T, \SE) = 0$ for any $i$.
But since $\SE$ is a non-constant tame local system on $T$, we have
$H_c^i(T,\SE) = 0$ for any $i$.  This proves the proposition.
\end{proof}

\para{3.3.} 
We consider the complexes 
$K^{\pm}_{\Bm, T, \CE}, K^{\pm}_{\Bm', T',\CE'}$, and their 
characteristic functions 
$\x^{\pm}_{\Bm, T, \CE}, \x^{\pm}_{\Bm', T', \CE'}$.
We put $N(T, T') = \{ n \in G \mid n\iv Tn = T'\}$. 
For $\ve, \ve' \in \{ +, -\}$ and $n \in N(T, T')$, we define
$a_{\ve, \ve'}(\Bm, \Bm'; n)$ by
\begin{equation*}
\tag{3.3.1}
a_{\ve,\ve'}(\Bm, \Bm'; n) = 
  \sum_{i=1}^{r-1}\dim (M^{\ve}_{p_i} \cap n({M'}^{\ve'}_{p'_i})),
\end{equation*}
where $M^+_{p_i}$ is as before, and we put $M_{p_i} = M^-_{p_i}$.  We define 
$M'^{\ve'}_{p'_i}$ similarly. 
Also put 
\begin{equation*}
\tag{3.3.2}
p_-(\Bm) = \dim \prod_{i=1}^{r-1} M_{p_i} = \sum_{i=1}^{r-1}p_i, \quad
p_+(\Bm) = \dim \prod_{i=1}^{r-1} M^+_{p_i} = p_{r-1}.
\end{equation*}  

\par
We have the following orthogonality relations, which is 
an analogue of Theorems 9.2 and 9.3 in [L2].
Also see Theorems 3.4 and 3.5 in [SS2] for the case of exotic symmetric space of level 2.

%%%%%
%%%%%
\begin{thm}[Orthogonality relations for $\x^{\pm}_{\Bm, T, \SE}$]  %%% Theorem 3.4
Let $\SE = \SE_{\th}, \SE' = \SE_{\th'}$
with $\th \in (T^F)\wg, \th' \in ({T'}^F)\wg$.  Then we have 
\begin{equation*}
\tag{3.4.1}
\begin{split}
(-1)^{p_{\ve}(\Bm) + p_{\ve'}(\Bm')}&|G^F|\iv
   \sum_{z \in \SX^F}
        \x^{\ve}_{\Bm, T, \SE}(z)\x^{\ve'}_{\Bm',T',\SE'}(z) \\
&= |T^F|\iv|{T'}^F|\iv\sum_{\substack{n \in N(T,T')^F \\ t \in T^F}}
            \th(t)\th'(n\iv tn)q^{a_{\ve,\ve'}(\Bm, \Bm';n)}.
\end{split}
\end{equation*}
\end{thm}

%%%%%
%%%%%
\begin{thm}[Orthogonality relations for Green functions]  %%%  Theorem 3.5.
\begin{equation*}
\tag{3.5.1}
\begin{split}
(-1)^{p_{\ve}(\Bm) + p_{\ve'}(\Bm')}&|G^F|\iv
   \sum_{z \in \SX^F\uni}
       Q^{\ve}_{\Bm, T}(z)Q^{\ve'}_{\Bm', T'}(z)  \\
   &= |T^F|\iv |{T'}^F|\iv 
       \sum_{n \in N(T, T')^F}q^{a_{\ve,\ve'}(\Bm, \Bm';n)}.
\end{split}
\end{equation*}
\end{thm}

\para{3.6.}
As was discussed in Section 2, the functions $\x^{\pm}_{\Bm, T,\SE}, Q^{\pm}_{\Bm, T}$
make sense if we replace $G$ by its subgroup $Z_G(s)$, and 
$G \times V^{r-1}$ by $Z_G(s) \times V^{r-1}$ for a semisimple element $s \in G^F$. 
Theorem 3.4 and 3.5 are then formulated for this general setting.  
In what follows, we shall prove Theorem 3.4 and 3.5 
simultaneously under this setting. 
\par
First we note 
\par\medskip\noindent
(3.6.1) \  Theorem 3.4 holds if 
$\th'$ is the trivial character and $\th$ is a non-trivial character.  
\par\medskip
In fact, by the trace formula 
for the Frobenius maps, the left hand side of (3.4.1) coincides,
up to scalar,  with 
\begin{equation*}
\sum_i(-1)^i\Tr(F^*, \BH_c^i(\SX, K^{\ve}_{\Bm, T, \SE}
        \otimes K^{\ve'}_{\Bm',T',\SE'})),
\end{equation*}
where $F^*$ is the 
isomorphism induced from 
$\vf: F^*K^{\ve}_{\Bm,T,\SE} \isom K^{\ve}_{\Bm,T,\SE}$ 
and 
\linebreak 
$\vf': F^*K^{\ve'}_{\Bm',T',\SE'} \isom K^{\ve'}_{\Bm',T',\SE'}$. 
Since $\SE'$ is a constant sheaf and $\SE$ is a non-constant sheaf, 
by Proposition 3.2, 
the left hand side of (3.4.1) is equal to zero.  
On the other hand, the right hand side of (3.4.1) is equal to zero
by the orthogonality relations for irreducible characters of 
$T^F$.  Hence (3.6.1) holds. 
\par
Next we show
\par\medskip\noindent
(3.6.2) \ Theorem 3.4 holds if there exist $F$-stable Borel subgroups 
$B, B'$ such that $B \supset T, B' \supset T'$.
\par\medskip
Since $B, B'$ are $F$-stable, we can compute $\x^{\pm}_{\Bm,T,\SE}$ by using the 
complexs $(\pi_{\Bm})_!\a^*\SE$ and $(\pi^+_{\Bm})_!(\a^+)^*{\SE}$ defined in 
Section 1 and 2.  
Then by the trace formula for the Frobenius map, we have

\begin{align*}
(-1)^{d_{\Bm}}\x_{\Bm,T,\SE}(x, \Bv) &= |B^F|\iv\sum_{\substack{g \in G^F \\
             g\iv xg \in B^F  \\
             g\iv \Bv \in \prod_iM_{p_i}}}\x_{\SE, \vf_0}(p_T(g\iv xg)),  \\
(-1)^{d^+_{\Bm'}}\x^+_{\Bm', T',\SE'}(x, \Bv) &= (-q)^{-\dim U_{P'}}
         |B'^F_{L'}|\iv\sum_{\substack{g \in G^F \\
             g\iv xg \in B'^F \\
             g\iv \Bv \in \prod_iM'^+_{p_i}}}\x_{\SE', \vf'_0}(p_{T'}(g\iv xg)). 
\end{align*}
Since $d_{\Bm} = n^2 + p_-(\Bm)$ by Lemma 1.2 and 
$d^+_{\Bm'} = n^2 + \dim U_{P'} + p_+(\Bm')$ by (2.7.2), we have

\begin{equation*}
\tag{3.6.3}
\begin{split}
(-1)^{p_{\ve}(\Bm) + p_{\ve'}(\Bm')}&|G^F|\iv\sum_{(x, \Bv) \in \SX^F}
                 \x^{\ve}_{\Bm,T,\SE}(x, \Bv)\x^{\ve'}_{\Bm',T',\SE'}(x,\Bv) \\
        &= |G^F|\iv|B^F|\iv|B'^F|\iv
             \sum _{(*)}\th(p_T(g\iv xg))\th'(p_{T'}(g'^{-1}xg')), 
\end{split}
\end{equation*}
where the sum (*) is taken over all $x \in G^F, \Bv \in (V^{r-1})^F, g,g' \in G^F$
such that $g\iv xg \in B^F, g'^{-1}xg' \in B'^F, g\iv \Bv \in \prod_i (M^{\ve}_{p_i})^F,
             g'^{-1}\Bv \in \prod_i (M^{\ve'}_{p'_i})^F$.  
We change the variables; put $y = g\iv xg, h = g\iv g', \Bv' = g\iv \Bv$.  Then the
condition (*) is equivalent to the condition (**) 
$g \in G^F,h \in G^F, y \in B^F, \Bv' \in \prod_i(M^{\ve}_{p_i})^F$ 
such that $h\iv yh \in B'^F, h\iv \Bv' \in \prod_i (M^{\ve'}_{p'_i})^F$. 
We consider the partition of $G$ into double cosets $B\backslash G/B'$.
For each $F$-stable coset $BnB'$ of $G$, we may assume that $n\iv Tn = T'$ 
since $T$ and $T'$ are $G^F$-conjugate.  Then 
the sum $\sum_{(*)}$ in (3.6.3) can be written as

\begin{equation*}
|G^F||U^F||U'^F|
     \sum_{n \in (B\backslash G/B')^F} 
   \sum_{\substack{ h \in (TnT')^F \\ t \in T^F}}\th(t)\th'(n\iv tn)q^{a(\Bm,\Bm';n)}.
\end{equation*} 
Thus Theorem 3.4 holds.  (3.6.2) is proved.

\para{3.7.}
In this subsection, we show that Theorem 3.4 holds for $G$ under 
the assumption that Theorem 3.5 holds for subgroups $Z_G(s)$ of $G$. 
By making use of the character formulas in Theorem 2.6 and Corollary 2.15,
we have
\begin{equation*}
\begin{split}
&(-1)^{p_{\ve}(\Bm) + p_{\ve'}(\Bm')}|G^F|\iv \sum_{z \in \SX^F}\x^{\ve}_{\Bm, T, \SE}(z)
               \x^{\ve'}_{\Bm', T', \SE'}(z) \\
  &= |G^F|\iv (-1)^{p_{\ve}(\Bm) + p_{\ve'}(\Bm')}\sum_{\substack{s \in G^F \\
            x, x' \in G^F \\
            x\iv sx \in T^F \\
            {x'}\iv sx' \in {T'}^F}}
      f(s, x, x')|Z_G(s)^F|^{-2} \th(x\iv sx)\th'({x'}\iv sx'),
\end{split}
\end{equation*}
where 
\begin{equation*}
f(s,x,x') = \sum_{\substack{u \in Z_G(s)^F\uni \\
             \Bv \in (V^{r-1})^F}}
    Q^{\ve, Z_G(s)}_{\Bm, xTx\iv}(u, \Bv)
    Q^{\ve', Z_G(s)}_{\Bm', x'T{x'}\iv}(u, \Bv).
\end{equation*}
By applying Theorem 3.5 for $Z_G(s)$, we see that
\begin{equation*}
f(s,x,x') = (-1)^{p_{\ve}(\Bm) + p_{\ve'}(\Bm')}
    |Z_G(s)^F||T^F|\iv |{T'}^F|\iv\sum_{\substack{
             n \in Z_G(s) \\ n\iv xTx\iv n = x'T'{x'}\iv}}
                q^{a_{\ve,\ve'}(\Bm, \Bm'; n)}.  
\end{equation*}
It follows that the previous sum is equal to
\begin{equation*}
\begin{split}
&|G^F|\iv |T^F|\iv |{T'}^F|\iv   \\
   &\times \sum_{\substack{
           s \in G^F \\
           x, x' \in G^F \\           
           x\iv sx \in T^F \\
           {x'}\iv sx' \in {T'}^F \\
                    }}
|Z_G(s)^F|\iv \sum_{\substack{n \in Z_G(s)^F \\
                  n\iv xTx\iv n = x'T'{x'}\iv}}
                      \th(x\iv sx)\th'({x'}\iv sx') 
                        q^{a_{\ve, \ve'}(\Bm, \Bm';n)}.
\end{split}
\end{equation*}
Now put $t = x\iv sx \in T^F$ and $y = x\iv nx'$.
Then $y \in G^F$ such that $y\iv Ty = T'$. 
Under this change of variables, the above sum can be rewritten as
\begin{equation*}
|G^F|\iv |T^F|\iv |{T'}^F|\iv 
\sum_{\substack{ x \in G^F \\
                        t \in T^F}}
|Z_G(t)^F|\iv \sum_{\substack{ y \in N(T,T')^F\\ 
     x' \in (xZ_G(t)y)^F}}\th(t)\th'(y\iv t y)
   q^{a_{\ve, \ve'}(\Bm, \Bm'; y)},
\end{equation*}
which is equal to
\begin{equation*}
|T^F|\iv|{T'}^F|\iv\sum_{t \in T^F}\sum_{y \in N(T,T')^F}
       \th(t)\th'(y\iv ty)q^{a_{\ve,\ve'}(\Bm, \Bm'; y)}.
\end{equation*}
Thus our assertion holds.

\para{3.8.}
We shall show that Theorem 3.5 holds for $G$ under the assumption that
it holds for $Z_G(s)$ if $s$ is not central.  Hence Theorem 3.4 holds
for such groups by 3.7.
Put 
\begin{align*}
A &= (-1)^{p_{\ve}(\Bm) + p_{\ve'}(\Bm')}
   |G^F|\iv\sum_{z \in \SX^F\uni}
       Q^{\ve}_{\Bm, T}(z)Q^{\ve'}_{\Bm', T'}(z), \\
B &= |T^F|\iv |{T'}^F|\iv 
       \sum_{n \in N(T, T')^F}q^{a_{\ve,\ve'}(\Bm, \Bm';n)}.
\end{align*}
We want to show that $A = B$.
By making use of a part of the arguments in 3.7 (which can be applied to
the case where $s \notin Z(G)^F$), we see that
\begin{equation*}
\tag{3.8.1}
\begin{split}
&(-1)^{p_{\ve}(\Bm) + p_{\ve'}(\Bm')}|G^F|\iv \sum_{z \in \SX^F}\x^{\ve}_{\Bm, T, \SE}(z)
               \x^{\ve'}_{\Bm', T', \SE'}(z) 
- A \sum_{s \in Z(G)^F}\th(s)\th'(s)  \\
&= |T^F|\iv|{T'}^F|\iv\sum_{t \in T^F}\sum_{y \in N(T,T')^F}
       \th(t)\th'(y\iv ty)q^{a_{\ve,\ve'}(\Bm, \Bm'; y)} - B|Z(G)^F|.
\end{split}
\end{equation*}
This formula holds for any $\th \in (T^F)\wg$ and 
$\th' \in ({T'}^F)\wg$.  The case where $G = T$ is included in the case 
discussed in (3.6.2).  So we may assume that $G \ne T$, hence 
$Z(G) \ne G$.  Now assume that $q > 2$.  Then one can 
find a linear character $\th$ on $T^F$ such that 
$\th|_{ZG)^F} = \id$ and that $\th \ne \id$.  We choose $\th'$ the
identity character on ${T'}^F$.  
Then by (3.6.1), the first term of the left hand side of (3.8.1) 
coincides with the first term of the right hand side.  
This implies that $A = B$ as asserted. 
The remaining case is the case where $q = 2$.  If $T$ is not a split torus, 
still one can find a linear character $\th$ satisfying the above property,
and the above discussion can be applied.  So we may assume that both of
$T, T'$ are $\BF_q$-split.  But in this case Theorem 3.4 holds by (3.6.2).  
Hence the first term of the left hand side of (3.8.1) coincides with the first 
term of the right hand side.  By choosing the identity characters $\th, \th'$, 
again  we have $A = B$.  
Thus our assertion holds.

\para{3.9.}
We are now ready to prove Theorems 3.4 and 3.5.
First note that Theorem 3.4 and 3.5 hold for $Z_G(s)$ in the case where 
$Z_G(s) = T$, i.e., in the case where $s$ is regular semisimple. 
In fact, in that case, we have $T = T'$ and $N(T, T') = T$.
We have $\x^-_{\Bm, T, \SE}(t, \Bv) = (-1)^{d_{\Bm}}\th(t)$ if 
$\Bv \in \prod_i (M^{\ve}_i)^F$  
and is equal to zero otherwise.  
A similar formula holds for $\x^+_{\Bm, T,\SE}$ by replacing $d_{\Bm}$ by $d^+_{\Bm}$.
Theorem 3.4 and 3.5 follows from this.  Now 
by induction of the semisimple rank of $Z_G(s)$, we may assume that Theorem 3.5
holds for $Z_G(s)$ if $s$ is non central.  Then by 3.8, Theorem  3.5 holds for $G$, 
and by 3.7, 
Theorem 3.4 holds for $G$.  This completes the proof of Theorems 3.4 and 3.5. 

\para{3.10.}
An $r$-tuple of partitions $\Bla = (\la^{(1)}, \dots, \la^{(r)})$ is called 
an $r$-partition of $n$ if $n = \sum_{i=1}^r|\la^{(i)}|$.
We denote by $\SP_{n,r}$ the set of all $r$-partitions of $n$.  In the case where
$r = 1$, we denote $\SP_{n,1}$ simply by $\SP_n$.   
For each 
$\Bm \in \SQ_{n,r}$, we denote by $\SP(\Bm)$ the subset of $\SP_{n,r}$ 
consisting of $\Bla$ such that $|\la^{(i)}| = m_i$. 
Let $S_{\Bm} = S_{m_1} \times \cdots \times S_{m_r}$ be the Young subgroup 
of $S_n$ corresponding to $\Bm = (m_1, \dots, m_r) \in \SQ_{n,r}$.   
Recall that irreducible characters of $S_n$ are parametrized by partitions of $n$.
We denote by $\x^{\la}$ the irreducible character of $S_n$ corresponding to 
$\la \in \SP_n$ (here $\x^{(n)}$ corresponds to the trivial character, and 
$\x^{(1^n)}$ corresponds to the sign character).
Thus the irreducible characters of $S_{\Bm}$ are parametrized by $\SP(\Bm)$. 
We denote by
$\x^{\Bla} = \x^{\la^{(1)}}\boxtimes\cdots \boxtimes \x^{\la^{(r)}}$
the irreducible character of $S_{\Bm}$ corresponding to $\Bla \in \SP(\Bm)$. 
\par
We identify $(W_0)_{\Bm}$ with $S_{\Bm}$. 
For each $w \in S_{\Bm}$, let $T_w$ be an $F$-stable maximal torus in $G$ 
as given in 2.1. 
For each $\Bla \in \SP(\Bm)$, we define functions 
$Q^{\pm}_{\Bla}$ on $\SX\uni^F$
by 
\begin{equation*}
\tag{3.10.1}
Q^{\pm}_{\Bla} = |S_{\Bm}|\iv \sum_{w \in S_{\Bm}}
                    \x^{\Bla}(w)Q^{\pm}_{\Bm, T_w}.
\end{equation*}
\par
Let $\SC_q = \SC_q(\SX\uni)$ 
be the $\Ql$-space of all $G^F$-invariant $\Ql$-functions on 
$\SX\uni^F$.
We define a bilinear form $\lp f, h\rp $ on $\SC_q$ by 
\begin{equation*}
\tag{3.10.2}
\lp f, h \rp = \sum_{z \in \SX\uni^F}f(z)h(z).
\end{equation*}
Concerning the functions $Q_{\Bla}^{\pm}$, the following formula holds.

\begin{prop}  %%% Prop. 3.11.
For
any $\Bla \in \SP(\Bm)$, $\Bmu \in \SP(\Bm')$ and $\ve, \ve' \in \{ +,-\}$, we have 
\begin{equation*}
\begin{split}
\lp Q^{\ve}_{\Bla}, Q^{\ve'}_{\Bmu} \rp 
    &= (-1)^{p_{\ve}(\Bm) + p_{\ve'}(\Bm')}|S_{\Bm}|\iv |S_{\Bm'}|\iv |G^F|  \\
        &\times \sum_{\SO \in S_{\Bm}\backslash S_n /S_{\Bm'}}
             q^{a_{\ve,\ve'}(\Bm, \Bm'; n_{\SO})}
          \sum_{\substack{
                     x\in \SO \\ w \in S_{\Bm} \cap xS_{\Bm'}x\iv }}
   |T_w^F|\iv\x^{\Bla}(w)\x^{\Bmu}(x\iv wx),
\end{split}
\end{equation*}
where $n_{\SO}$ is an element in $N(T_w, T_{x\iv wx})^F $ 
attached to $\SO$. 
\end{prop}

\begin{proof}
By making use of the orthogonality relations for Green functions
(Theorem 3.5), we have
\begin{align*}
(-1)&^{p_{\ve}(\Bm) + p_{\ve'}(\Bm')}|G^F|\iv\lp Q^{\ve}_{\Bla}, Q^{\ve'}_{\Bmu} \rp \\
&= |S_{\Bm}|\iv |S_{\Bm'}|\iv (-1)^{p_{\ve}(\Bm) + p_{\ve'}(\Bm')}
         |G^F|\iv\sum_{\substack{w \in S_{\Bm} \\
     w' \in S_{\Bm'}}}\x^{\Bla}(w)\x^{\Bmu}(w') 
        \lp Q^{\ve}_{\Bm, T_w}, Q^{\ve'}_{\Bm', T_{w'}} \rp \\
&= |S_{\Bm}|\iv |S_{\Bm'}|\iv \sum_{\substack{
                  w \in S_{\Bm} \\ w' \in S_{\Bm'}}}
        \x^{\Bla}(w)\x^{\Bmu}(w')     
     |T_w^F|\iv |T_{w'}^F|\iv   
         \sum_{n \in N(T_w, T_{w'})^F} 
                     q^{a_{\ve, \ve'}(\Bm, \Bm';n)}.
\end{align*}
Let $P_{\Bm}$ be the parabolic subgroup of 
$G$ containing $B_0$ which is the stabilizer of the partial 
flag $(M_{0,p_i})$ with respect to $\Bm$, and define $P_{\Bm'}$ 
similarly. 
We put $P_{\Bm} = L_{\Bm}U_{\Bm}$, where $L_{\Bm}$ is the Levi 
subgroup of $P_{\Bm}$ containing $T_0$ and $U_{\Bm}$ is the unipotent 
radical of $P_{\Bm}$.  Define $P_{\Bm'} = L_{\Bm'}U_{\Bm'}$ 
similarly. 
We may assume that $T_w \subset P_{\Bm}$ and $T_{w'} \subset P_{\Bm'}$.
Thus $a_{\ve,\ve'}(\Bm,\Bm';n)$ is defined with respect to 
$(M^{\ve}_{0,p_i}), (M^{\ve'}_{0, p'_i})$.
It follows that 
$a_{\ve,\ve'}(\Bm, \Bm'; n)$ is independent of the choice of 
$n \in \wt\SO^F$ for each orbit 
$\wt\SO \in L_{\Bm}\backslash G /L_{\Bm'}$.  Let $\wh\SO$ be an 
element in $P_{\Bm}\backslash G/P_{\Bm'}$ containing $\wt\SO$.
Then 
\begin{equation*}
\wh\SO \cap N(T_w, T_{w'})^F 
  = \wt\SO \cap N(T_w, T_{w'})^F
\end{equation*}
since $\wh\SO = \bigcup_{u \in U_{\Bm}, u' \in U_{\Bm'}}u\wt\SO u'$.
It follows that  
$a_{\ve,\ve'}(\Bm, \Bm'; n)$ is independent of the choice for any 
$n \in \wh\SO \cap N(T_w, T_{w'})^F$.
Note that $P_{\Bm}\backslash G /P_{\Bm'} 
   \simeq S_{\Bm}\backslash S_n /S_{\Bm'}$.
We denote by $\SO \in S_{\Bm}\backslash S_n /S_{\Bm'}$ the orbit in $S_n$ corresponding 
to $\wh\SO$, and choose $n_{\SO} \in \wh\SO \cap N(T_w, T_{w'})^F$ for each $\SO$.
Then the last sum is equal to
\begin{equation*}
\tag{3.11.1}
\begin{split}
|S_{\Bm}|\iv &|S_{\Bm'}|\iv \sum_{\substack{
  w \in S_{\Bm} \\  w' \in S_{\Bm'}}}
    \x^{\Bla}(w)\x^{\Bmu}(w')|T_w^F|\iv |T_{w'}^F|\iv |N_{L_{\Bm}}(T_w)^F|
              |N_{L_{\Bm'}}(T_{w'})^F|  \\
&\times \sum_{\substack{
            \SO \in S_{\Bm}\backslash S_n/S_{\Bm'} \\
                 n_{\SO}\iv T_w n_{\SO} = T_{w'} }} 
 |N_{L_{\Bm}}(T_w)^F \cap n_{\SO}N_{L_{\Bm'}}(T_{w'})^Fn_{\SO}\iv|\iv 
    q^{a_{\ve,\ve'}(\Bm, \Bm'; n_{\SO})}.
\end{split}
\end{equation*}
Here we have
\begin{equation*}
\tag{3.11.2}
|N_{L_{\Bm}}(T_w)^F|/|T_w^F| = |Z_{S_{\Bm}}(w)|, \quad
|N_{L_{\Bm'}}(T_{w'})^F|/|T_{w'}^F| = |Z_{S_{\Bm'}(w')}|. 
\end{equation*}
Moreover since $n_{\SO}\iv T_wn_{\SO} = T_{w'}$ for $n_{\SO} \in G^F$, 
there exists $x_{\SO} \in S_n$ such that $x_{\SO}\iv w x_{\SO} = w'$.
Now $\ad n_{\SO}\iv$ gives an isomorphism $N_G(T_w) \isom N_G(T_{w'})$, 
which induces an isomorphism $N_G(T_w)/T_w \isom N_G(T_{w'})/T_{w'}$, 
hence by taking the $F$-fixed point subgroups on both side, 
an isomorphism $Z_{S_n}(w) \isom Z_{S_n}(w')$.  We may assume that 
this isomorphism is induced by $\ad x_{\SO}\iv$.  It follows that
\begin{equation*}
(N_{L_m}(T_w)^F \cap n_{\SO} N_{L_{\Bm'}}(T_{w'})^Fn_{\SO}\iv)/T_w^F
 \simeq Z_{S_{\Bm}}(w) \cap x_{\SO} Z_{S_{\Bm'}}(w')x_{\SO}\iv,
\end{equation*}
and so
\begin{equation*}
\tag{3.11.3}
|N_{L_{\Bm}}(T_w)^F \cap n_{\SO}N_{L_{\Bm'}}(T_{w'})^Fn_{\SO}\iv|
  = |T_w^F||Z_{S_{\Bm}}(w) \cap x_{\SO}Z_{S_{\Bm'}}(w')x_{\SO}\iv|.
\end{equation*}
Substituting (3.11.2) and (3.11.3) into (3.11.1), the formula (3.11.1)
turns out to be
\begin{equation*}
\begin{split}
|S_{\Bm}|\iv&|S_{\Bm'}|\iv\sum_{\substack{
         w \in S_{\Bm} \\  w' \in S_{\Bm'}}}
    \x^{\Bla}(w)\x^{\Bmu}(w')
       |T_w^F|\iv |Z_{S_{\Bm}}(w)||Z_{S_{\Bm'}}(w')| \\
&\times \sum_{\substack{
                 \SO \in S_{\Bm}\backslash S_n/S_{\Bm'} \\
                   n_{\SO}\iv T_w n_{\SO} = T_{w'}}} 
    |Z_{S_{\Bm}}(w) \cap x_{\SO}Z_{S_{\Bm'}}(w')x_{\SO}\iv|\iv
              q^{a_{\ve,\ve'}(\Bm, \Bm'; n_{\SO})}.         
\end{split}
\end{equation*}
But we have
\begin{equation*}
|Z_{S_{\Bm}}(w)||Z_{S_{\Bm'}}(w')||{Z_{S_{\Bm}}(w)} \cap 
        x_{\SO}Z_{S_{\Bm'}}(w')x_{\SO}\iv|\iv 
 = |Z_{S_{\Bm}}(w)x_{\SO}Z_{S_{\Bm'}}(w')|,
\end{equation*}
and for given $w \in S_{\Bm}, w' \in S_{\Bm'}$, 
the choice of $x \in S_n$ such that $w' = x\iv wx$ is
given by $x \in Z_{S_{\Bm}}(w)x_{\SO}Z_{S_{\Bm'}}(w')$. 
Hence the last formula is equal to
\begin{equation*} 
|S_{\Bm}|\iv |S_{\Bm'}|\iv \sum_{\substack{
     w \in S_{\Bm} \\ w' \in S_{\Bm'}}}
          \sum_{\substack{
                 \SO \in S_{\Bm}\backslash S_n/S_{\Bm'} \\ 
            x \in \SO \\  w' = x\iv wx  }}
               \x^{\Bla}(w)\x^{\Bmu}(w')|T_w^F|\iv 
                 q^{a_{\ve,\ve'}(\Bm,\Bm'; n_{\SO})}.
\end{equation*}
This proves the proposition.
\end{proof}

\section{Unipotent variety}

\para{4.1.}
We express an $r$-partition $\Bla$ by $\Bla = (\la^{(i)}_j)$, where 
$\la^{(i)} = (\la^{(i)}_1, \la^{(i)}_2, \dots, \la^{(i)}_m)$ is a partition 
with $\la^{(i)}_m \ge 0$ for some fixed number $m$. 
For an $r$-partition $\Bla$, we define a sequence of non-negative integers 
$c(\Bla)$ associated to $\Bla$ by 
\begin{equation*}
c(\Bla) = (\la^{(1)}_1, \la^{(2)}_1, \dots, \la^{(r)}_1, 
            \la^{(1)}_2, \la^{(2)}_2, \dots, \la^{(r)}_2,
            \dots, 
            \la^{(1)}_m, \la^{(2)}_m, \dots, \la^{(r)}_m). 
\end{equation*}
\par
For a sequence of non-negative integers, $\la = (\la_1, \la_2, \dots, \la_m), 
\mu = (\mu_1, \mu_2, \dots,\mu_m)$,  we denote by $\la \le \mu$ 
if 
\begin{equation*}
\la_1 + \cdots + \la_k \le \mu_1 + \cdots + \mu_k
\end{equation*}
for $k = 1, 2, \cdots, m$.
We define a dominance order $\le $  on $\SP_{n,r}$ by 
the condition that $\Bla \le \Bmu$ if $c(\Bla) \le c(\Bmu)$.
In the case where $r = 1$, this is the standard dominance order 
on the set of partitions.  
In the case of $r = 2$, this coincides with the partial order 
used in [AH] and [SS2]. 
\par
For a partition $\la = (\la_1, \dots, \la_m) \in \SP_n$, we define 
$n(\la) \in \BZ$ by $n(\la) = \sum_i(i-1)\la_i$.  
For $\Bla = (\la^{(1)}, \dots, \la^{(r)}) \in \SP_{n,r}$, we define $n(\Bla) \in \BZ$ 
by $n(\Bla) = \sum_{i=1}^rn(\la^{(i)})$.  Hence if we put 
$\nu = \la^{(1)} + \cdots + \la^{(r)}$, we have $n(\Bla) = n(\nu)$.

\para{4.2.}
Let $\SX\uni = G\uni \times V^{r-1}$.  In the case where $r = 2$, it is known 
by [AH], [T] that $G$-orbits of $\SX\uni$ are parametrized by $\SP_{n,2}$. 
The parametrization is given as follows; take $(x,v) \in G\uni \times V$.
Put $E^x = \{ y \in \End(V) \mid xy = yx\}$. $E^x$ is a subalgebra of $\End(V)$ 
containing $x$. Put $W = E^xv$.  Then $W$ is an $x$-stable 
subspace of $V$.  We denote by $\la^{(1)}$ the Jordan type of $x|_W$, and 
by $\la^{(2)}$ the Jordan type of $x|_{V/W} $.  Thus one can define  
$\Bla = (\la^{(1)}, \la^{(2)}) \in \SP_{n,2}$.  We denote by $\SO_{\Bla}$ 
the $G$-orbit containing $(x,v)$.  This gives the required parametrization.
Note that if $(x,v) \in \SO_{\Bla}$, the Jordan type of $x$ is given by
$\la^{(1)} + \la^{(2)}$. 
For $(x,v) \in G\uni \times V$, we say that $(x,v)$ has type $\Bla$ if 
$(x,v) \in \SO_{\Bla}$. 
\par  
The normal basis of $(x,v) \in \SO_{\Bla}$ was constructed in [AH] as follows.
Assume that $\nu = (\nu_1, \dots, \nu_{\ell})$ for $\nu = \la^{(1)} + \la^{(2)}$.
Choose a basis $\{ u_{jk} \mid 1 \le j \le \ell, 1 \le k \le \nu_j\}$ of $V$, 
and define an action of $x$ on the basis by 
$(x-1)u_{j,k} = u_{j, k-1}$ under the convention that $u_{j,0} = 0$. 
Put $v = \sum_{j = 1}^{\ell}u_{j, \la^{(1)}_j}$.  
Then $(x,v)$ gives a representative of the orbit $\SO_{\Bla}$.  

\par
If $r \ge 3$, the number of $G$-orbits in $\SX\uni$ is infinite. In [S3, 5.3], 
the partition of $X$ into $G$-stable pieces $X_{\Bla}$ (possibly a union of 
infinitely many $G$-orbits) labelled by  
$\Bla \in \SP_{n,r}$ is given;

\begin{equation*}
\tag{4.2.1}
\SX\uni = \coprod_{\Bla \in \SP_{n,r}}X_{\Bla}.
\end{equation*}
Following [S3],  we define $X_{\Bla}$ by induction on $r$ as follows. 
Take $(x, \Bv) \in \SX\uni$ with $\Bv = (v_1, \dots, v_{r-1})$.  
Put $W = E^xv_1$, $\ol V = V/W$ and $\ol G = GL(\ol V)$.
We consider the variety 
$\SX'\uni = \ol G \times \ol V^{r-2}$. 
Assume that $(x,v_1) \in G\uni \times V$ is of type $(\la^{(1)}, \nu')$, 
where $\nu = \la^{(1)} + \nu'$ is the Jordan type of $x$.
Let $\ol x$ be the restriction of $x$ on $\ol V$.  Then the Jordan 
type of $\ol x \in GL(\ol V)$
is $\nu'$. Put $\ol \Bv = (\ol v_2, \dots \ol v_{r-1})$, where $\ol v_i$ is the image 
of $v_i$ on $\ol V$.
Thus $(\ol x, \ol \Bv) \in \SX'\uni$. By induction, we have a partition 
$\SX'\uni = \coprod_{\Bmu \in \SP_{n',r-1}}  X'_{\Bmu}$, where $\dim \ol V = n'$.       
Thus there exists a unique $X'_{\Bla'}$ containing $(\ol x, \ol\Bv)$.  
If we write $\Bla' = (\la^{(2)}, \dots, \la^{(r)})$, we have 
$\la^{(2)} + \cdots + \la^{(r)} = \nu'$.
It follows that $\Bla = (\la^{(1)}, \dots \la^{(r)}) \in \SP_{n,r}$. 
We define the type of $(x,\Bv)$ by $\Bla$, and define a subset $X_{\Bla}$ of 
$\SX\uni$ as the set of all $(x,\Bv)$ with type $\Bla$. Then 
$X_{\Bla}$ is a $G$-stable subset of $\SX\uni$, and we obtain the required 
partition (4.2.1).

\par
$X_{\Bla}$ has an alternate description.  Assume that $\Bla \in \SP(\Bm)$ 
for $\Bm \in \SQ_{n,r}$, and let $P = P_{\Bm} = LU_P$ be the parabolic subgroup of $G$ 
attached to $\Bm$ as in 1.5. 
Put $\ol M_{p_i} = M_{p_i}/M_{p_{i-1}}$. 
Then $L \simeq G_1 \times \cdots \times G_r$ with $G_i = GL(\ol M_{p_i})$. 
Let $\SM_{\Bla}$ be the subset of $U \times \prod_{i=1}^{r-1}M_{p_i}$ defined by 
the following properties; take $(x, \Bv)$ with $x \in U$ and $v_i \in M_{p_i}$.
Put $x_i = x|_{\ol M_{p_i}}$, and let $\ol v_i \in \ol M_{p_i}$ 
be the image of $v_i$.  Then $(x,\Bv) \in \SM_{\Bla}$ if the Jordan type of $x$ 
is $\la^{(1)} + \cdots + \la^{(r)}$, the 
$G_i$-orbit of $(x_i, \ol v_i) \in (G_i)\uni \times \ol M_{p_i}$ 
has type $(\la^{(i)}, \emptyset)$ for $i = 1, \dots, r-1$, and the Jordan type of 
$x_r$ is $\la^{(r)}$. 
Let $\ol X_{\la}$ be the closure of $X_{\Bla}$ in $\SX\uni$.  
Then by a similar argument as in the 
proof of [S3, Lemma 6.17], we have 

\begin{equation*}
\tag{4.2.2}
\ol X_{\Bla} = \ol{\bigcup_{g \in G}g\SM_{\Bla}}.
\end{equation*} 
 
By Propositions 5.4 and 5.14 in [S3], we have

\begin{prop}  %%%%  Prop. 4.3.
$X_{\Bla}$ is a $G$-stable, smooth, irreducible, locally closed subvariety of 
$\SX\uni$ with 

\begin{equation*}
\dim X_{\Bla} = (n^2 - n - 2n(\Bla)) + \sum_{i=1}^{r-1}(r-i)|\la^{(i)}|.
\end{equation*}
\end{prop}

Let $\le$ be the dominance order on $\SP_{n,r}$ defined in 4.1.
Concerning the closure relations of $X_{\Bla}$, we have

\begin{prop}[{[S3, Prop. 5.11]}] %%%%  Prop.4.4
Let $\ol X_{\Bla}$ be the closure of $X_{\Bla}$ in $\SX\uni$.  Then 

\begin{equation*}
\ol X_{\Bla} \subset \bigcup_{\Bmu \le \Bla} X_{\Bmu}. 
\end{equation*}
\end{prop}
For each $\Bm \in \SQ_{n,r}$, recall the map 
$\pi_{\Bm,1} : \wt\SX_{\Bm, \unip} \to \SX_{\Bm,\unip}$ 
in 1.1.  Define $\Bla(\Bm) \in \SP_{n,r}$ by 
$\Bla(\Bm) = ((m_1), (m_2), \dots, (m_r))$.  Then we have 

\begin{prop}[{[S3, Prop. 5.9]}]   %%%% Prop. 4.5
For $\Bm \in \SQ_{n,r}$, we have
\begin{enumerate}
\item  $\SX_{\Bm,\unip} = \ol X_{\Bla(\Bm)}$.
\item  $\dim \SX_{\Bm\unip} = n^2 - n + \sum_{i=1}^{r-1}(r-i)m_i$.
\item  For $\Bmu \in \SP(\Bm)$, $X_{\Bmu} \subset \SX_{\Bm, \unip}$.
\end{enumerate}
\end{prop}

\para{4.6.}
As in [S3, 5.10], we define a distinguished element 
$z_{\Bla} \in X_{\Bla}$ as follows.  Put $\nu = (\nu_1, \dots, \nu_{\ell}) \in \SP_n$ for 
$\nu = \la^{(1)} + \cdots + \la^{(r)}$. Take $x \in G\uni$ of Jordan type $\nu$, and let 
$\{ u_{j,k}\mid 1 \le j \le \ell, 1 \le k \le \nu_j \}$ be a Jordan basis of $x$ in $V$
having the property $(x-1)u_{j,k} = y_{j,k-1}$ with the convention $u_{j,0} = 0$. 
We define $v_i \in V$ for $i = 1, \dots, r-1$ by the condition that 
\begin{equation*}
\tag{4.6.1}
v_i = \sum_{1 \le j \le \ell} u_{j, \la^{(1)}_j + \cdots + \la^{(i)}_j}.
\end{equation*}
Put $\Bv = (v_1, \dots, v_{r-1})$ and $z_{\Bla} = (x, \Bv)$. 
Then $z_{\Bla} \in X_{\Bla}$. 
We denote by $\SO^-_{\Bla} \subset X_{\Bla}$ the $G$-orbit containing $z_{\Bla}$.
We also define $z_{\Bla}^+ \in X_{\Bla}$ as follows; put
\begin{equation*}
\tag{4.6.2}
v_i' = \sum_{1 \le j \le \ell}\ol u_{j, \la^{(1)}_j + \cdots + \la^{(i)}_j},
\end{equation*}
where $\ol u_{j, \la^{(1)}_j+\cdots +\la^{(i)}_j} = u_{j,\la^{(1)}_j + \cdots + \la^{(i)}_j}$ 
if $\la^{(i)}_j \ne 0$ and is equal to zero if $\la^{(i)}_j = 0$.  
Put $\Bv' = (v'_1, \dots, v'_{r-1})$, and $z^+_{\Bla} = (x, \Bv')$.  
Then $z^+_{\Bla} \in X_{\Bla}$, and we denote by $\SO_{\Bla}^+$ the 
$G$-orbit in $X_{\Bla}$ containing 
$z^+_{\Bla}$.  The $G$-orbits $\SO_{\Bla}^-, \SO^+_{\Bla}  \subset X_{\Bla}$ 
are determined, independently 
of the choice of the Jordan basis $\{ u_{j,k}\}$. 
\par
Now take $\Bm \in \SQ_{n,r}$, and let 
$M_{p_i}^+$ be the subspace of $M_{p_i}$ isomorphic to $\ol M_{p_i}$ defined in 
2.7.  Recall the set $\SX^+_{\Bm}$ in (2.7.1).  We define a set $\SX^+_{\Bm \unip}$ by 

\begin{equation*}
\SX^+_{\Bm \unip} = \bigcup_{g \in G}g(U \times \prod_i M^+_{p_i}),
\end{equation*}
which coincides with $\SX\uni \cap \SX^+_{\Bm}$.
Note that for each $\Bla \in \SP(\Bm)$, 
$\SO_{\Bla}^+ \subset X_{\Bla} \cap \SX^+_{\Bm, \unip}$. 
\par
In general, $X_{\Bla} \cap \SX^+_{\Bm}$ consists of infinitely many $G$-orbits 
even if $\Bla \in \SP(\Bm)$.  
The following special case would be worth mentioning.

\begin{lem}  %%%%  Lemma 4.8
Assume that $\Bm \in \SQ_{n,r}$ is of the form 
$m_j = 0$ for $j \ne i_0, r$ for some $i_0$ $($possibly 
$i_0 = r$$)$.  Then $X_{\Bla} \cap \SX^+_{\Bm,\unip} = \SO_{\Bla}^+$ for any 
$\Bla \in \SP_{n,r}$ if it is non-empty. 
\end{lem}

\begin{proof}
Assume that $1 \le i_0 \le r-1$ and that $X_{\Bla} \cap \SX^+_{\Bm} \ne \emptyset$. 
Take $(x, \Bv) \in X_{\Bla} \cap \SX^+_{\Bm}$. 
Since $(x, \Bv) \in \SX^+_{\Bm}$, we must have $v_j = 0$ for $j \ne i_0$. 
Then the description of $X_{\Bla}$ in 4.1 implies that $\la^{(i)} = \emptyset$ 
for $i \ne i_0$. It follows that $\Bla \in \SP(\Bm')$ with $\Bm' = (m_1', \dots, m'_r)$
such that $m'_{i_0} \le m_{i_0}, m'_j = 0$ for $j \ne i_0, r$.  But this is essentially 
the same as the case of $r = 2$, hence $X_{\Bla} \cap \SX^+_{\Bm}$ coincides with $\SO_{\Bla}^+$. 
The case where $i_0 = r$ is dealt with similarly.  
\end{proof}
\par\medskip
Note that irreducible representations of $S_{\Bm}$ are, up to isomorphism, 
parametrized by $\SP(\Bm)$ (see 3.10).   We denote by $V_{\Bla}$ the irreducible 
representation of $S_{\Bm}$ corresponding to $\Bla \in \SP(\Bm)$.  
The following result gives the Springer correspondence between 
$\SX_{\Bm, \unip}$ and $S_{\Bm}$.

\begin{thm}[{[S3, Theorem 8.13]}]  %%%%  Theorem 4.8.
For any $\Bm \in \SQ_{n,r}$, put $d'_{\Bm} = \dim \SX_{\Bm,\unip}$. 
Then $(\pi_{\Bm,1})_!\Ql[d'_{\Bm}]$ is a semisimple perverse sheaf on $\SX_{\Bm,\unip}$, 
equipped 
with the action of $S_{\Bm}$, and is decomposed as 

\begin{equation*}
(\pi_{\Bm,1})_!\Ql[d'_{\Bm}] \simeq \bigoplus_{\Bla \in \SP(\Bm)}
                  V_{\Bla} \otimes \IC(\ol X_{\Bla}, \Ql)[\dim X_{\Bla}].
\end{equation*}
\end{thm}

\para{4.9.}
For each $z = (x, \Bv) \in \SX_{\Bm}$, we consider the Springer fibre 
$(\pi_{\Bm})\iv(z) \simeq \SB^{(\Bm)}_z$, where $\SB^{(\Bm)}_z$ is a closed 
subvariety of $\SB = G/B$ defined as 

\begin{equation*}
\SB^{(\Bm)}_z = \{ gB \in \SB \mid g\iv xg \in B, g\iv \Bv \in \prod_i M_{p_i} \}.
\end{equation*}
For each $\Bla \in \SP(\Bm)$, put $d_{\Bla} = (\dim \SX_{\Bm \unip} - \dim X_{\Bla})/2$.
One can check that $d_{\Bla} = n(\Bla)$ (see 4.1 for the definition of $n(\Bla)$).  
We have

\begin{lem}  %%%%  Lemma 4.10.
Assume that $\Bla \in \SP(\Bm)$.  Then for any $z \in X_{\Bla}$, we have
$\dim \SB^{(\Bm)}_z = d_{\Bla}$. 
\end{lem}

\begin{proof}
By [S3, Lemma 8.5], we have $\dim \SB^{(\Bm)}_z \ge d_{\Bla}$. 
On the other hand, clearly we have $\SB^{(\Bm)}_z \subset \SB_x$, where 
$\SB_x = \{ gB \in \SB \mid g\iv xg \in B\}$. 
Here the Jordan type of $x$ is given by $\nu = \la^{(1)} + \cdots + \la^{(r)}$.
By the classical result for the case of $GL_n$, we know that $\dim \SB_x = n(\nu)$.
Since $n(\nu) = n(\Bla)$, we have $\dim \SB^{(\Bm)}_z \le d_{\Bla}$. 
The lemma is proved.
\end{proof}

\par\bigskip\bigskip
\section{Kostka functions}

\para{5.1.}
Kostka functions associated to complex reflection groups were
introduced in [S1], [S2] as a generalization of Kostka polynomials.
In this section, we discuss the relationship between our Green 
functions and those Kostka functions.  
In [S1], [S2], Kostka functions $K^{\pm}_{\Bla, \Bmu}(t)$ indexed by  
$\Bla, \Bmu \in \SP_{n,r}$ (and depending on the sign $+$, $-$) 
were introduced as coefficients of the transition 
matrix between the basis of Schur functions and those of Hall-Littlewood 
functions, as in the case of original Kostka polynomials.   
A-priori they are rational functions in  $\BQ(t)$.   
We define a modified Kostka function $\wt K^{\pm}_{\Bla, \Bmu}(t)$ by 
\begin{equation*}
\tag{5.1.1}
\wt K^{\pm}_{\Bla, \Bmu}(t) = t^{a(\Bmu)}K^{\pm}_{\Bla, \Bmu}(t\iv),
\end{equation*}
where the $a$-function $a(\Bla)$ is defined by 
\begin{equation*}
\tag{5.1.2}
a(\Bla) = r\cdot n(\Bla) + |\la^{(2)}| + 2|\la^{(3)}| + \cdots + (r-1)|\la^{(r)}|
\end{equation*} 
for $\Bla \in \SP_{n,r}$. 
Note that in the case where $r = 1$, $a(\Bla)$ coincides with $n(\la^{(1)})$, 
and in the case where $r = 2$, $a(\Bla)$ coincides with the $a$-function on $\SP_{n,2}$
used in [AH] and [SS2] (in [AH], the notation  $b(\mu;\nu)$ 
is used instead of $a(\Bla)$ for $\Bla = (\mu;\nu)$).
\par
Following [S1], we give a combinatorial characterization of 
modified Kostka functions $\wt K^{\pm}_{\Bla, \Bmu}(t)$. 
Let $W_{n,r}$ be the complex reflection group $S_n\ltimes (\BZ/r\BZ)^n$.
For a (not necessarily
irreducible)  character $\x$ of $W_{n,r}$, we define the fake degree 
$R(\x)$ by 
\begin{equation*}
\tag{5.1.3}
R(\x) = \frac{\prod_{i=1}^n(t^{ir}-1)}{|W_{n,r}|}\sum_{w \in W_{n,r}}
    \frac{\det_{\BV}(w)\x(w)}{\det_{\BV}(t - w)},
\end{equation*}
where $\BV$ is a representation space (over $\Ql$) of the reflection representation of
$W_{n,r}$, and $\det_{\BV}$ means the determinant on $\BV$. 
Here we have $|W_{n,r}| = n!r^n$.  Note that $R(\x) \in \BZ_{\ge 0}[t]$;  
if $\x$ is irreducible, 
$R(\x)$ is given as the graded
multiplicity of $\x$ in the coinvariant algebra $R(W_{n,r})$ of $W_{n,r}$.
Let $N^*$ be the number of reflections of $W_{n,r}$, which is given
as the maximum degree of $R(W_{n,r})$, and is explicitly given as 
\begin{equation*}
\tag{5.1.4}
N^* = \frac{rn(n+1)}{2} - n = \binom{n}{2}r + (r-1)n.
\end{equation*}
\par
It is known that irreducible characters of $W_{n,r}$ are
parametrized by $\SP_{n,r}$.  We denote by $\r^{\Bla}$ the irreducible
character of $W_{n,r}$ corresponding to $\Bla \in \SP_{n,r}$. 
(For example, $\r^{\Bla}$ is the trivial character for 
$\Bla = (n;-;\cdots;-)$. See 5.6 for details). 
For $\Bla, \Bmu \in \SP_{n,r}$, we define a square matrix 
$\Om = (\w_{\Bla,\Bmu})_{\Bla,\Bmu \in \SP_{n,r}}$ by 
\begin{equation*}
\w_{\Bla,\Bmu} = t^{N^*}R(\r^{\Bla}\otimes\ol{\r^{\Bmu}}\otimes\ol\det_V),
\end{equation*}
where $\ol \x$ denotes the complex conjugate of the character $\x$ 
(in fact, the function $\ol\x$ is defined by $\ol\x(w) = \x(w\iv)$ for 
$w \in W_{n,r}$).
Here we fix a total order $\Bla \preceq \Bmu$ 
on $\SP_{n,r}$ compatible with the partial
order $\Bla \le \Bmu$ defined in 4.1, and consider the square matrix with respect to 
this total order.
Note that $\w_{\Bla,\Bmu} \in \BZ_{\ge 0}[t]$.
We have the following result.

\begin{thm}[{[S1, Theorem 5.4]}]  %%% Theorem 6.3
There exist unique matrices 
$P^{\pm} = (p^{\pm}_{\Bla,\Bmu}), \vL = (\xi_{\Bla,\Bmu})$ 
over $\BQ(t)$ satisfying the equation 
\begin{equation*}
\tag{5.2.1}
P^-\vL\, {}^t\!P^+ = \Om
\end{equation*}
subject to the conditions that $\vL$ is a diagonal matrix and that
\begin{equation*}
p^{\pm}_{\Bla,\Bmu} = \begin{cases}
                    0  &\quad\text{ unless } \Bmu \preceq \Bla, \\
                    t^{a(\Bla)}  &\quad\text{ if } \Bla = \Bmu.
                 \end{cases}
\end{equation*}
Then the entry $p^{\pm}_{\Bla,\Bmu}$ of the matrix $P^{\pm}$ coincides
 with $\wt K^{\pm}_{\Bla,\Bmu}(t)$.
\end{thm}

\remarks{5.3.}
(i) \ 
Since $\Om$ is non-symmetric unless $r = 1$ or 2, $P^+$ does not
coincide with $P^-$ if $r \ge 3$.  
\par
(ii) \ Our construction of Kostka functions $K_{\Bla,\Bmu}(t)$ depends 
on the choice of the total order $\preceq$ on $\SP_{n,r}$.  In the case where
$r = 1$ or 2, it is known that Kostka functions are independent of the
choice of the total order whenever it is compatible with the partial
order $\le$ (see [M] for $r = 1$, and [S2], [AH], [SS2] for $r = 2$).  

\para{5.4.}
We define an involution $\t$ on $\SP_{n,r}$ by 
\begin{equation*}
\tag{5.4.1}
\t: (\la^{(1)}, \dots, \la^{(r)}) \mapsto 
(\la^{(r-1)}, \la^{(r-2)},\dots, \la^{(1)}, \la^{(r)})
\end{equation*}
for $\Bla = (\la^{(1)}, \dots, \la^{(r)}) \in \SP_{n,r}$. 
We shall prove the following result.

%%%%
%%%%
\begin{thm}  %%%% Theorem 5.5
For any $\Bla \in \SP(\Bm), \Bmu \in \SP(\Bm')$, we have
\begin{equation*}
\tag{5.5.1}
\begin{split}
q^{-a(\Bla) - a(\tau(\Bmu))}&\w_{\Bla,\Bmu}(q) =  \\ 
&(-1)^{p_-(\Bm) + p_+(\Bm')}q^{-r(n(\Bla) + n(\Bmu))}\sum_{z \in \SX\uni^{F^r}}
           Q^-_{\Bla}(z)Q^+_{\Bmu}(z).
\end{split}
\end{equation*}
\end{thm}

\para{5.6.}
The right hand side of (5.5.1) can be computed by 
applying Proposition 3.11 to the case where $\ve = -, \ve' = +$.
Hence we will compute the left hand side of (5.5.1), i.e., $\w_{\Bla,\Bmu}(t)$
for the indeterminate $t$.
In order to compute them, we
recall some known facts on the characters of $W_{n,r}$.
Let $\d : W_{n,r} \to \Ql^*$ be the linear character of $W_{n,r}$
defined by $\d|_{S_n} = 1_{S_n}$ and $\d(s_0) = \z$, where 
$s_0$ is a (simple) reflection of order $r$, and $\z$ is a primitive $r$-th
root of unity in $\Ql$.  Assume given 
$\Bla = (\la^{(1)}, \dots, \la^{(r)})\in \CP(\Bm)$.
Then 
the irreducible character $\r^{\Bla}$ is constructed by 
\begin{equation*}
\tag{5.6.1}
\r^{\Bla} = \Ind_{W_{\Bm,r}}^{W_{n,r}}
(\x^{\la^{(1)}}\boxtimes \d\x^{\la^{(2)}}
     \boxtimes\cdots\boxtimes \d^{r-1}\x^{\la^{(r)}}), 
\end{equation*}
where $W_{\Bm, r} = W_{m_1,r}\times \cdots \times W_{m_r,r}$,  and  
$\x^{\la^{(i)}}$ is the irreducible character of $S_{m_i}$
corresponding to the partition $\la^{(i)}$, which is regarded as 
a character of $W_{m_i,r}$ by the natural surjection 
$W_{m_i,r} \to S_{m_i}$, and $\d^{i-1}\x^{\la^{(i)}}$ is the irreducible 
character of $W_{m_i,r}$ obtained by multiplying $\d^{i-1}$ on it.
(Here we use the same symbol $\d$ to denote the corresponding linear 
character of $W_{m_i,r}$ for each $i$.) 
Let $\ve$ be the irreducible character of $W_{n,r}$ obtained by the 
pull-back of the sign character of $S_n$ under the map $W_{n,r} \to S_n$.
(Do not confuse the character $\ve$ of $W_{n,r}$  with the sigantures 
$\{ \ve, \ve'\} = \{ -,+\}$.  In the present discussion, the signatures are
fixed.)   
Then we have $\det_{\BV} = \ve\d$.
For $i = 1, \dots, r$, put $\Bla_i = (\la^{(1)}, \dots, \la^{(r)})$ with 
$\la^{(j)} = \emptyset$ for  $j \ne i$ and $\la^{(i)} = (n)$, and 
put $\Bmu_i = (\mu^{(1)}, \dots, \mu^{(r)})$ with 
$\mu^{(j)} = \emptyset$ for $j \ne i$ and $\mu^{(i)} = (1^n)$.
Then the following fact is easily verified from (5.6.1).  
\begin{equation*}
\tag{5.6.2}
\r^{\Bla_i} = \d^{i-1}, \qquad \r^{\Bmu_i} = \ve\d^{i-1}.  
\end{equation*}
In particular, we have
\begin{equation*}
\tag{5.6.3}
1_{W _{n,r}}= \r^{\Bla_1}, \qquad \d = \r^{\Bla_2}, \qquad 
  {\det}_{\BV} = \r^{\Bmu_2},  \qquad \ol\det_{\BV} = \r^{\Bmu_r}= \ve\d^{r-1}.
\end{equation*}
Although the following facts (5.6.4) $\sim$ (5.6.6) are not used 
in the discussion below, 
we write them for the reference. 
\begin{align*}
\tag{5.6.4}
\ol{\r^{\Bla}} &= \r^{\Bla'} \quad\text{ for } 
   \Bla' = (\la^{(1)}, \la^{(r)}, \la^{(r-1)}, \dots, \la^{(2)}),   \\
\tag{5.6.5}
\r^{{}^t\Bla} &= \ve \r^{\Bla},
\end{align*}
where ${}^t\Bla = ({}^t\la^{(1)}, \dots, {}^t\la^{(r)})$. 
It follows from (5.6.5) that
\begin{equation*}
\tag{5.6.6}
\w_{\Bla,\Bmu} = \w_{{}^t\Bla, {}^t\Bmu}
\end{equation*}
for any $\Bla, \Bmu \in \SP_{n,r}$.
\par
For a given $\Bm \in \SQ_{,r}$, we consider $W_{\Bm, r} \simeq S_{\Bm} \ltimes (\BZ/r\BZ)^n$.
As in 3.10, we denote by $\x^{\Bla}$ the irreducible character of 
$S_{\Bm}$ corresponding to $\Bla \in \SP(\Bm)$.  We define a linear character 
$\d_{\Bm}$ of $W_{\Bm,r} = W_{m_1,r} \times \cdots \times W_{m_r,r}$
by 
\begin{equation*}
\d_{\Bm} =  \d^0\boxtimes \d^1 \boxtimes \cdots \boxtimes \d^{r-1}.
 \end{equation*}
Let $\wt\x^{\Bla}$ be the irreducible character $\d_{\Bm}\x^{\Bla}$
on $W_{\Bm,r}$.  Then (5.6.1) can be rewritten as 
\begin{equation*}
\tag{5.6.7}
\r^{\Bla} = \Ind_{W_{\Bm,r}}^{W_{n,r}}\wt\x^{\Bla}.
\end{equation*}

\para{5.7.}
We now compute $R(\r^{\Bla}\otimes \ol\r^{\Bmu}\otimes \ol\det_{\BV})$
for $\Bla \in \SP(\Bm),\Bmu \in \SP(\Bm')$.
In the computation below, we write $W_{n,r}, W_{\Bm,r}$, etc.  simply 
as $W_n, W_{\Bm}$, etc. by omitting $r$. 
\begin{align*}
\tag{5.7.1}
R(\r^{\Bla}\otimes \ol{\r^{\Bmu}}\otimes \ol\det_{\BV})
&= \frac{\prod_{k=1}^n(t^{kr} -1)}{|W_n|}
  \sum_{w \in W_n}\frac{\bigl(\Ind_{W_{\Bm}}^{W_n}\wt\x^{\Bla}\bigr)(w)
     \ol{\bigl(\Ind_{W_{\Bm'}}^{W_n}\wt\x^{\Bmu}\bigr)(w)}}
             {{\det}_{\BV}(t - w)} \\
&= \frac{\prod_{k=1}^n(t^{kr}-1)}{|W_n||W_{\Bm}||W_{\Bm'}|}
     \sum_{\substack{
              w, w_1, w_2 \in W_n \\
               w_1\iv ww_1 \in W_{\Bm} \\
               w_2\iv ww_2 \in W_{\Bm'}}} 
\frac{\wt\x^{\Bla}(w_1\iv ww_1)\ol{\wt\x^{\Bmu}(w_2\iv ww_2)}}
                      {\det_{\BV}(t -w)}  \\
&= \frac{\prod_{k=1}^n(t^{kr}-1)}{|W_{\Bm}||W_{\Bm'}|}
      \sum_{\substack{
             x \in W_n \\ y \in W_{\Bm} \cap x W_{\Bm'}x\iv  }}
       \frac{\wt\x^{\Bla}(y)\ol{\wt\x^{\Bmu}(x\iv yx)}}{\det_{\BV}(t - y)},  
\end{align*}
where the last formula follows from the change of variables 
$x = w_1\iv w_2, y = w_1\iv ww_1$.
\par
We consider the set of double cosets 
$W_{\Bm}\backslash W_n/W_{\Bm'}$. This set is described by a certain 
set of matrices as given in Section 5 in [AH].  
We define $\SM_{\Bm, \Bm'}$ as the set of degree $r$ matrices
$(h_{ij})$ with entires in 
$\BZ_{\ge 0}$ satisfying the following properties;

\begin{equation*}
\tag{5.7.2}
\sum_{i=1}^{r} h_{ij} = m_j  \text{ for all } j, \quad   
\sum_{j=1}^{r}h_{ij}  = m'_i \text{  for all } i.
\end{equation*}
Then there exists a bijective correspondence
\begin{equation*}
\SM_{\Bm, \Bm'} \simeq S_{\Bm} \backslash S_n /S_{\Bm'}
\simeq W_{\Bm} \backslash W_n /W_{\Bm'} 
\end{equation*}
satisfying the properties
\begin{equation*}
S_{\Bm} \cap xS_{\Bm'}x\iv \simeq \prod_{i,j}S_{h_{ij}}
\quad \text{ and } \quad 
W_{\Bm} \cap x W_{\Bm'}x\iv \simeq \prod_{i,j} W_{h_{ij}}
\end{equation*}
if $x$ is contained in the orbit 
$\SO \in S_{\Bm} \backslash S_n /S_{\Bm'}$ corresponding to
$(h_{ij}) \in \SM_{\Bm, \Bm'}$.
Note that if $x \in S_n$, we have
$W_{\Bm} \cap x W_{\Bm'}x\iv \simeq 
     (S_{\Bm} \cap xS_{\Bm'}x\iv) \ltimes (\BZ/r\BZ)^n$. 

\par
By applying the above expression, we see that 
\begin{equation*}
\frac{1}{|W_{\Bm} \cap x W_{\Bm'}x\iv|}
    \sum_{y \in W_{\Bm} \cap x W_{\Bm'}x\iv}
         \frac{\wt\x^{\Bla}(y)\ol{\wt\x^{\Bmu}(x\iv yx)}}
                  {\det_{\BV}(t - y)}  
= \frac{R(\wt\x_0^{\Bla}\otimes \ol{\wt\x_0^{\Bmu}}\otimes \ol\det_{\BV})}
      {\prod_{i,j}\prod_{k=1}^{h_{ij}}(t^{kr} - 1)}
\end{equation*}
for $x \in \SO$ corresponding to $(h_{ij}) \in \SM_{\Bm,\Bm'}$, where 
$\wt\x^{\Bla}_0$ (resp. $\wt\x^{\Bmu}_0$) is the restriction of 
$\wt\x^{\Bla}$ (resp. ${}^x (\wt\x^{\Bmu})$) on 
$W_{\Bm} \cap xW_{\Bm'}x\iv$.
Substituting this into the last formula of (5.7.1), we have
\begin{equation*}
\tag{5.7.3}
R(\r^{\Bla}\otimes\ol{\r^{\Bmu}}\otimes\ol\det_{\BV})
= \sum_{\SO \in W_{\Bm}\backslash W_n/W_{\Bm'}}
\frac{\prod_{k=1}^n(t^{kr} -1)}
       {\prod_{i,j}\prod_{k=1}^{h_{ij}}(t^{kr} - 1)} \\
              R(\wt\x^{\Bla}_0\otimes \ol{\wt\x^{\Bmu}_0}\otimes
                   \ol\det_{\BV})
\end{equation*}
since $|W_{\Bm}||W_{\Bm'}|/|W_{\Bm} \cap x W_{\Bm'}x\iv| = |\SO|$
for $\SO = W_{\Bm}xW_{\Bm'}$.

\para{5.8.}
Our next aim is to compute 
$R(\wt\x^{\Bla}_0\otimes\ol{\wt\x^{\Bmu}_0}\otimes\ol\det_{\BV})$.
For each $\SO \subset W_n$, we choose a representative 
$x \in \SO$ such that $x \in S_n$.  Then under the isomorphism
$W_{\Bm} \cap x W_{\Bm'}x\iv \simeq 
    (S_{\Bm} \cap xS_{\Bm'}x\iv) \ltimes (\BZ/r\BZ)^n$, we have 
\begin{align*}
\wt\x^{\Bla}_0 &= \wt\x^{\Bla}|_{W_{\Bm} \cap x W_{\Bm'}x\iv}
           = \x^{\Bla}|_{S_{\Bm} \cap xS_{\Bm'}x\iv}\otimes \d_{\Bm}, \\ 
\wt\x^{\Bmu}_0 &= {}^x(\wt\x^{\Bmu})|_{W_{\Bm} \cap x W_{\Bm'}x\iv}
           = {}^x(\x^{\Bmu})|_{S_{\Bm} \cap x S_{\Bm'}x\iv}\otimes 
                                 {}^x\d_{\Bm'},
\end{align*}
where  $\d_{\Bm}, {}^x\d_{\Bm'}$ are the restriction  
of the corresponding linear characters of $W_n$ to 
$W_{\Bm} \cap x W_{\Bm'}x\iv$.  It follows from this that
\begin{equation*}
\wt\x^{\Bla}_0 \otimes \ol{\wt\x^{\Bmu}_0} \otimes \ol\det_{\BV}
  = (\x^{\Bla}\otimes{}^x(\x^{\Bmu})\otimes\ve)|_{S_{\Bm} 
             \cap xS_{\Bm'}x\iv} \otimes \d_1\ve, 
\end{equation*}
where $\d_1 = \d_{\Bm} \otimes \ol{{}^x\d_{\Bm'}} \otimes \ol\det_{\BV}$, and  
$\ve$ is as in 5.6.  
Note that $\ol\det_{\BV} = \ve\d^{r-1}$ by (5.6.3).  Hence we have 
$\d_1\ve  = \d_{\Bm} \otimes \ol{{}^x\d_{\Bm'}} \otimes \d^{r-1}$, which 
is a linear character of $(\BZ/r\BZ)^n$ extended to that of 
$W_{\Bm} \cap x W_{\Bm'}x\iv$ by the trivial action of 
$S_{\Bm} \cap x S_{\Bm'}x\iv$.  
We claim that
\begin{equation*}
\tag{5.8.1}
R(\wt\x^{\Bla}_0\otimes\ol{\wt\x_0^{\Bmu}}\otimes\ol\det_{\BV})
  = R(\x^{\Bla}\otimes{}^x(\x^{\Bmu})\otimes\ve)|_{t \to t^r}R(\d_1\ve),
\end{equation*}
where we denote by $f|_{t \to t^r}$
the polynomial $f(t^r)$ obtained from  $f(t) \in \BZ[t]$.
\par
We show (5.8.1).
Let $R = R(W_n)$ be the coinvariant algebra of $W_n$.  Then we have
\begin{equation*}
R \simeq \Ql[x_1, \dots, x_n]/I_+(W_n),
\end{equation*}
where $I_+(W_n)$  is the ideal of $\Ql[x_1, \dots, x_n]$ generated by 
$e_1(x^r), \dots, e_n(x^r)$.  Here $e_i(x) = e_i(x_1, \dots, x_n)$
is the $i$th elementary symmetric polynomial with variables $x_1, \dots, x_n$, 
and $e_i(x^r)$ denotes such a polynomial with variables $x_1^r, \dots, x_n^r$.
Note that $(\BZ/r\BZ)^n$ acts on $R$ via $t_i\cdot x_i = \z x_i$ and 
$t_i\cdot x_j = x_j$ if $j \ne i$ for
a generator $t_i$ of the $i$th factor cyclic group $\BZ/r\BZ$.
Let $\p_{\Ba}$ be a linear character of $(\BZ/r\BZ)^n$ defined by
$\p_{\Ba}(t_i) = \z^{a_i}$ for $\Ba = (a_1, \dots, a_n) \in \BZ^n$.
  Then the $\p_{\Ba}$-isotypic 
subspace of $R$ is given by 
\begin{equation*}
x_1^{a_1}\cdots x_n^{a_n}\Ql[x_1^r, \dots, x_n^r]/
        (I_+(W_n)\cap \Ql[x_1^r, \dots, x_n^r])
\end{equation*}
with $0 \le a_i <r$. If $\p_{\Ba}$ can be lifted to the linear character
of $W_n$, then $R(\p_{\Ba}) = \prod_i t^{a_i}$.
We have
\begin{equation*}
\Ql[x_1^r, \dots x_n^r]/(I_+(W_n) \cap \Ql[x_1^r, \dots, x_n^r])
   \simeq \Ql[x_1^r, \dots, x_n^r]/I_+, 
\end{equation*}
where $I_+$ is the ideal of $\Ql[x_1^r, \dots, x_n^r]$ generated by $e_1(x^r), \dots, e_n(x^r)$, 
hence it is isomorphic to the coinvariant algebra of $S_n$ 
with respect to the variables $x_1^r, \dots, x_n^r$.
Now suppose that $S_n$ stabilizes $\p_{\Ba}$.  Then for an 
irreducible character $\x$ of $S_n$, $\r = \x\otimes \p_{\Ba}$
gives rise to an irreducible character of $W_n$. Moreover any irreducible
submodule in $R$ affording $\r$ is contained  in the $\p_{\Ba}$-isotypic
subspace $R^{\p_{\Ba}}$ of $R$, and the graded multiplicity of $\r$ 
is determined by the graded multiplicity of $\x$ in the graded
$S_n$-module $R^{\p_{\Ba}}$.  Hence in this case we have
$R(\r) = R(\x)|_{t\to t^r}R(\p_{\Ba})$.  A similar argument works 
also for the group 
$W_{\Bm} \cap x W_{\Bm'}x\iv \simeq \prod_{i,j}W_{h_{ij}}$.
Thus (5.8.1) holds.

\par
The definition of the fake degree in (5.1.3) makes sense for the 
case of symmetric groups, and we have
\begin{equation*}
R(\x^{\Bla}\otimes{}^x(\x^{\Bmu})\otimes\ve)|_{t \to t^r}
  = \frac{\prod_{i,j}\prod_{k=1}^{h_{ij}}(t^{kr}-1)}
       {|S_{\Bm} \cap x S_{\Bm'}x\iv|}
            \sum_{y \in S_{\Bm}\cap xS_{\Bm'}x\iv}
                \frac{\x^{\Bla}(y)\x^{\Bmu}(x\iv yx)}
                          {\det_{\BV}(t^r - y)}.
\end{equation*}
On the other hand, $R(\d_1\ve)$ is determined by the double coset 
$\SO \in S_{\Bm}\backslash S_n / S_{\Bm'}$, 
so one can write it as $R(\d_1\ve) = t^{A_{\SO}}$ for some 
integer $A_{\SO}$.  Substituting these data into (5.8.1), the formula
(5.7.3) turns out to be
\begin{equation*}
\tag{5.8.2}
\begin{split}
R(\r^{\Bla}\otimes\ol{\r^{\Bmu}}\otimes\ol\det_{\BV})
= \frac{\prod_{k=1}^n(t^{kr}-1)}{|S_{\Bm}||S_{\Bm'}|} 
     \sum_{\SO \in S_{\Bm}\backslash S_n/S_{\Bm'}}t^{A_{\SO}}
            \sum_{\substack{
          x \in \SO \\ y \in S_{\Bm} \cap xS_{\Bm'}x\iv }}  
     \frac{\x^{\Bla}(y)\x^{\Bmu}(x\iv yx)}{\det_{\BV}(t^r - y)},
\end{split}
\end{equation*}
\par

Next we compute the value $A_{\SO}$ for each $\SO$.
For a matrix $(h_{ij}) \in \SM_{\Bm, \Bm'}$, we introduce the 
notation $h_{\le i, \le j} = \sum_{k \le i}\sum_{l\le j} h_{{k,l}}$, 
$h_{i,\le j} = \sum_{k \le j} h_{i,k}$, etc..
Assume that $\SO$ corresponds to $(h_{ij}) \in \SM_{\Bm,\Bm'}$.
We define an integer $B_{\SO}(\Bla, \Bmu)$ by 
\begin{equation*}
B_{\SO}(\Bla,\Bmu) = \binom{n}{2} - n(\Bla) - n(\Bmu) 
           + \sum_{i=1}^{r-1}h_{i, \le i}.
\end{equation*}
We have the following lemma.

%%%
%%%
\begin{lem}  %%% Lemma 5.9.
Under the above notation, we have
\begin{equation*}
N^* - a(\Bla) - a(\tau(\Bmu)) + A_{\SO} = r\cdot B_{\SO}(\Bla,\Bmu).
\end{equation*}
\end{lem}

\para{5.10.}
Assuming the lemma, we continue the proof.
By (5.8.2) together with Lemma 5.9, we have
\begin{equation*}
\tag{5.10.1}
\begin{split}
t^{-a(\Bla) - a(\tau(\Bmu))}\w_{\Bla,\Bmu}(t)
  &= \frac{\prod_{k=1}^n(t^{kr}-1)}{|S_{\Bm}||S_{\Bm'}|}  \\
     &\times \sum_{\SO \in S_{\Bm}\backslash S_n/S_{\Bm'}}
         t^{r\cdot B_{\SO}(\Bla,\Bmu)}
          \sum_{\substack{
                 x \in \SO \\ y \in S_{\Bm}\cap xS_{\Bm'}s\iv }}
               \frac{\x^{\Bla}(y)\x^{\Bmu}(x\iv yx)}{\det_{\BV}(t^r - y)}.       
\end{split}
\end{equation*}

Here we note that the set of double cosets $P_{\Bm}\backslash G /P_{\Bm'}$
is also in bijective correspondence with $\SM_{\Bm, \Bm'}$ in such a way
that if $\wh\SO \in P_{\Bm}\backslash G /P_{\Bm'}$ corresponds to
$(h_{ij}) \in \SM_{\Bm, \Bm'}$, then we have 
\begin{equation*}
\dim (M_{p_j} \cap g(M'_{p'_i})) = h_{\le i, \le j} \text{ for } g \in \wh\SO 
\text{ and for any } i,j.
\end{equation*}
We have a decomposition $M'_{p'_i} = M_{p'_1}'^+\oplus\cdots \oplus M'^+_{p'_i}$,
and the maximal torus $T = T_0$ is contained in $L_{\Bm}$ and 
$L_{\Bm'}$.  Recall that $\{ e_1, \dots, e_n\}$ gives a basis of $V$ 
consisting of weight vectors of $T$, and bases of $M_{p_j}, M'^+_{p'_i}$ are given 
by subsets of this basis.    
For each $\wh\SO \in P_{\Bm}\backslash G /P_{\Bm'}$, 
one can choose a representative $n_{\SO} \in \wh\SO$ such that 
$n_{\SO} \in N_G(T)$, hence $n_{\SO}$ permutes the basis $\{ e_1, \dots, e_n\}$
up to scalar.   
It follows that 
\begin{equation*}
\dim (M_{p_i} \cap n_{\SO}(M_{p'_i}'^+)) = h_{\le i, \le i} - h_{\le i-1, \le i}
                                = h_{i, \le i},
\end{equation*}
and we see that 
\begin{equation*}
\tag{5.10.2}
a_{-,+}(\Bm, \Bm'; n_{\SO}) = \sum_{i=1}^{r-1} h_{i, \le i}.
\end{equation*}
We now compare Proposition 3.11 with (5.10.1). Since 
$q^{r\binom{n}{2}}\prod_{k=1}^n(q^{kr} - 1) = |G^{F^r}|$, 
and $|\det_{\BV}(q^r - y)| = |T_y^{F^r}|$, we have,  by (5.10.2), that 
\begin{equation*}
(-1)^{p_-(\Bm) + p_+(\Bm')}q^{-r(n(\Bla) + n(\Bmu))}\sum_{z \in \SX_{\Bm, \unip}^{F^r}}
                 Q_{\Bla}^-(z)Q_{\Bmu}^+(z) 
                    = q^{-a(\Bla) - a(\tau(\Bmu))}\w_{\Bla, \Bmu}(q).
\end{equation*}
This proves the theorem modulo Lemma 5.9.

\para{5.11.}
We shall prove Lemma 5.9.
Note that the isomorphism 
$W_{\Bm} \cap x W_{\Bm'}x\iv  \isom $ $\prod_{i,j} W_{h_{ij}}$
is chosen so that it satisfies the following properties; 
the factor of $\d_{\Bm}$ corresponding to $W_{h_{ij}}$ is equal to
$\d_0^{j-1}$, and the factor of ${}^x\d_{\Bm'}$
 corresponding to $W_{h_{ij}}$ is equal to $\d_0^{i-1}$, where
we denote by $\d_0$ the linear character of $W_{h_{ij}}$ 
corresponding to $\d$ given in 5.6. 
It follows that the factor of 
$\d_1\ve = \d_{\Bm}\otimes\ol{{}^x\d_{\Bm'}}\otimes\d^{r-1}$
corresponding to $W_{h_{ij}}$ is given by
\begin{equation*}
\tag{5.11.1}
\d_0^{(j-1) + (r-i+1) + (r-1)} 
   = \d_0^{(j-1) + (r-i)}.
\end{equation*}
Since $\d_0 = \p_{\Ba}$ with $\Ba = (1, \dots, 1)$ in the notation of 
5.8, we see that
\begin{equation*}
\tag{5.11.2}
R(\d_0^k) = q^{kh_{ij}} 
\end{equation*}
for $k = 0, \dots, r-1$.  It follows that 
\begin{equation*}
A_{\SO} = \sum_{1 \le i,j \le r}[j - i -1]h_{ij},
\end{equation*}
where $[j-i-1]$ is an integer between 0 and 
$r-1$ which is congruent to $(j-1) + (r-i)$ modulo $r$.
On the other hand, by (5.1.2), (5.1.4) and  (5.4.1), 
we have
\begin{equation*}
N^* - a(\Bla) - a(\t(\Bmu)) 
  = r\bigl\{\binom{n}{2} - n(\Bla) - n(\Bmu)\bigr\}  + C, 
\end{equation*}
where 
\begin{equation*}
\begin{split}
C = (r-1)n  &- \bigl\{ |\la^{(2)}| + 2|\la^{(3)}| + \cdots + 
                    (r-1)|\la^{(r)}|\bigr\}  \\ 
            &- \bigl\{ |\mu^{(r-2)}| + 2|\mu^{(r-3)}| + \cdots + 
                     (r-2)|\mu^{(1)}| + (r-1)|\mu^{(r)}| \bigr\}.
\end{split}
\end{equation*}
Hence in order to prove the lemma, it is enough to show that 
\begin{equation*}
\tag{5.11.3}
C + \sum_{1 \le i, j \le r}[j - i -1]h_{ij} 
  = r\sum_{i=1}^{r-1}h_{i,\le i}.
\end{equation*}
We show (5.11.3).
Since 
$n = \sum_{1 \le i,j \le r} h_{ij}$, we have 
\begin{equation*}
\tag{5.11.4}
\begin{split}
(r-1)n &+ \sum_{1 \le i, j \le r}[j - i -1]h_{ij}  \\
 &= \sum_{ 1 \le i,j \le r}\bigl\{[j - i -1] + (r-1)\bigr\}h_{ij}.
\end{split}
\end{equation*}
On the other hand, since $|\la^{(j)}| = m_j$ and 
$|\mu^{(i)}| = m_i'$, by (5.7.2), we have
\begin{align*}
(j-1)|\la^{(j)}| &= \sum_{1 \le i \le r}(j-1)h_{ij}
    \quad\text{ for } j = 1, 2, \dots, r, \\
(r-1-i)|\mu^{(i)}| &= \sum_{1 \le j \le r}(r-1-i)h_{ij}
     \quad\text{ for } i = 0, 1, \dots, r-1, 
\end{align*}
where we understand that $\mu^{(0)} = \mu^{(r)}$ and 
$h_{0j} = h_{rj}$.
It follows that 
\begin{equation*}
\tag{5.11.5}
\begin{split}
\bigl\{ |\la^{(2)}| &+ 2|\la^{(3)}| + \cdots + 
                    (r-1)|\la^{(r)}|\bigr\}  \\ 
            &+ \bigl\{ |\mu^{(r-2)}| + 2|\mu^{(r-3)}| + \cdots + 
                     (r-2)|\mu^{(1)}| + (r-1)|\mu^{(r)}| \bigr\} \\
            &= \sum_{\substack{1 \le i \le r \\ 1 \le j \le r}}
                      (j-1)h_{ij}
                    + \sum_{\substack{0 \le i < r \\ 1 \le j \le r}}
                 (r-1-i)h_{ij}  \\            
&= \sum_{\substack{1 \le i < r \\ 1 \le j \le r}}
                           (j - i + r-2)h_{ij} 
    + \sum_{1 \le j \le r}(j + r-2)h_{rj}.
\end{split}
\end{equation*}
Subtracting (5.11.5) from (5.11.4), we see that the left hand side
of (5.11.3) is equal to 
\begin{align*}
&\sum_{\substack{ 1\le i < r \\ 1 \le j \le r}}
             \bigl([j-i-1] - (j-i-1)\bigr)h_{ij} \\  
         &+ \sum_{1 \le j \le r}([j-r-1] -(j-1))h_{rj}
\\
&= r\sum_{i=1}^{r-1}h_{i, \le i}.
\end{align*}
Hence (5.11.3) holds, and the lemma is proved.
The proof of Theorem 5.5 is now complete.
  
%%%
\par\bigskip\bigskip

\section{ Geometric properties of Kostka functions}

\para{6.1.}
We consider the complex $K_1 = (\pi_{\Bm,1})_!\Ql$ on $\SX_{\Bm,\unip}$.
By Theorem 4.8, $K_1[d'_{\Bm}]$
is a semisimple perverse sheaf equipped with $W_{\Bm}$-action, and is decomposed as 
\begin{equation*}
K_1[d'_{\Bm}] \simeq \bigoplus_{\Bla \in \SP(\Bm)}V_{\Bla}\otimes A_{\Bla},
\end{equation*} 
where $A_{\Bla} = \IC(\ol X_{\Bla}, \Ql)[\dim X_{\Bla}]$ is a simple perverse 
sheaf on $\SX_{\Bm,\unip}$. 
Assume that the map $\pi_{\Bm,1} : \wt\SX_{\Bm, \unip} \to \SX_{\Bm,\unip}$ is defined 
with respect to an $F$-stable Borel subgroup $B$. Then $\wt\SX_{\Bm, \unip}$ 
has a natural $\Fq$-structure, and $\pi_{\Bm, 1}$ is $F$-equivariant. Thus 
one can define a canonical isomorphism $\vf_1: F^*K_1 \isom K_1$. 
Since each $X_{\Bla}$ is $F$-stable, we have $F^*A_{\Bla} \simeq A_{\Bla}$, and 
$\vf_1$ induces an isomorphism $F^*(V_{\Bla} \otimes A_{\Bla}) \isom (V_{\Bla} \otimes A_{\Bla})$.
It follows that there exists a unique isomorphism 
$\vf_{\Bla,1} : F^*A_{\Bla} \isom A_{\Bla}$ such that 
$\vf_1 = \sum_{\Bla}\s_{\Bla} \otimes \vf_{\Bla,1}$, where $\s_{\Bla}$ is the identity map
on the representation space $V_{\Bla}$.
Let $\f_{\Bla}: F^*A_{\Bla} \isom A_{\Bla}$ be the natural isomorphism  induced from 
the $\Fq$-structure of $X_{\Bla}$.  Since $A_{\Bla}$ is a simple perverse sheaf, 
$\f_{\Bla}$ coincides with $\vf_{\Bla,1}$ up to scalar. 
Let $d_{\Bla}$ be as in 4.9. 
The following result can be verified in a similar way as in [SS2, (4.1.1)].

\par\medskip\noindent
(6.1.1) \ $\vf_{\Bla,1} = q^{d_{\Bla}}\f_{\Bla}$. 
In particular, the map $\vf_{\Bla,1}$ coincides with the scalar multiplication 
$q^{d_{\Bla}}$ on $A_{\Bla}|_{X_{\Bla}}$.     
\par\medskip

In fact, for $z \in \SX_{\Bm, \unip}$, we have 
$\SH^i_zK_1 \simeq H^{i}(\SB_z^{(\Bm)}, \Ql)$. 
For $z \in X^F_{\Bla}$, we have, by Theorem 4.8 (and by [S3, Prop. 8.16]), 
\begin{equation*}
H^{2d_{\Bla}}(\SB_z^{(\Bm)},\Ql) \simeq V_{\Bla} \otimes \SH^0_z\IC(\ol X_{\Bla},\Ql)
                \simeq V_{\Bla}.
\end{equation*} 
Here $d_{\Bla} = \dim \SB_z^{(\Bm)}$ by Lemma 4.10, and $H^{2d_{\Bla}}(\SB_z^{\Bm)},\Ql)$
is an irreducible $W_{\Bm}$-module.  
Since the Frobenius action on $H^{2d_{\Bla}}(\SB_z^{(\Bm)},\Ql)$ commutes with 
the $W_{\Bm}$-action, the Frobenius map acts on $H^{2d_{\Bla}}(\SB_z^{(\Bm)},\Ql)$ 
as a scalar multiplication by Schur's lemma. 
In particular, all the irreducible components of $\SB_z^{(\Bm)}$ are $F$-stable, and 
this scalar is given by $q^{d_{\Bla}}$.
It follows that $\vf_{\Bla,1}$ acts as a scalar multiplication by $q^{d_{\Bla}}$
on $\SH_z^0\IC(\ol X_{\Bla}, \Ql) \simeq \Ql$.  
Since $\f_{\Bla}$ acts as the identity map on this space, we obtain (6.1.1). 

\para{6.2.}
In general, let $K$ be a complex on a variety $X$ defined over $\Fq$ such that 
$F^*K \simeq K$.  An isomorphism $\f: F^*K \isom K$ is said to be pure 
at $x \in X^F$ if the eigenvalues of $\f$ on $\SH^i_xK$ are algebraic numbers all of whose 
complex conjugates have absolute value $q^{i/2}$. 
\par
We prove the following.

\begin{prop}  %%%%  Prop. 6.3.
Let $\f_{\Bla}: F^*A_{\Bla} \isom A_{\Bla}$ be as in 6.1. 
Assume that $z \in \SO^{\pm}_{\Bmu} \subset X_{\Bmu}^F$. 
Then  $\f_{\Bla}$ is pure at $z$.
\end{prop}

\begin{proof}
Note that if $z \in \SO^-_{\Bmu}$ (resp. $z \in \SO^{+}_{\Bmu})$, then $z = (x,\Bv)$ 
with $\Bv = (v_1, \dots, v_{r-1})$ satisfies 
the condition (a) (resp. (b)), where 
\par\medskip
\noindent
(a) \ $v_i = v_{i-1}$ if $\mu^{(i)} = \emptyset$, \\
(b) \ $v_i = 0$ if $\mu^{(i)} = \emptyset$. 
\par\medskip
First we show 
\par\medskip\noindent
(6.3.1) \ $\vf_1: F^*K_1 \isom K_1$ is pure at $z \in X_{\Bmu}$ if $z$ satisfies the condition
(a) or (b). 
\par\medskip
Assume that $\Bmu \in \SP(\Bm')$ with 
$\Bm' = (m_1', \dots, m_r') \in \SQ_{n,r}$, and let $0 \le p_1'\le \cdots \le p_r'= n$ 
be integers associated to $\Bm'$.  
For given $\Bm, \Bm'$, we define a sequence of integers $\Bm'' = (\Bm^{(1)}, \dots, \Bm^{(r)})$, 
where $\Bm^{(i)} \in \SQ_{m_i', r_i}$ with $\sum_{i=1}^r (r_i-1) = r-1$ as follows;
write the sequences $\{ p_i\}, \{ p_i'\}$ in the increasing order
$0 \le p_1''\le \cdots \le p_{2r}''= n$.  Then the sequence $\{ p_i''\}$ determines 
a composition of $m_i'$ which we denote by $\Bm^{(i)}$. 
Here $r_i$ is given by $r_i = \sharp\{ j \mid p'_{i-1} < p_j \le p'_i\} + 1$.
We consider a parabolic subgroup $P = P_{\Bm'}$, 
which is the stabilizer of a partial flag $(M_{p'_i})$ in $G$. 
Let $L = L_{\Bm'}$ be the Levi subgroup of $P$ containing $T$, and $U_P$ the unipotent 
radical of $P$. 
Note that $L \simeq G_1 \times \cdots \times G_r$ with $G_i = GL_{m_i'}$.
We define varieties
\begin{align*}
\wt\SX^L_{\Bm'',\unip} &= \wt\SX^{G_1}_{\Bm^{(1)},\unip} \times \cdots \times 
                                \wt \SX^{G_r}_{\Bm^{(r)}, \unip}, \\
\SX^L_{\Bm'',\unip}    &= \SX^{G_1}_{\Bm^{(1)},\unip} \times \cdots \times 
                                 \SX^{G_r}_{\Bm^{(r)}, \unip}, 
\end{align*}
where $\wt\SX^{G_i}_{\Bm^{(i)},\unip}$ and $\SX^{G_i}_{\Bm^{(i)}, \unip}$ are varieties
with respect to $G_i$ defined similarly to $\wt\SX_{\Bm \unip}$ and $\SX_{\Bm,\unip}$.
Thus we can define a map $\pi^L_{\Bm'',1}: \wt\SX^L_{\Bm'',\unip} \to \SX^L_{\Bm'',\unip}$
as the product of $\pi^{G_i}_{\Bm^{(i)}, 1}: 
 \wt\SX^{G_i}_{\Bm^{(i)},\unip} \to \SX^{G_i}_{\Bm^{(i)}, \unip}$.  
\par
By generalizing the diagram in (1.5.1), we obtain a diagram 
\begin{equation*}
\tag{6.3.2}
\begin{CD}
\wt\SX_{\Bm,\unip} @<\wt p_1<<  G \times \wt\SX^P_{\Bm,\unip} @>\wt q_1>>  \wt\SX^L_{\Bm'',\unip}  \\
     @V\pi'_1VV                 @VV r_1 V                       @VV\pi^L_{\Bm'',1} V    \\
\wh \SX^P_{\Bm,\unip}  @<p_1 <<  G \times \SX_{\Bm,\unip}^P  @>q_1>>  \SX^L_{\Bm'',\unip}  \\
      @V\pi''_1VV                                                       \\
   \SX_{\Bm,\unip}, 
\end{CD}
\end{equation*}
where, by putting $P\uni = L\uni U_P$ (the set of unipotent elements in $P$),  
\begin{align*}
\SX_{\Bm,\unip}^P &= \bigcup_{g \in P}g(U \times \prod_i M_{p_i}), \\ 
\wh \SX^P_{\Bm,\unip} &= G \times^P\SX_{\Bm,\unip}^P, \\ 
\wt \SX^P_{\Bm,\unip} &= P \times^{B}(U \times \prod_iM_{p_i}).  
\end{align*}
The maps are defined similarly to (1.5.1).  In particular, 
$\pi''_1\circ\pi'_1 = \pi_{\Bm,1}$. As in (1.5.1), both squares are cartesian.
The map $p_1$ is a principal $P$-bundle, and the map $q_1$ is a locally 
trivial fibration with fibre isomorphic to $G \times U_P \times \prod_{i=1}^{r-2}M_{p'_i}$.
We assume that $B$ and $L$ are $F$-stable.  Then all the objects in (6.3.2) are 
defined over $\Fq$. 
Put $K_1' = (\pi'_1)_!\Ql$ and $K_1^L = (\pi^L_{\Bm'', 1})_!\Ql$. 
As in the discussion in 1.5, we see by (6.3.2) 

\begin{equation*}
\tag{6.3.3}
p_1^*K_1' \simeq q_1^*K_1^L.
\end{equation*}
Take $z \in (\SX_{\Bm, \unip} \cap X_{\Bmu})^F$ satisfying the condition (a) or (b). 
We have 
\begin{equation*}
\wh\SX^P_{\Bm,\unip} \simeq \{ (y, gP) \in (G\uni \times V^{r-1}) \times G/P 
                                   \mid g\iv y \in \SX^P_{\Bm, \unip} \}.
\end{equation*}
Up to $G$-conjugate, one can choose $\xi = (z, P) \in (\wh\SX^P_{\Bm,\unip})^F$ 
such that $\pi_1''(\xi) = z$. 
Choose $\e \in (G \times \SX^P_{\Bm,\unip})^F$ 
such that $p_1(\e) = \xi$. Put $z'= q_1(\e) \in (\SX^L_{\Bm'',\unip})^F$.
By (6.3.3), we have

\begin{equation*}
\tag{6.3.4}
\SH^i_{\xi}(K_1') \simeq \SH^i_{\e}(p_1^*K_1') 
              \simeq \SH^i_{\e}(q_1^*K^L_1)
              \simeq \SH^i_{z'}(K^L_1).
\end{equation*}

Put $z' = (z_1', \dots, z_r') \in \SX^L_{\Bm'', \unip}$ with 
$z'_i \in \SX^{G_i}_{\Bm^{(i)}, \unip}$. 
If we denote $(\pi^{G_i}_{\Bm^{(i)},1})_*\Ql$ by $K_1^{G_i}$, $K^L_1$ can be written as 
\begin{equation*}
\tag{6.3.5}
K_1^L \simeq K^{G_1}_1 \boxtimes \cdots \boxtimes K^{G_r}_1
\end{equation*}
and so 
\begin{equation*}
\tag{6.3.6}
\SH^i_{z'}K_1^L \simeq \bigoplus_{i_1 + \cdots + i_r = i}\SH^{i_1}_{z_1'}K^{G_1}_1
              \otimes \cdots \otimes \SH^{i_r}_{z'_r}K^{G_r}_1.
\end{equation*}
The isomorphism $\vf_1^L : F^*K^L_1 \isom K_1^L$ is defined similar to $\vf_1$, 
and under the isomorphism (6.3.5), $\vf^L_1$ can be written as 
$\vf^L_1 = \vf^{G_1}_1 \otimes \cdots \otimes \vf^{G_r}_1$, where 
$\vf^{G_i}_1: F^*K^{G_i}_1 \isom K^{G_i}_1$.
Note that $z'_i \in \SX^{G_i}_{\Bm^{(i)}, \unip}$ also satisfies the condition (a) or (b).
(6.3.1) holds in the case where $r = 2$ by [AH, Corrigendum, Prop. 3]. So 
by double induction on $n$ and $r$, we may assume that $\vf^{G_i}_1$ is pure at $z_i$ for each $i$ 
unless $\Bm' = (n; -; \dots; -)$.  Now assume that $\Bm' \ne (n;-;\dots;-)$. 
Then $\vf^L_1$ is pure at $z'$ by (6.3.6).   
By (6.3.4), $K_1'$ is pure at $\xi$ with respect to $\vf_1': F^*K_1' \isom K_1'$. 
Note that $P = P_{\Bm'}$ is the unique parabolic subgroup of $G$, conjugate to $P$, 
containing $x$ such that the image of $x$ on $L \simeq G_1 \times \cdots \times G_r$ has 
Jordan type $(\mu^{(1)}, \dots, \mu^{(r)})$.  Thus $(\pi''_1)\iv(z) = \{ (z, P)\}$, and so
$\SH^i_zK_1 \simeq \SH^i_{\xi}K_1'$.   This proves (6.3.1) in the case where 
$\Bm' \ne (n;-;\dots;-)$. 
\par
It remains to consider the case $\Bm' = (n;-;\dots;-)$.  In this case, by our assumption 
we have 
$v_1 = v_2 = \dots = v_{r-1}$ or $v_2 = \dots = v_{r-1} = 0$. 
Hence the complex $K_1'$ is isomorphic to a similar complex in the case where $r = 2$.
So in this case, (6.3.1) holds by [AH].   
\par
(6.3.1) implies that the eigenvalues of $F^*$ on $H^i(\SB_z^{(\Bm)}, \Ql)$
have absolute value $q^{i/2}$. 
By Theorem 4.8, we have

\begin{equation*}
\SH^i_zK_1 \simeq \bigoplus_{\Bla \in \SP(\Bm)}V_{\Bla} 
              \otimes \SH_z^{i- d'_{\Bm} + \dim X_{\Bla}}\IC(\ol X_{\Bla}, \Ql)
\end{equation*}
Since $d_{\Bla} = (d'_{\Bm} - \dim X_{\Bla})/2$, 
$i - d'_{\Bm } + \dim X_{\Bla} = i-2d_{\Bla}$.  Since the eigenvalues of 
$\vf_1$ on $\SH^i_zK_1$ have absolute value $q^{i/2}$ by (6.3.1), 
the eigenvalues of $\vf_{\Bla,1}$ on $\SH_z^{i - 2d_{\Bla}}\IC(\ol X_{\Bla}, \Ql)$  
also have absolute value $q^{i/2}$. 
By (6.1.1), this implies that the eigenvalues of $\f_{\Bla}$ have absolute value 
$q^{i/2}$.  The proposition is proved.
\end{proof}

\remark{6.4.}
In the case where $r = 2$, $X_{\Bla}$ is a single $G$-orbit.  In this case 
the corresponding fact was proved in [AH, Corrigendum, Prop. 3], by making use of 
the contraction of a suitable transversal slice,  and by the 
result of [MS].  The purity result for the exotic 
case with $r = 2$ was also proved in [SS2, Prop. 4.4] based on the discussion in [L3], 
again by making use of the transversal slice.  
However for $r \ge 3$, the discussion by using the transversal slice 
will not work since it often happens that $Z_G(z)$ turns out to be a unipotent group
for $z \in \SX_{\Bm, \unip}$, and one cannot construct a one-parameter subgroup. 
It is not clear whether $\f_{\Bla}$ is pure for all $z \in X_{\Bmu}^F$. 

\para{6.5.}
For $\Bla, \Bmu \in \SP_{n,r}$, we define a polynomial $\IC^-_{\Bla, \Bmu}(t) \in \BZ[t]$ by
\begin{equation*}
\IC^-_{\Bla,\Bmu}(t) = \sum_{\i \ge 0}\dim \SH^{2i}_z\IC(\ol X_{\Bla}, \Ql)t^i
\end{equation*}
for $z \in \SO^-_{\Bmu}$. 
Also we define a polynomial $\IC^+_{\Bla, \Bmu}(t) \in \BZ[t]$ by 
the same formula as above if $z \in \SO^+_{\Bmu} \cap \SX^+_{\Bm, \unip}$ and by 0
if $\SO^+_{\Bmu} \cap \SX^+_{\Bm,\unip} = \emptyset$, where we assume that $\Bla \in \SP(\Bm)$.

\para{6.6.}
As in 3.10, let $\SC_q = \SC_q(\SX\uni)$ be the $\Ql$-space of all $G^F$-invariant $\Ql$-valued 
functions on $\SX\uni^F$.  The bilinear form $\lp \ , \ \rp$ on $\SC_q$ is defined as in 
(3.10.2). 
Let $\SC^{\pm}_q$ be the $\Ql$-subspace of $\SC_q$ generated by $Q^{\pm}_{\Bla}$ 
for various $\Bla \in \SP_{n,r}$. Then $\{ Q^{\pm}_{\Bla} \mid \Bla \in \SP_{n,r} \}$ 
gives a basis of $\SC^{\pm}_q$.  For an $F$-stable  $G$-orbit $C$ in 
$G\uni \times V^{r-1}$, we denote by $y_C$ the characteristic function of $C^F$. 
In the case where $C = \SO^{\pm}_{\Bla}$ we denote it by $y^{\pm}_{\Bla}$. 
Then $Q^{\pm}_{\Bla}$ can be written as a sum of various $y_C$. For each $\Bla \in \SP_{n,r}$, 
we define 
a function $Y^{\pm}_{\Bla} \in \SC^{\pm}_q$ by the condition that 
$Y^{\pm}_{\Bla}$ is a linear combination of $Q^{\pm}_{\Bnu}$ with 
$\Bnu \le \Bla$ and that
\begin{equation*}
Y^{\pm}_{\Bla} = y^{\pm}_{\Bla} + \sum_{C}b^{\pm}_{\Bla,C}y_C,
\end{equation*}
where $b^{\pm}_{\Bla,C} = 0$ for $y_C = y^{\pm}_{\Bnu}$ with $\Bnu < \Bla$.
Note that $Y^{\pm}_{\Bla}$ is determined uniquely by this condition. 
Although $Y^{\pm}_{\Bla}$ are not characteristic functions on $(\SO^{\pm}_{\Bla})^F$, 
they enjoy similar properties, 

\begin{equation*}
\tag{6.6.1}
Y^{\pm}_{\Bla}(z) = \begin{cases}
                    1   &\quad\text{ if } z \in \SO^{\pm}_{\Bla}, \\
                    0   &\quad\text{ if $z \in \SO^{\pm}_{\Bnu}$ with $\Bnu < \Bla$}.
                \end{cases}
\end{equation*}
$\{ Y^{\pm}_{\Bla} \mid \Bla \in \SP_{n,r}\}$ gives a basis of $\SC^{\pm}_q$.  
We define a matrix $\wh\vL_q = (\wh\xi_{\Bla,\Bmu}(q))$ by 
$\wh\xi_{\Bla,\Bmu}(q) = \lp Y^-_{\Bla}, Y^+_{\Bmu} \rp$. 
In the case where $r = 2$, $Y^-_{\Bla} = Y^+_{\Bla}$ is the characteristic function of 
a single $G$-orbit, hence $\wh\vL_q$ is a diagonal matrix. 
For $r \ge 3$, $Y^-_{\Bla}$ and $Y^+_{\Bmu}$ are not necessarily 
orthogonal for $\Bla \ne \Bmu$.  
We consider the following condition for $\wh\vL_q$. 
\par\medskip\noindent 
(A) \ The matrix $\wh\vL_q$ is 
upper triangular with respect to the total order $\vl$. 
\par\medskip

We show the following theorem.

\begin{thm}  %%%%  Theorem 6.7 
Suppose that the condition (A) holds for $\wh\vL_q$.
\begin{enumerate}
\item
$
\wt K^-_{\Bla,\Bmu}(t) = t^{a(\Bla)}\IC^-_{\Bla,\Bmu}(t^r).
$

\item 
Assume that $z \in (\SO^-_{\Bmu})^F$ for $\Bmu \le \Bla$.  Then 
$\SH^i_z\IC(\ol X_{\Bla},\Ql) = 0$ if $i$ is odd, and the eigenvalues of 
$\f_{\Bla}$ on $\SH^{2i}_z\IC(\ol X_{\Bla},\Ql)$ are $q^i$. 

\end{enumerate}
\end{thm}

\begin{proof}
In view of (2.4.2), $Q^-_{\Bm, T_w}(z)$ can be written as
\begin{equation*}
Q^-_{\Bm, T_w}(z) = (-1)^{d_{\Bm}}\sum_{i \ge 0}(-1)^i\Tr(\vf_1w,\SH^i_zK_1).
\end{equation*} 
Thus by Theorem 4.8, noticing that $d'_{\Bm} - \dim X_{\Bla} = 2d_{\Bla}$ is 
an even integer, we have 
\begin{equation*}
\tag{6.7.1}
Q^-_{\Bla}(z) = (-1)^{d_{\Bm}}\sum_{i \ge 0}(-1)^i
      \Tr(\vf_{\Bla,1}, \SH^i_z\IC(\ol X_{\Bla}, \Ql)).
\end{equation*}
$Q^+_{\Bla}(z)$ is obtained by restricting $Q^-_{\Bla}$ on $(\SX^+_{\Bm, \unip})^F$ 
by Proposition 2.13.
\par
We express $Q^{\pm}_{\Bla}$ as 
\begin{equation*}
\tag{6.7.2}
Q^{\pm}_{\Bla} = \sum_{\Bnu \le \Bla}\wh p^{\pm}_{\Bla,\Bnu}(q)Y^{\pm}_{\Bnu},
\end{equation*}  
Then we have 
\begin{equation*}
\tag{6.7.3}
\lp Q_{\Bla}^-, Q^+_{\Bmu}\rp = 
\sum_{\substack{\Bnu \le \Bla \\
                 \Bnu' \le \Bmu}}
     \wh p^-_{\Bla,\Bnu}(q)\wh \xi_{\Bnu,\Bnu'}(q)\wh p^+_{\Bmu,\Bnu'}(q).
\end{equation*}
We also note, by (6.1.1), that $\wh p^{\pm}_{\Bla,\Bla}(q) = (-1)^{d_{\Bm}}q^{d_{\Bla}}$.
Let $\wt Q^{\pm}_{\Bla} \in \SC^{\pm}_{q^r}$ be the functions defined by 
\begin{align*}
\wt Q^-_{\Bla}(z) &= (-1)^{p_-(\Bm)}q^{a(\Bla)}(q^r)^{-d_{\Bla}}Q^-_{\Bla}(z), \\ 
\wt Q^+_{\Bla}(z) &= (-1)^{p_+(\Bm)}q^{a(\t(\Bla))}(q^r)^{-d_{\Bla}}Q^+_{\Bla}(z). \\ 
\end{align*}
for $z \in \SX^{F^r}_{\Bm,\unip}$. 
Put, for $\Bla \in \SP(\Bm), \Bmu \in \SP(\Bm')$,
\begin{align*}
p^-_{\Bla,\Bnu}(q) &= (-1)^{d_{\Bm}}
        q^{a(\Bla)}(q^r)^{-d_{\Bla}}\wh p^-_{\Bla,\Bnu}(q^r), \\
p^+_{\Bmu, \Bnu'}(q) &= (-1)^{d_{\Bm'} + p_-(\Bm') + p_+(\Bm')}
        q^{a(\t(\Bmu)) + a(\Bnu') - a(\t(\Bnu'))}(q^r)^{-d_{\Bmu}} 
           \wh p^+_{\Bmu,\Bnu'}(q^r), \\
\xi_{\Bnu,\Bnu'}(q) &= q^{a(\t(\Bnu')) - a(\Bnu')}
                    \wh \xi_{\Bnu,\Bnu'}(q^r). 
\end{align*}
Since $d_{\Bm} = \dim G + p_-(\Bm)$ by Lemma 1.2, we have 
\begin{align*}
d_{\Bm} + d_{\Bm'}+ p_-(\Bm') + p_+(\Bm') \equiv p_-(\Bm) + p_+(\Bm') \pmod 2. 
\end{align*}
Hence by Theorem 5.5 together with (6.7.3), we have 

\begin{equation*}
\tag{6.7.4}
\sum_{z \in \SX^{F^r}\uni} \wt Q^-_{\Bla}(z)\wt Q^+_{\Bmu}(z)  = 
\sum_{\substack{\Bnu \le \Bla \\ \Bnu' \le \Bmu}}
     p^-_{\Bla,\Bnu}(q)\xi_{\Bnu,\Bnu'}(q)p^+_{\Bmu, \Bnu'}(q) = \w_{\Bla,\Bmu}(q),
\end{equation*}
with $p^-_{\Bla,\Bla}(q) = q^{a(\Bla)}$, 
$p^+_{\Bmu,\Bmu}(q) = (-1)^{p_-(\Bm') + p_+(\Bm')}q^{a(\Bmu)}$.    
Moreover, $p^{\pm}_{\Bla,\Bnu} = 0$ unless $\Bnu \le \Bla$. 
We consider the matrix $P_1^{\pm} = (p^{\pm}_{\Bla, \Bmu})$, 
$\vL_1 = (\xi_{\Bnu,\Bnu'})$ and $\Om = (\w_{\Bla,\Bmu})$.
We have a matrix equation 
\begin{equation*}
\tag{6.7.5}
P_1^-(\vL_1\,{}^t\!P_1^+) = \Om.
\end{equation*} 
Since $P_1^-$ is a lower triangular matrix with diagonal 
entries $q^{a(\Bla)}$, and $\vL_1\,{}^t\!P_1^+$ is an upper triangular matrix 
with diagonal entries $\pm q^{a(\Bmu)}\xi_{\Bmu,\Bmu}(q)$ by the condition (A),     
$P_1^-$ and $\vL_1\,{}^t\!P_1^+$ are determined uniquely from $\Om$.  In particular, 
the diagonal part of $\vL_1$ is determined uniquely from $\Om$. 
Let $\vL$ be a diagonal matrix and consider the matrix equation 
$P^-\vL\, {}^t\!P^+ = \Om$, where $P^{\pm}$ satisfy similar conditions as 
in Theorem 5.2. Then $P^{\pm}, \vL$ are determined uniquely from the equation, 
and by the uniqueness of the solution for (6.7.5), we have 
$P^- = P_1^-$ and $\vL{}^t\!P^+ = \vL_1\, {}^t\!P_1^+$.  In particular, 
$\vL$ coincides with the diagonal entries of $\vL_1$.  
(But $\vL_1$ and $P_1^+$ are not determined from the equation (6.7.5).)
Thus by Theorem 5.2, we have $p^-_{\Bla,\Bmu}(q) = \wt K^-_{\Bla,\Bmu}(q)$.
\par
By (6.6.1) and (6.7.2), we have $Q^{\pm}_{\Bla}(z) = \wh p^{\pm}_{\Bla,\Bnu}(q)$ 
for $z \in (\SO^{\pm}_{\Bnu})^F$. 
We consider the equation (6.7.1).
We replace $\vf_{\Bla,1}$ by $\f_{\Bla}$, and use (6.1.1). By replacing 
$q$ by $q^m$ for a positive integer $m$, we have  

\begin{align*}
\tag{6.7.6}
\wt K^-_{\Bla, \Bnu}(q^m) = (q^m)^{a(\Bla)}\sum_{i \ge 0}(-1)^i
\Tr(\f^{rm}_{\Bla}, \SH^i_z\IC(\ol X_{\Bla}, \Ql)).
\end{align*}
By Proposition 6.3, $\f_{\Bla}$ is pure at $z \in (\SO^-_{\Bnu})^F$.  Thus, if we put
$\SH^i_z = \SH^i_z\IC(\ol X_{\Bla},\Ql)$,  one can write as 

\begin{equation*}
\tag{6.7.7}
\Tr(\f_{\Bla}^{rm}, \SH^i_z) = \sum_{j=1}^{k_i}(\a_{ij}q^{i/2})^{rm},
\end{equation*}
where $k_i = \dim \SH^i_z$ and 
$\{ \a_{ij}q^{i/2} \mid 1 \le j \le k_i \}$ are eigenvalues of $\f_{\Bla}$ on 
$\SH^i_z$ such that $\a_{ij}$ are algebraic numbers all of whose 
complex conjugates have absolute value 1.  
By Theorem 5.2, $\wt K^-_{\Bla,\Bnu}(t)$ is  
a rational function on $t$.  Here we note that
\par\medskip\noindent
(6.7.8) \ $\wt K^-_{\Bla,\Bnu}(t)$ is a polynomial in $t$.
\par\medskip
In fact, we can write $\wt K_{\Bla,\Bmu}(t)$ as 
$\wt K_{\Bla,\Bnu}(t) = P(t) + R(t)/Q(t)$, where $P,Q,R$ are polynomials with 
$\deg R < \deg Q$ or $R = 0$. 
By (6.7.6) and (6.7.7), the absolute value of the 
right hand side of (6.7.6) goes to $\infty$ when $m \to \infty$. 
It follows that the absolute value of $\wt K_{\Bla,\Bnu}(q^m) - P(q^m)$ 
goes to $\infty$ when $m \to \infty$, if it is non-zero. 
This implies that $R = 0$, and (6.7.8) holds.
\par\medskip
Now by applying Dedekind's theorem, we see that $\a_{ij} = 0$ for odd $i$, 
and $\a_{ij} = 1$ for even $i$ such that $\SH^i_z \ne 0$. It follows that  
$\SH^i_z = 0$ for odd $i$, and $\sum_{j=1}^{k_i}\a_{ij}q^{i/2} = (\dim \SH^i_z)q^{i/2}$.  
Thus by (6.7.6), we have
$\wt K^-_{\Bla,\Bnu}(q) = q^{a(\Bla)}\IC_{\Bla,\Bnu}^-(q^r)$,
which holds for any prime power $q$. 
The assertion (i) follows from this in view of (6.7.8). 
The assertion (ii) is already shown in the proof of (i).  The theorem is proved. 
\end{proof}

\par
By using similar arguments, we can prove the following result.  Note that 
in this case, we do not need to appeal the condition (A). 
%%%%
%%%%
\begin{prop}  %%%%  Prop.  6.8
Let $\Bnu \in \SP(\Bm'')$ with $\Bm'' = (m_1'', \dots, m_r'')$. 
Assume that $m_i'' = 0$ for $i = 1, \dots, r-2$.  
Then for any $\Bla \in \SP(\Bm)$ and $\Bmu \in \SP(\Bm')$, 
the following holds.
\begin{enumerate}
\item
We have
\begin{align*}
\wt K^-_{\Bla,\Bnu}(t) &= t^{a(\Bla)}\IC^-_{\Bla,\Bnu}(t^r), \\
\wt K^+_{\Bmu,\Bnu}(t) &= t^{a(\t(\Bmu)) + a(\Bnu) - a(\t(\Bnu))} \IC^+_{\Bmu,\Bnu}(t^r). 
\end{align*}
\item Assume that $z \in X_{\Bnu}^F$ for $\Bnu \le \Bla$.  Then 
$\SH^i_z\IC(\ol X_{\Bla},\Ql) = 0$ if $i$ is odd, and the eigenvalues of 
$\f_{\Bla}$ on $\SH^{2i}_z\IC(\ol X_{\Bla},\Ql)$ are $q^i$. 
\end{enumerate}
\end{prop}

\begin{proof}
Put $\Bla_0 = (-;\cdots;-;n;-)$. 
Then for any $\Bnu \le \Bla_0$, $X_{\Bnu} = \SO^{\pm}_{\Bnu}$ is a single $G$-orbit. 
It follows that $Y^{\pm}_{\Bla}$ coincides with the characteristic function 
$y_{\Bla} = y_{\Bla}^{\pm}$ of $X_{\Bla}^F$ for $\Bla \le \Bla_0$. 
Assume that $\Bla \le \Bla_0$ or $\Bmu \le \Bla_0$.  Then (6.7.4) can be rewritten as 
\begin{equation*}
\tag{6.8.1}
\sum_{z \in \SX^{F^r}\uni} \wt Q^-_{\Bla}(z)\wt Q^+_{\Bmu}(z)  = 
\sum_{\Bnu \le \Bla}
     p^-_{\Bla,\Bnu}(q)\xi_{\Bnu,\Bnu}(q)p^+_{\Bmu, \Bnu}(q) = \w_{\Bla,\Bmu}(q),
\end{equation*}
We consider the matrix $P_0^{\pm} = (p^{\pm}_{\Bla, \Bmu})$ 
$\vL_0 = (\xi_{\Bla,\Bmu})$ and $\Om_0 = (\w_{\Bla,\Bmu})$
indexed by $\Bla, \Bmu \le \Bla_0$.
By (6.8.1), the matrices $P^{\pm}_0, \vL_0$ and $\Om_0$ satisfy 
a similar condition as in Theorem 5.2. (Note that $p_+(\Bm') = p_-(\Bm')$ 
if $\Bmu \le \Bla_0$, hence $p^+_{\Bmu,\Bmu}(q) = q^{a(\Bmu)}$.) 
Thus the equation (6.8.1) determines 
uniquely the matrices $P^{\pm}_0$ and $\vL_0$. 
Hence by Theorem 5.2, $p^{\pm}_{\Bla,\Bmu}(q) = \wt K^{\pm}_{\Bla,\Bmu}(q)$ 
for any $\Bla,\Bmu \le \Bla_0$. In particular, $p^+_{\Bmu, \Bnu}(q)$ and 
$\xi_{\Bnu,\Bnu}$ are determined 
for $\Bnu \le \Bmu \le \Bla_0$. We now consider arbitrary $\Bla \in \SP_{n,r}$,  Since 
$p^+_{\Bnu,\Bnu}(q) = q^{a(\Bnu)}$, the equation (6.8.1) determines uniquely 
$p^-_{\Bla,\Bnu}(q)$ for $\Bnu \le \Bla_0$, by induction on the total order $\vl$ on $\SP_{n,r}$.   
Again by Theorem 5.2, we see that $p^-_{\Bla,\Bnu}(q) = \wt K^-_{\Bla,\Bnu}(q)$.
Similar argument also holds for $p^+_{\Bmu,\Bnu}(q)$.  
Thus we have proved

\begin{equation*}
\tag{6.8.2}
p^-_{\Bla, \Bnu}(q) = \wt K^-_{\Bla,\Bnu}(q), \quad 
 p^+_{\Bmu,\Bnu}(q) = \wt K^+_{\Bmu, \Bnu}(q) \qquad \text{ for any $\Bnu \le \Bla_0$.}
\end{equation*}

By using a similar argument as in the proof of (6.7.6), we have 

\begin{align*}
\tag{6.8.3}
\wt K^-_{\Bla, \Bnu}(q^m) &= (q^m)^{a(\Bla)}\sum_{i \ge 0}(-1)^i
     \Tr(\f^{rm}_{\Bla}, \SH^i_z\IC(\ol X_{\Bla}, \Ql)), \\
\wt K^+_{\Bmu, \Bnu}(q^m) &= (q^m)^{a(\t(\Bmu)) + a(\Bnu) - a(\t(\Bnu))}\sum_{i \ge 0}(-1)^i
     \Tr(\f^{rm}_{\Bmu}, \SH^i_z\IC(\ol X_{\Bmu}, \Ql)). \\
\end{align*}
(Here in the latter formula, we assume that $z \in \SO_{\Bnu}^+ \cap \SX^+_{\Bm'}$ 
for $\Bmu \in \SP(\Bm')$.  In the case where $\SO^+_{\Bnu} \cap \SX^+_{\Bm'} = \emptyset$, 
we have $\wt K^+_{\Bmu, \Bnu}(q^m) = 0$.) 
Now the proposition follows by a similar argument as in the proof of Theorem 6.7.
\end{proof}

The condition(A) can be verified in small rank cases.  
We have the following result.

\begin{prop}  %%%%  Prop. 6.8.
Assume that $n = 1, 2$, and $r$ is arbitrary.  Then the condition (A) holds 
for $\wh\vL_q$. 
\end{prop}

\begin{proof}
First consider the case where $n = 1$.  In this case, 
$X_{\Bnu} \cap \SX^+_{\Bm'}$ coincides with $\SO^+_{\Bnu}$ for any $\Bnu$ and 
$\Bm'$, if it is non-empty.  
Hence $Y^+_{\Bmu} = y^+_{\Bmu}$ for any $\Bmu$.  One can check that $Y^-_{\Bmu}$ 
coincides with the characteristic function of $X_{\Bmu}^F$.  Thus the condition (A)  
holds.  In fact, in this case $\wh\vL_q$ is a diagonal matrix, and 
by applying the arguments in the proof of Proposition 6.8, 
$\wt K^-_{\Bla,\Bnu}(t), \wt K^+_{\Bmu,\Bnu}(t)$ are given as in the formulas
in Proposition 6.8 (i). 
\par
Next consider the case where $n = 2$.
We determine the pair $(\Bm', \Bnu)$ satisfying the condition 
(*) $X_{\Bnu} \cap \SX^+_{\Bm'}$  
splits into more that two orbits.  This occurs only when 
$\Bm'$ is of the form  $\Bm' = (\cdots,1,\cdots,1,\cdots)$, where 
non-zero factors occur on $1 \le k < \ell< r$, and 
$\Bnu$ is 
such that $\nu^{(k)} = (1^2)$, or 
$\Bnu = \Bla(\Bm') = (\cdots;1;\cdots;1;\cdots)$.   
In that case, we have 
$X_{\Bnu} \cap \SX^+_{\Bm'} = \SO^+_{\Bnu} \coprod X^+_{\Bnu, \Bm'}$, where 
$X^+_{\Bnu,\Bm'}$ consists of infinitely many $G$-orbits. 
Thus we have, for any $1\le k < \ell < r$,
\begin{equation*}
\tag{*}  \Bm' = (\cdots, 1, \cdots, 1, \cdots), \quad 
     \Bnu = (\cdots;1^2; \cdots; -; \cdots) \quad \text{ or } \quad \Bnu = \Bla(\Bm'),
\end{equation*}
where 1 appear in the $k,\ell$-th factors for $\Bm'$, and $1^2$ appears in 
the $k$-th factor for $\Bnu$.
We show
\par\medskip\noindent
(6.9.1) \ 
The function $Y^+_{\Bmu}$ coincides with $y^+_{\Bmu}$ unless 
$\Bmu \in \SP(\Bm')$ for $\Bm'$ in (*), in which case,
$Y^+_{\Bmu} = y^+_{\Bmu} + \sum_{\Bnu,C}a_Cy_C$ with $C \subset  X^+_{\Bnu,\Bm'}$ 
for $\Bnu$ in (*).
\par\medskip
If $\Bmu = (\cdots;1^2;\cdots)$, any $\Bnu \le \Bmu$ has a similar form as $\Bmu$.
Hence by (*), $Y^+_{\Bmu} = y^+_{\Bmu}$.  If $\Bmu = (\cdots;1;\cdots;1)$, 
any $\Bnu \le \Bmu$ has a similar form or a simialr form as in the previous case.
Hence again by (*), $Y^+_{\Bmu} = y^+_{\Bmu}$. 
Assume that $\Bmu = (\cdots;1;\cdots;1;\cdots)$ with $1 \le k < \ell<r$.
Then $X_{\Bnu} \cap \SX^+_{\Bm'} \ne \emptyset$ only when 
$\Bnu = \Bmu$,  $(\cdots;1^2;\cdots)$, or $(\cdots;1; \cdots;1)$.
In the second and the third case, by the previous results, $Y^+_{\Bnu} = y^+_{\Bnu}$.
Except the case where $\Bnu$ is as in (*), 
$X_{\Bnu} \cap \SX^+_{\Bm'} = \SO_{\Bnu}^+$.
Hence by subtracting those functions $y^+_{\Bnu}$ from $Q^+_{\Bmu}$, 
we obtain the function which has supports only on $\SO_{\Bmu}^+$ and $X^+_{\Bnu,\Bm'}$.
Finally assume that $\Bmu = (\cdots;2;\cdots)$. 
Then $X_{\Bnu} \cap \SX^+_{\Bm'} \ne \emptyset$ only when
$\Bnu = \Bmu$, or $\Bnu = (\cdots:1^2;\cdots), (\cdots;1; \cdots;1)$. 
Moreover, by (*), in each case, $X_{\Bnu} \cap \SX^+_{\Bm} = \SO^+_{\Bnu}$. 
Hence by subtracting those $y^+_{\Bnu}$ such that $\Bnu \ne \Bmu$ from 
$Q^+_{\Bmu}$, we obtain $Y^+_{\Bmu} = y^+_{\Bmu}$.  
This proves (6.9.1).
\par
Next we determine the pair $(\Bm, \Bnu)$ satisfying the condition 
(**) $X_{\Bnu} \cap \SX_{\Bm} \neq \emptyset$ and 
$X_{\Bnu} \not\subset \SX_{\Bm}$.  
If $\Bm = (\dots, 2, \dots)$, clearly $X_{\Bnu} \subset \SX_{\Bm}$.  So 
assume that $\Bm$ is of the form $(\dots, 1, \dots, 1, \dots)$ for some $1 \le a < b \le r$. 
If $\Bnu$ is of the form $(\cdots; 2; \cdots)$ for $k$-factor, then $k \ge b$
since $\Bla(\Bm) \ge \Bnu$.  But in this case, $X_{\Bnu} \subset \SX_{\Bm}$.  
If $\Bnu = (\cdots;1;\cdots;1\cdots)$ for some $k ,\ell$-factors, then 
$a \le k, b \le \ell$, and $X_{\Bnu} \subset \SX_{\Bm}$.  Hence we assume 
that $\Bnu = (\cdots; 1^2; \cdots)$ for some $k$-factor.  If $k \ge b$, 
then $X_{\Bnu} \subset \SX_{\Bm}$.  So, we have $a \le k < b$.  Then 
$X_{\Bnu} \not\subset \SX_{\Bm}$.  Hence we have

\begin{equation*}   
\tag{**}
\Bm = (\cdots, 1,\cdots, 1, \cdots), \quad \Bnu = (\cdots; 1^2; \cdots),
\end{equation*}
where $k$ factor appears for $\Bnu$, and $a,b$ factors appear for $\Bm$ 
with $a \le k < b$. 
For each $\Bla$, let $\wt y_{\Bla}$ be the characteristic function of $X_{\Bla}^F$.
We show
\par\medskip\noindent
(6.9.2) \ $Y^-_{\Bla}$ coincides with  
$\wt y_{\Bla}$ unless $\Bla = (\cdots;1;\cdots;1;\cdots)$ 
with $\la^{(a)} = \la^{(b)} = 1$ for $a < b$, in which case,  
$Y^-_{\Bla} = \wt y_{\Bla} + \sum_Ca_Cy_C$ with $C \subset X_{\Bnu}\backslash \SO^-_{\Bnu}$
for $\Bnu = (\cdots;1^2;\cdots)$ with $\nu^{(a)} = (1^2)$. 
\par\medskip
If $\Bla = (;\cdots;1^2;\cdots)$, any $\Bnu \le \Bla$ has a similar type as 
$\Bla$. We have $X_{\Bnu} \subset \SX_{\Bm}$ and $\SB^{(\Bm)}_z$ has a common 
structure for any $z \in X_{\Bla}$.  Thus by backwards induction on $k$ 
($1^2$ appears in the $k$-th factor), we see that $Y^-_{\Bla} = \wt y_{\Bla}$.
Assume that $\Bla = (\cdots;1;\cdots;1;\cdots)$, where 
$\la^{(i)} = (1)$ for $i = a, b$ with $a < b - 1$.  
Put $\Bla' = (\cdots;1;\cdots;1;\cdots)$, where $\la'^{(a+1)} = \la'^{(b)} = (1)$.
If $\Bnu \le \Bla$,  then $\Bnu$ has the type 
$\Bnu = (\cdots;1^2;\cdots)$ or $(\cdots;1;\cdots;1;\cdots)$. 
Assume that $\Bnu$ is such that $\nu^{(k)} = (1^2)$ for some $k \ge a$.
We have
\begin{align*} 
\tag{6.9.3}
X_{\Bnu} &= \{ (x,\Bv) \in G\uni \times V^{r-1} 
      \mid x = 1, v_i = 0 \text{ for } i \le k-1, v_k \ne 0\}, \\ 
X_{\Bnu} \cap \SX_{\Bm} &= \{(x, \Bv) \in X_{\Bla} \mid \lp v_i \rp = \lp v_a \rp 
         \text{ for } k +1 \le i < b\}.
\end{align*}
Take $z \in X_{\Bnu} \cap \SX_{\Bm}$.
Then $\SB^{(\Bm)}_z$ is equal to a one point if $k < b$, and is equal to $\SB$ 
if $k \ge b$.
If $\Bnu = (\cdots;1;\cdots;1;\cdots)$, then 
$\SB^{(\Bm)}_z$ is a one point for any 
$z \in X_{\Bnu}$.  A similar property also holds for the pair $(\Bm'', \Bnu')$ 
if $\Bla' \in \SP(\Bm'')$ and $\Bnu'\le \Bla'$.  
It follows that the function $Q^-_{\Bla} - Q^-_{\Bla'}$ 
has supports only on $X_{\Bla}$ and 
$X_{\Bnu} \cap \SX_{\Bm}$ for $\Bnu$ such that $\nu^{(a)} = (1^2)$.
By subtracting $\wt y_{\Bnu}$ from this, we see that 
$Y^-_{\Bla} = \wt y_{\Bla} + \sum_Ca_Cy_C$, where  
$C \subset X_{\Bnu} \backslash \SO_{\Bnu}^-$ for $\Bnu$ with $\nu^{(a)} = (1^2)$.
Next assume that $\Bla = (\cdots;1;1;\cdots)$ with $\la^{(a)} = \la^{(a+1)} = 1$.
If $a+1 = r$, it is easy to see that $Y^-_{\Bla} =  \wt y_{\Bla}$.  So 
assume that $a+1 < r$.  We consider $\Bla' = (\cdots;1;-;1;\cdots)$ with 
$\la'^{(a)} = \la'^{(a+2)} = 1$. Then by a similar consideration as above, 
$Q^-_{\Bla} - Q^-_{\Bla'}$ has supports 
only on $X_{\Bla}$ and $X_{\Bnu}$ with $\Bnu$ such that 
$\nu^{(a)} = (1^2)$. Hence $Y^-_{\Bla}$ can be written as in the previous case.
Finally assume that $\Bla = (\cdots;2;\cdots)$ with $\la^{(a)} = (2)$.
If $a = r$, it is easy to see that $Y^-_{\Bla} = \wt y_{\Bla}$.
So assume that $a < r$, and put
$\Bla' = (\cdots;1;1;\cdots)$ with $\la'^{(a)} = \la'^{(a+1)} = 1$.
As in 
the previous discussion, $Q^-_{\Bla} - Q^-_{\Bla'}$ has supports only on 
$\Bla$, $\Bnu = (\cdots;2;\cdots)$ with $\nu^{(k)} = (2)$ for $a < k$, 
or $\Bnu = (\cdots;1^2;\cdots)$ with $\nu^{(a)} = 1$.  
If $\Bnu$ is in the former case, by induction we may assume 
that $Y^-_{\Bnu} = \wt y_{\Bnu}$.  By the previous results, we also have
$Y^-_{\Bnu} = \wt y_{\Bnu}$ for the latter $\Bnu$.  
Hence we have $Y^-_{\Bla} = \wt y_{\Bla}$.  This proves 
(6.9.2).
\par
Now assume that $\Bla$ is of the form $\Bla = (\cdots;1;\cdots;1;\cdots)$ 
with $\la^{(a)} = \la^{(b)} = 1$ for $a < b$.
By (6.9.2), $Y^-_{\Bla}$ has an additional support only for  
$\Bnu = (\cdots;1^2;\cdots)$ with $\nu^{(a)} = (1^2)$. 
Let $\Bmu = (\cdots;1;\cdots;1\cdots) \in \SP(\Bm')$ be as in (*) with $k = a$. 
Then $X_{\Bnu} \cap \SX^+_{\Bm'} = \SO^+_{\Bnu} \coprod X^+_{\Bnu,\Bm'}$.
We show that 
\par\medskip\noindent
(6.9.4)  \ If $b < \ell$, then $\lp Y^-_{\Bla}, Y^+_{\Bmu} \rp = 0$. 
In particular, if $\lp Y^-_{\Bla}, Y^+_{\Bmu}\rp \ne 0$, then $\Bla \not > \Bmu$. 
\par\medskip
First assume that $a < b-1$, and recall the discussion in the proof 
of (6.9.2).  Consider $Q^-_{\Bla} - Q^-_{\Bla'}$ as in (6.9.2) 
for $\Bla' = (\cdots;1;\cdots;1;\cdots)$ with $\la'^{(a+1)} = \la'^{(b)} = 1$.
By our assumption $b < \ell$, it follows from (6.9.3) that 
$X_{\nu} \cap \SX_{\Bm}$ contains $\SO^+_{\Bnu} \cup X^+_{\Bnu,\Bm'}$. 
Then if we subtract $\wt y_{\Bnu}$ from $Q^-_{\Bla} - Q^-_{\Bla'}$, the resulting 
function does not have support on $X^+_{\Bnu,\Bm'}$.  Hence $Y^-_{\Bla}$ 
does not have support on $X^+_{\Bnu,\Bm'}$.  Then by (6.9.1), we see that 
$\lp Y^-_{\Bla}, Y^+_{\Bmu} \rp = 0$. 
The proof for $a = b-1$ case is similar, following the discussion of the proof
of (6.9.2).  
Thus (6.9.4) is proved.    
\par
We now define a total order $\vl$ on $\SP_{n,r}$ compatible with the partial 
order $\le$ so that $\wh\vL_q$ is upper triangular.   
$\SP_{n,r}$ is decomposed as $\SP_{n,r} = \coprod_{\nu \in \SP_n}\SP[\nu]$, where 
$\SP[\nu] = \{ \Bla \in \SP_{n,r} \mid \sum \la^{(i)} = \nu \}$. 
We arrange $\SP[\nu]$ according to the total order compatible with the dominance order 
on $\SP_n$. In our case, $\SP_{2,r} = \SP[(2)] \coprod \SP[(1^2)]$. 
We put the lexicographic order on $\SP[(2)]$ and on $\SP[(1^2)]$. We define 
$\SP[(2)] \vg \SP[(1^2)]$. By (6.9.1), (6.9.2) and (6.9.4), 
we see that if 
$\lp Y^-_{\Bla}, Y^+_{\Bmu}\rp \ne 0$, 
then $\Bla, \Bmu \in \SP[(1^2)]$, and $\Bla \vle \Bmu$.  
The proposition is proved. 
\end{proof}

\remarks{6.10.} \ 
(i) In the discussion of the proof of Theorem 6.7, the equation (6.7.5) 
holds without the assumption (A).  Thus, under the notation there, if 
one can show that $P_1 = P$, then we have $\vL\,{}^t\!P^+ = \vL_1\,{}^t\!P_1^+$.
This implies that $\vL_1$ is upper triangular, and so the condition (A)
holds. Since $Q^-_{\Bla}(z) = \wh p^-_{\Bla}(q)$, the condition $P_1 = P$ 
is equivalent to the formula (cf. (6.7.6)) 
\begin{equation*}
\tag{6.10.1}
\wt K_{\Bla,\Bmu}^-(q) = q^{a(\Bla)}\sum_{i \ge 0}(-1)^i
             \Tr(\f_{\Bla}^r, \SH^i_z\IC(\ol X_{\Bla},\Ql)).
\end{equation*}
Thus the condition (A) is equivalent to the formula (6.10.1). 
\par
(ii)
The matrix $\wh \vL_q$ is not diagonal even in the case where $n = 2, r= 3$.
In fact, assume that $\Bla = (1;-;1), \Bmu = (1;1;-)$ and $\Bnu = (1^2;-;-)$ 
with $\Bla \in \SP(\Bm), \Bmu \in \SP(\Bm')$.
In this case, $X_{\Bnu} \cap \SX_{\Bm}$ does not contain $X^+_{\Bnu,\Bm'}$, hence 
$Y^-_{\Bla} = \wt y_{\Bla} + y$, where $y$ is a function whose support is 
contained in $X_{\Bnu}^F$, and $y$ is non-zero on $X^+_{\Bnu,\Bm'}$. Moreover, 
by (6.9.1), $Y^+_{\Bmu} = y^+_{\Bmu} + \sum_Ca_Cy_C$ with $C \subset X^+_{\Bnu,\Bm'}$.  
It follows that $\lp Y^-_{\Bla}, Y^+_{\Bmu}\rp \ne 0$.
Also we have $\lp Y^-_{\Bnu}, Y^+_{\Bmu}\rp \ne 0$ since $Y^-_{\Bnu} = \wt y_{\Bnu}$ by 
(6.9.2).  
In fact, these are the only cases such that $\lp Y^-_{\Bla'}, Y^+_{\Bmu'}\rp \ne 0$.

\bigskip
\section{Some examples}

\para{7.1.}
We consider the matrix equation  $P^-\vL {}^tP^+ = \Om$  
with $P^{\pm} = (\wt K^{\pm}_{\Bla,\Bmu}(t))$ as in (5.2.1).
In the case where $n = 1$, 
Kostka functions have an interpretation in terms of the polynomials
$\IC^{\pm}_{\Bla, \Bmu}(t)$ by Proposition 6.9 (see also Proposition 6.8), namely, the
following formula holds.
\begin{align*}
\tag{7.1.1}
\wt K^-_{\Bla, \Bnu}(t) &= t^{a(\Bla)}\IC^-_{\Bla, \Bnu}(t^r), \\
\wt K^+_{\Bmu, \Bnu}(t) &= t^{a(\tau(\Bmu)) + a(\Bnu) - a(\tau(\Bnu))}
                               \IC^+_{\Bmu, \Bnu}(t^r).
\end{align*}
In this case, $W = G(r,1,1) \simeq \BZ/r\BZ$.  Put 
$\SP_{1,r} = \{ \Bla_1, \dots, \Bla_r\}$, arranged with respect to the
dominance order, $\Bla_1 < \Bla_2 < \cdots < \Bla_r$ 
where $\Bla_i = (\la^{(1)}, \dots, \la^{(r)})$ 
with $\la^{(r+1-i)} = (1), \la^{(j)} = \emptyset$ for $j \ne r+1-i$.  
\par
We first consider the simplest case, i.e., the case 
 where $n = 1$ and $r = 3$.  Hence 
$\SP_{1,r} = \{ \Bla_1, \Bla_2,\Bla_3\}$. 
In this case  
the equation $P^-\vL\, {}^t\!P^+ = \Om$  is given by
\begin{equation*}
\begin{pmatrix}
   t^2 &     &   \\
   t   &  t  &   \\
   1   &  1  &  1  
\end{pmatrix}
\begin{pmatrix}
  1  &   &     \\
     &  t^2 - t\iv  &  \\
     &   &  t^4 - t
\end{pmatrix}
\begin{pmatrix}
 t^2  &  1  &  t  \\
      &  t  &  0  \\
      &     &  1
\end{pmatrix}
= 
\begin{pmatrix}
 t^4  &  t^2  &  t^3  \\
 t^3  &  t^4  &  t^2  \\
 t^2  &  t^3  &  t^4
\end{pmatrix}.
\end{equation*}
Note that the matrix $\vL$ 
has non-polynomial entries. 
However, the left hand side of this equation is
changed to the form
\begin{equation*}
P'\vL'\, {}^t\!P'' = 
\begin{pmatrix}
  t^2  &       &     \\
  t    &   t   &     \\
  1    &   1   &  1  \\   
\end{pmatrix}
\begin{pmatrix}
  1    &         &         \\
       &  t^3-1  &          \\
       &         &  t^3 -1 
\end{pmatrix}
\begin{pmatrix}
  t^2   &  1  &  t  \\
        &  1  &  0  \\
        &     &  t    
\end{pmatrix},
\end{equation*}
where $P' = P^-, P'' = P^+\vT\iv, \vL' = \vL\vT$ with a diagonal 
matrix $\vT = \Diag (1,t,t\iv)$.
In this case, we have 
\begin{equation*}
(\IC^-_{\Bla,\Bmu}(t^3)) = \begin{pmatrix}
                            1  &     &   \\
                            1  &  1  &    \\
                            1  &  1  &  1
                          \end{pmatrix}, 
\quad
 (\IC^+_{\Bla,\Bmu}(t^3)) = \begin{pmatrix}
                            1  &      &    \\
                            1  &   1  &    \\
                            1  &   0  &  1
                           \end{pmatrix}.
\end{equation*}
We define  polynomials $\xi_{\Bla_i}(t)$ by 
$\xi_{\Bla_1}(t) = 1, \xi_{\Bla_2}(t) = t-1, 
\xi_{\Bla_3}(t) = t-1$.  Then we have,
\begin{equation*}
\vL' = \begin{pmatrix}
             \xi_{\Bla_1}(t^3)  &            &     \\
                       &  \xi_{\Bla_2}(t^3)  &     \\
                       &            &   \xi_{\Bla_3}(t^3)  \\
       \end{pmatrix}, 
\quad\text{ and }\quad 
\vL\iv \vL' = \begin{pmatrix}
                   1   &     &    \\
                       &  t  &    \\
                       &     &  t\iv
                \end{pmatrix}.
\end{equation*}
We have $(a(\Bla_1), a(\Bla_2), a(\Bla_3)) = (2,1,0)$.  
In our case $\t$ is a permutation 
$\Bla_1 \lra \Bla_1, \Bla_2 \lra \Bla_3$ and so 
$(a(\t(\Bla_1)), a(\t(\Bla_2)), a(\t(\Bla_3))) = (2,0,1)$. 
Then $P^- = P'$ is obtained from $(\IC^-_{\Bla,\Bmu}(t^3))$ by 
multiplying $t^2, t,1$ for corresponding rows. 
In turn, $P''$ is obtained from $(\IC^+_{\Bla,\Bmu}(t^3))$ by 
multiplying $t^2, 1, t$ for corresponding rows, and $P^+$ 
is obtained from $P''$ by multiplying $1,t, t\iv$ for 
corresponding columns.  
Moreover, we have 
\begin{equation*}
\vT = (t^{a(\Bla_1) - a(\tau(\Bla_1))}, t^{a(\Bla_2)-a(\tau(\Bla_2))},
   t^{a(\Bla_3) - a(\tau(\Bla_3))}).
\end{equation*}

\para{7.2.}
We consider the general $r$ with $n = 1$.
As in 7.1, we consider the relation  
$P'\vL'\,{}^t\!P''= \Om$ 
where $P' = P^-, P'' = P^+\vT\iv, \vL' = \vL\vT$ with a diagonal 
matrix $\vT = \Diag (\dots, t^{a(\Bla) -a(\t(\Bla))}, \dots)$.
We give the matrices 
$P^{\pm}, (\IC^{\pm}_{\Bla,\Bmu}(t^r)), \vL, \vL'$ and $\vT$.
The matrices $P^{\pm}$ are given as follows.

\begin{equation*}
P^- = \begin{pmatrix}
          t^{r-1} &         &         &        &         &   \\
          t^{r-2} & t^{r-2} &         &        &         &   \\
          t^{r-3} & t^{r-3} & t^{r-3} &        &         &   \\
          \cdots  & \cdots  & \cdots  & \ddots &         &   \\
          t       & t       & t       & t      &  t      &   \\
          1       & 1       & 1       & 1      &  1      &  1 
       \end{pmatrix},  \quad
P^+ = \begin{pmatrix}
          t^{r-1} &         &         &         &    &  \\
          1       & t^{r-2} &         &         &    &  \\
          t       & 0       & t^{r-3} &         &    &  \\
          \cdots  & \cdots  & \cdots  & \ddots  &    &   \\    
          t^{r-3} & 0       & \cdots  & 0       & t  &   \\
          t^{r-2} & 0       & \cdots  & 0       & 0  & 1
       \end{pmatrix}.
\end{equation*}
The diagonal matrix $\vL$ is given as
\begin{equation*}
\vL = \Diag (1, t^2 - t^{-r+2}, t^4 - t^{-r+4}, \dots, 
             t^{2r-4} - t^{r-4}, t^{2r-2} - t^{r-2} ).
\end{equation*}
The permutation of $\SP_{1,r}$ is given by 
$\Bla_1 \lra \Bla_1$ and 
$\Bla_i \lra \Bla_{r-i+2}$ for $i \ne 1$.
Then $\vT$ is given by 
\begin{equation*}
\vT = \Diag ( 1, t^{r-2}, t^{r-4}, \dots, t^{-r+4}, t^{-r+2}),
\end{equation*}
and the matrix $\vL' = \vL\vT$ is given by
\begin{equation*}
\vL' = (1, t^r-1, t^r-1,\cdots, t^r-1 ).
\end{equation*}
Finally the matrices $(\IC^{\pm}_{\Bla,\Bmu}(t^3))$ are given by 

\begin{align*}
(\IC^-_{\Bla,\Bmu}(t^3)) &= 
     \begin{pmatrix}
          1       &        &        &        &    &     \\
          1       & 1      &        &        &    &     \\
          1       & 1      & 1      &        &    &     \\ 
          \cdots  & \cdots & \cdots & \ddots &    &      \\
          1       & 1      & 1      & 1      & 1  &   \\
          1       & 1      & 1      & 1      & 1  & 1        
       \end{pmatrix}, \quad
(\IC^+_{\Bla,\Bmu}(t^3)) &=
     \begin{pmatrix}
          1       &        &        &        &    &     \\
          1       & 1      &        &        &    &     \\
          1       & 0      & 1      &        &    &     \\ 
          \cdots  & \cdots & \cdots & \ddots &    &      \\
          1       & 0      & 0      & 0      & 1  &   \\
          1       & 0      & 0      & 0      & 0  & 1    \\
      \end{pmatrix}.
\end{align*}

\para{7.3.}
Assume that $W = G(3,1,2) \simeq \FS_2\ltimes (\BZ/3\BZ)^2$.
We arrange the elements in  $\CP_{3,2}$ in the total order 
$\Bla_1 \vl  \Bla_2 \vl \cdots \vl \Bla_9$  compatible
with the dominance order, where 
\begin{align*}
            \Bla_1 &= (-;-;1^2), & \Bla_2 &= (-;1^2;-), 
               &  \Bla_3 &= (-;-;2),   \\
                \Bla_4 &= (1^2;-;-), & \Bla_5 &= (-;1;1), 
               &  \Bla_6 &= (1;-;1),    \\
                \Bla_7 &= (-;2;-), & \Bla_8 &= (1;1;-), 
               &  \Bla_9 &= (2;-;-).
\end{align*}
Then the matrices $P^{\pm}$ are given as follows.

\begin{align*}
P^- &= \begin{pmatrix}
          t^7 &  &  &     &     &     &     &     &  &\\
          t^5 & t^5 &     &     &     &     &     &  & \\
          t^4 &  0  & t^4 &     &     &     &     &  & \\
          t^3 & t^3 & 0   & t^3 &     &     &     &  &  \\
    t^6 + t^3 & t^3 & t^3 & 0   & t^3 &     &     &  & \\
    t^5 + t^2 & t^2 & t^2 & t^2 & t^2 & t^2 &     &  & \\
          t^2 & t^2 & t^2 & 0   & t^2 & 0   & t^2 &  & \\
    t^4 + t   & t^4 + t & t & t & t   & t   & t & t  &    &  \\
            1 &  1  &  1 &  1 &  1 &  1 &  1 &  1 &  1           
       \end{pmatrix},  \\ \\
P^+ &= \begin{pmatrix}
          t^7 &   &   &   &   &   &   &   &  \\
          t^3 & t^5 &   &   &   &   &   &    & \\
          t^4 & 0   & t^4  &   &   &   &    &   &  \\
          t^5 & 0   &  0   & t^3  &   &    &   &   & \\    
    t^5 + t^2 & t^4 & t^2  & 0    & t^3  &   &   &    &  \\
    t^6 + t^3 & 0   & t^3  & t    &  0   &  t^2 &   &   &  \\
          1   & t^2 & 1    & 0    &  t   &  0   & t^2 &   &  \\
    t^4 + t   & t^3 &  t   & t^2  &  t^2 &  0   & 0   & t  &  \\
         t^2  & 0   &  t^2 & 1    &  0   &  t   & 0   & 0  & 1        
    \end{pmatrix}.
\end{align*}
\par
The diagonal matrix $\vL$ is given as 
\begin{align*}
\vL = \Diag\biggl( &1, &  &t^{-2}(t^6-1), &   &(t^6-1), \\ 
        &t^2(t^6-1), &  &t^{-1}(t^3 - 1)(t^6-1), &  &t(t^3-1)(t^6-1), \\
           &t(t^3-1)(t^6-1), &  &t^3(t^3-1)(t^6-1), &
                                          &t^5(t^3-1)(t^6-1)\biggr).
\end{align*}
In this case the permutation $\t$ on $\CP_{2,3}$ is given as
\begin{equation*}
\Bla_2 \lra \Bla_4, \quad \Bla_5 \lra \Bla_6, \quad \Bla_7 \lra \Bla_9,
\quad  (\text{other $\Bla_j$ are fixed}).
\end{equation*}
Then $\vT$ is given by 
$\vT = \Diag (1, t^2, 1, t^{-2}, t, t\iv,  t^2, 1, t^{-2})$,
and the matrix $\vL' = \vL\vT$ is given by 
\begin{align*}
\vL' = \Diag\biggl( &1, & &(t^6-1),&  &(t^6-1), \\ 
        &(t^6-1), &  &(t^3-1)(t^6-1), &  &(t^3-1)(t^6-1), \\
           &t^3(t^3-1)(t^6-1), &  &t^3(t^3-1)(t^6-1), &
                                          &t^3(t^3-1)(t^6-1)\biggr).
\end{align*}

Now $P^{\pm} = (p^{\pm}_{\Bla,\Bmu}(t))$ can be modified to the matrices 
\begin{align*}
(t^{-a(\Bla)}p^-_{\Bla,\Bmu}(t)) &= 
     \begin{pmatrix}
          1  &   &   &   &   &   &   &   &  \\
          1  & 1 &   &   &   &   &   &   &  \\
          1  & 0 & 1 &   &   &   &   &   &   \\ 
          1  & 1 & 0 & 1 &   &   &   &   &   \\
      t^3+1  & 1 & 1 & 0 & 1 &   &   &   &   \\
      t^3+1  & 1 & 1 & 1 & 1 & 1 &   &   &   \\
          1  & 1 & 1 & 0 & 1 & 0 & 1 &   &   \\
      t^3+1  & t^3+1 & 1 & 1 & 1 & 1 & 1 & 1 &  \\
          1  & 1 & 1 & 1 & 1 & 1 & 1 & 1 & 1
  \end{pmatrix}, \\ \\
(t^{-a(\tau(\Bla))-a(\Bmu) + a(\tau(\Bmu))}p^+_{\Bla,\Bmu}(t)) &=
     \begin{pmatrix}
          1  &   &   &   &   &   &   &   &  \\
          1  & 1 &   &   &   &   &   &   &   \\
          1  & 0 & 1 &   &   &   &   &   &   \\
          1  & 0 & 0 & 1 &   &   &   &   &   \\
      t^3+1  & 1 & 1 & 0 & 1 &   &   &   &   \\
      t^3+1  & 0 & 1 & 1 & 0 & 1 &   &   &   \\
          1  & 1 & 1 & 0 & 1 & 0 & 1 &   &   \\
      t^3+1  & 1 & 1 & t^3 & 1 & 0 & 0 & 1 &  \\
          1  & 0 & 1 & 1 & 0 & 1 & 0 & 0 & 1          
     \end{pmatrix}.
\end{align*}
In this case, it is possible to compute $\IC^{\pm}_{\Bla,\Bmu}(t)$ directly.
One can check that $(t^{-a(\Bla)}p^-_{\Bla,\Bmu}(t))$  coincides with 
$(\IC^-_{\Bla,\Bmu}(t^3))$.  
This shows that (6.10.1) holds, and so, by Remarks 6.10 (i), the 
condition (A) holds in this case.  
On the other hand, we have

\begin{align*}
\IC^+_{\Bla,\Bmu}(t^3)) = 
     \begin{pmatrix}
          1  &   &   &   &   &   &   &   &  \\
          1  & 1 &   &   &   &   &   &   &   \\
          1  & 0 & 1 &   &   &   &   &   &   \\
          1  & 0 & 0 & 1 &   &   &   &   &   \\
      t^3+1  & 1 & 1 & 0 & 1 &   &   &   &   \\
      t^3+1  & 0 & 1 & 1 & 0 & 1 &   &   &   \\
          1  & 1 & 1 & 0 & 1 & 0 & 1 &   &   \\
      t^3+1  & t^3+1 & 1 & 1 & 0 & 1 & 1 & 1 &  \\
          1  & 0 & 1 & 1 & 0 & 1 & 0 & 0 & 1          
     \end{pmatrix}.
\end{align*}
In contrast to the cases in 7.1 and 7.2, 
$(t^{-a(\tau(\Bla))-a(\Bmu) + a(\tau(\Bmu))}p^+_{\Bla,\Bmu}(t))$ does not 
coincide with $(\IC^+_{\Bla,\Bmu}(t^3))$.  

\para{7.4.}
We give the table of $P^{\pm}$ and $\vL$ for the case where 
$n = 3$ and $r = 3$.   We fix the total order on $\SP_{3,3}$ as in the first 
column of the following table. 
The diagonal matrix $\vL = (\xi_{\Bla,\Bla})$ and the values 
of $a$-function are given as follows. 

\par\vspace{1cm}
\hspace*{2cm}
\begin{tabular}{c|c|c}
$\Bla$       &   $a(\Bla)$  &   $\xi_{\Bla,\Bla}$  \\
\hline
$(-,-,1^3)$  &    15  &   1   \\
$(-,1^3,-)$  &    12  &   $t^{-3}(t^9 - 1)$  \\  %$t^{-3}\Phi_1\Phi_3\Phi_9$
$(-,-,21)$   &    9   &   $(t^3 + 1)(t^9 - 1)$  \\ %$\Phi_1\Phi_2\Phi_3\Phi_6\Phi_9$
$(1^3,-,-)$  &    9   &   $t^3(t^9 - 1)$    \\  %$t^3\Phi_1\Phi_3\Phi_9$
$(-,1,1^2)$  &    8   &   $t^{-1}(t^6 - 1)(t^9 - 1)$   \\  %$t^{-1}\Phi_1^2\Phi_2\Phi_3^2\Phi_6\Phi_9$
$(-,-,3)$    &    6   &   $t^{3}(t^6 - 1)(t^9 - 1)$  \\  %$\Phi_1^2\Phi_2\Phi_3^2\Phi_6\Phi_9$ 
$(1,-,1^2)$  &    7   &   $t(t^6-1)(t^9 - 1)$  \\  %$\Phi_1^2\Phi_2\Phi_3^2\Phi_6\Phi_9$ 
$(-,1^2,1)$  &    7   &   $t(t^6 - 1)(t^9 - 1)$  \\  %$\Phi^2_1\Phi_2\Phi_3^2\Phi_6\Phi_9$
$(1^2,-,1)$  &    5   &   $t^5(t^6 - 1)(t^9 - 1)$ \\  %$\Phi_1^2\Phi_2\Phi_3^2\Phi_6\Phi_9$ 
$(-,21,-)$   &    6   &   $t^3(t^6 - 1)(t^9 - 1)$  \\  % $\Phi_1^2\Phi_2\Phi_3^2\Phi_6\Phi_9$ 
$(-,1,2)$    &    5   &   $t^2(t^3 - 1)(t^6-1)(t^9 - 1)$ \\  % $\Phi_1^3\Phi_2\Phi_3^3\Phi_6\Phi_9$ 
$(1,1^2,-)$  &    5   &   $t^5(t^6 - 1)(t^9 - 1)$  \\  % $ \Phi_1^2\Phi_2\Phi_3^2\Phi_6\Phi_9$ 
$(1,-,2)$    &    4   &   $t^4(t^3 -1)(t^6 - 1)(t^9 - 1)$ \\ % $\Phi_1^3\Phi_2\Phi_3^3\Phi_6\Phi_9$ 
$(-,2,1)$    &    4   &   $t^4(t^3 -1)(t^6 - 1)(t^9 - 1)$ \\% $ \Phi_1^3\Phi_2\Phi_3^3\Phi_6\Phi_9$ 
$(1^2,1,-)$  &    4   &   $t^4(t^3 -1)(t^6 - 1)(t^9 - 1)$ \\% $\Phi_1^3\Phi_2\Phi_3^3\Phi_6\Phi_9$  
$(-,3,-)$    &    3   &   $t^6(t^3 -1)(t^6 - 1)(t^9 - 1)$ \\% $\Phi_1^3\Phi_2\Phi_3^3\Phi_6\Phi_9$  
$(1,1,1)$    &    2   &   $t^8(t^3 -1)(t^6 - 1)(t^9 - 1)$ \\% $\Phi_1^3\Phi_2\Phi_3^3\Phi_6\Phi_9$  
$(21,-,-)$   &    3   &   $t^9(t^3 -1)(t^6 - 1)(t^9 - 1)$ \\% $\Phi_1^3\Phi_2\Phi_3^2\Phi_6\Phi_9$  
$(1,2,-)$    &    2   &   $t^8(t^3 -1)(t^6 - 1)(t^9 - 1)$ \\%  $\Phi_1^3\Phi_2\Phi_3^3\Phi_6\Phi_9$ 
$(2,-,1)$    &    2   &   $t^8(t^3 -1)(t^6 - 1)(t^9 - 1)$ \\% $\Phi_1^3\Phi_2\Phi_3^3\Phi_6\Phi_9$  
$(2,1,-)$    &    1   &   $t^{10}(t^3 -1)(t^6 - 1)(t^9 - 1)$ \\ % $ \Phi_1^3\Phi_2\Phi_3^3\Phi_6\Phi_9$ 
$(3,-,-)$    &    0   &   $t^{12}(t^3 -1)(t^6 - 1)(t^9 - 1)$ \\ % $ \Phi_1^3\Phi_2\Phi_3^3\Phi_6\Phi_9$.     
\hline
\end{tabular}

\newpage
\normalsize
\begin{center}
$P^-$ for $n = 3, r= 3$ 
\end{center}
\bigskip
\footnotesize 
\par
\hspace*{2cm}
\begin{sideways}
\begin{tabular}{|c|cccccccccc|}
\hline
\phantom{$\displaystyle\prod$} &  $(-,-,1^3)$ & $(-,1^3,-)$ & $(-,-,21)$  & $(1^3,-,-)$ & $(-,1,1^2)$ 
      &  $(-,-,3)$   & $(1,-,1^2)$ & $(-,1^2,1)$  &  $(1^2,-,1)$   &  $(-,21,-)$  \\ 
%& $(1^2,-,1)$ & $(-,21,-)$
%      &  $(-,1,2)$   & $(1,1^2,-)$ & $(1,-,2)$   & $(-,2,1)$   & $(1^2,1,-)$
%      &  $(-,3,0)$   & $(1,1,1)$   & $(21,-,-)$  & $(1,2,-)$   & $(2,-,1)$
%      &  $(2,1,-)$   & $(3,-,-)$  \\   
\hline
$(-,-,1^3)$   &  $t^{15}$  &  \phantom{$\displaystyle\prod$} &     &    &     &    &    &   &  &\\
$(-,1^3,-)$   &  $t^{12}$       &  $t^{12}$ &         &           &    &    &    &   &  &  \\
$(-,-,21)$    &  $t^{12} + t^9$ &   0       &  $t^9$  &           &    &    &    &   &  &  \\
$(1^3,-,-)$   &  $t^9$          &  $t^9$    &    0    &    $t^9$  &    &    &    &   &  & \\
$(-,1,1^2)$   &  $t^{14} + t^{11} + t^8$   &   $t^8$  &  $t^8$   &  0  &  $t^8$  
              &    &    &   &   &  \\
$(-,-,3)$     &  $t^6$  &  0   &   $t^6$   &  0   &   0  &  $t^6$  &    &  &  &  \\
$(1,-,1^2)$   &  $t^{13} + t^{10} + t^7$   &  $t^7$  &  $t^7$  &   $t^7$  &  $t^7$  
              &  0  &   $t^7$  &   &  &  \\   
$(-,1^2,1)$   &  $t^{13} + t^{10} + t^7$  &  $t^{10} + t^7$  &  $t^7$  &  0  
              &  $t^7$   &  0   &   0  &  $t^7$  &  &  \\  
$(1^2,-,1)$   &  $t^{11} + t^8 + t^5$  &  $t^8 + t^5$  &  $t^5$  &  $t^8 + t^5$ 
              &  $t^5$   &   0  &  $t^5$   &  $t^5$  &   $t^5$  &   \\
$(-,21,-)$    &  $t^9 + t^6$   &  $t^9 + t^6$  &   $t^6$  &  0   &   $t^6$  
              &  0    &   0  &   $t^6$   &   0   &  $t^6$  \\
$(-,1,2)$     &  $t^{11} + t^8 + t^5$   &  $t^5$  &  $t^8 + t^5$  &  0
              &  $t^5$   &  $t^5$  &  0  & $t^5$  &   0   &  0  \\
$(1,1^2,-)$   &  $t^{11} + t^8 + t^5$   &  $t^{11} + t^8 + t^5$  &  $t^5$
              &   $t^5$   &   $t^5$  &   0  &  $t^5$   &  $t^5$  &    0  &  $t^5$  \\
$(1,-,2)$     &  $t^{10} + t^7 + t^4$  &  $t^4$   &  $t^7 + t^4$   &   $t^4$  
              &   $t^4$   &   $t^4$    &  $t^4$  &  $t^4$   &   $t^4$  &  0  \\
$(-,2,1)$     &  $t^{10} + t^7 + t^4$  &  $t^7 + t^4$  &  $t^7 + t^4$  &   0
              &  $t^7 + t^4$   &  $t^4$  &   0   &  $t^4$   &  0  &  $t^4$  \\   
$(1^2,1,-)$   &  $t^{10} + t^7 + t^4$   &  $t^{10} + t^7 + t^4$   &  $t^4$
              &  $t^7 + t^4$    &   $t^4$  &  0  &   $t^4$  &  $t^4$  &  $t^4$  &  $t^4$   \\ 
$(-,3,-)$     &  $t^3$   &  $t^3$   &   $t^3$  &   0  &  $t^3$  &  $t^3$  
              &   0   &  $t^3$  &  0  &   $t^3$  \\ 
$(1,1,1)$     &  $t^{12} + 2t^9 + 2t^6 + t^3$  &   $t^9 + 2t^6 + t^3$   &  $2t^6 + t^3$
              &  $t^6 + t^3$   &  $2t^6 + t^3$   &   $t^3$  &  $t^6 + t^3$  &  $t^6 + t^3$  
              &  $t^3$  &   $t^3$  \\
$(21,-,-)$    &  $t^6 + t^3$   &   $t^6 + t^3$   &   $t^3$  &  $t^6 + t^3$  &  $t^3$  &   0
              &  $t^3$   &   $t^3$   &  $t^3$  &   $t^3$  \\  
$(1,2,-)$     &  $t^8 + t^5 + t^2$    &   $t^8 + t^5 + t^2$  &   $t^5 + t^2$  &  $t^2$ 
              &  $t^5 + t^2$   &   $t^2$   &  $t^2$   &   $t^5 + t^2$  &  $t^2$  &  $t^5 + t^2$  \\
$(2,-,1)$     &  $t^8 + t^5 + t^2$   &   $t^5 + t^2$   &   $t^5 + t^2$  &   $t^5 + t^2$
              &  $t^5 + t^2$   &   $t^2$ &  $t^5 + t^2$  &   $t^2$  &  $t^2$ &  $t^2$  \\
$(2,1,-)$     &  $t^7 + t^4 + t$   &  $t^7 + t^4 + t$   &   $t^4 + t$  &  $t^4 + t$ 
              &  $t^4 + t$   &   $t$   &   $t^4 + t$  &   $t^4 + t$  &   $t$  &  $t^4 + t$  \\
$(3,-,-)$     &  1  &   1   &   1   &   1   &  1  &  1  &  1  &  1  &  1  &  1  \\       
\hline
\end{tabular}
\end{sideways}

\newpage
\normalsize
\begin{center}
$P^-$ for $n = 3, r= 3$ (continued) 
\end{center}
\medskip
\footnotesize 
\par
\hspace*{2cm}
\begin{sideways}
\begin{tabular}{|c|cccccccccccc|}
\hline
\phantom{$\displaystyle\prod$}   
      &  $(-,1,2)$   & $(1,1^2,-)$ & $(1,-,2)$   & $(-,2,1)$   & $(1^2,1,-)$
      &  $(-,3,-)$   & $(1,1,1)$   & $(21,-,-)$  & $(1,2,-)$   & $(2,-,1)$
      &  $(2,1,-)$   & $(3,-,-)$  \\   
\hline
$(-,-,1^3)$   & \phantom{$\displaystyle\prod$}  &&&&&&&&&&&    \\  
$(-,1^3,-)$   &                                 &&&&&&&&&&&     \\
$(-,-,21)$   &&&&&&&&&&&&   \\
$(1^3,-,-)$  &&&&&&&&&&&&   \\
$(-,1,1^2)$  &&&&&&&&&&&&   \\
$(-,-,3)$    &&&&&&&&&&&&   \\
$(1,-,1^2)$  &&&&&&&&&&&&   \\
$(-,1^2,1)$  &&&&&&&&&&&&   \\
$(1^2,-,1)$  &&&&&&&&&&&&   \\
$(-,21,-)$   &&&&&&&&&&&&   \\
$(-,1,2)$    &  $t^5$  &&&&&&&&&&&  \\
$(1,1^2,-)$  &  0   &  $t^5$  &&&&&&&&&&  \\
$(1,-,2)$    &  $t^4$  &  0  &  $t^4$  &&&&&&&&&  \\
$(-,2,1)$    &  $t^4$  &  0  &  0  &  $t^4$  &&&&&&&&  \\
$(1^2,1,-)$  &  0  &  $t^4$  &  0  &  0  &  $t^4$  &&&&&&&  \\
$(-,3,-)$    &  $t^3$  &  0  &  0  &  $t^3$  &  0  &  $t^3$  &&&&&&  \\
$(1,1,1)$    &  $t^3$  &  $t^3$  &  $t^3$  &  $t^3$  &  $t^3$  
             &   0  &  $t^3$  &&&&&  \\
$(21,-,-)$   &   0  &  $t^3$  &  0  &  0  &  $t^3$  &  0  
             &  0  &  $t^3$  &&&&  \\
$(1,2,-)$    &  $t^2$  &  $t^2$  &  $t^2$  &  $t^2$  
             &  $t^2$  &  $t^2$   &   $t^2$  &  0  &  $t^2$  &&&  \\
$(2,-,1)$    &  $t^2$  &  $t^2$  &  $t^2$  &  $t^2$  &  $t^2$  &  0
             &  $t^2$  &  $t^2$  &  0  &  $t^2$  &&  \\
$(2,1,-)$    &  $t$  &  $t^4 + t$  &  $t$  &   $t$  &  $t$  &  $t$  
             &  $t$  &   $t$   &  $t$   &  $t$  &  $t$  &  \\
$(3,-,-)$    &   1   &   1  &  1  &  1  &  1  &  1  
             &   1   &   1  &  1  &  1  &  1  &  1  \\
\hline   
\end{tabular}
\end{sideways}

\newpage
\normalsize
\begin{center}
$P^+$ for $n = 3, r= 3$ 
\end{center}
\medskip
\footnotesize 
\par
\hspace*{2cm}
\begin{sideways}
\begin{tabular}{|c|cccccccccc|}
\hline
\phantom{$\displaystyle\prod$} &  $(-,-,1^3)$ & $(-,1^3,-)$ & $(-,-,21)$  & $(1^3,-,-)$ & $(-,1,1^2)$ 
      &  $(-,-,3)$   & $(1,-,1^2)$ & $(-,1^2,1)$  &  $(1^2,-,1)$   &  $(-,21,-)$  \\ 
\hline
$(-,-,1^3)$   &  $t^{15}$  &  \phantom{$\displaystyle\prod$} &     &    &     &    &    &   &  &\\
$(-,1^3,-)$   &  $t^{9}$       &  $t^{12}$ &         &           &    &    &    &   &  &  \\
$(-,-,21)$    &  $t^{12} + t^9$ &   0       &  $t^9$  &           &    &    &    &   &  &  \\
$(1^3,-,-)$   &  $t^{12}$          &  0   &    0    &    $t^9$  &    &    &    &   &  & \\
$(-,1,1^2)$   &  $t^{13} + t^{10} + t^7$   &   $t^{10}$  &  $t^7$   &  0  &  $t^8$  
              &    &    &   &   &  \\
$(-,-,3)$     &  $t^6$  &  0   &   $t^6$   &  0   &   0  &  $t^6$  &    &  &  &  \\
$(1,-,1^2)$   &  $t^{14} + t^{11} + t^8$   &  0   &  $t^8$  &   $t^5$  &  0  
              &  0  &   $t^7$  &   &  &  \\   
$(-,1^2,1)$   &  $t^{11} + t^{8} + t^5$  &  $t^{11} + t^8$  &  $t^5$  &  0  
              &  $t^6$   &  0   &   0  &  $t^7$  &  &  \\  
$(1^2,-,1)$   &  $t^{13} + t^{10} + t^7$  &  0  &  $t^7$  &  $t^7 + t^4$ 
              &  0   &   0  &  $t^6$   &  0  &   $t^5$  &   \\
$(-,21,-)$    &  $t^6 + t^3$   &  $t^9 + t^6$  &   $t^3$  &  0   &   $t^4$  
              &  0    &   0  &   $t^5$   &   0   &  $t^6$  \\
$(-,1,2)$     &  $t^{10} + t^7 + t^4$   &  $t^7$  &  $t^7 + t^4$  &  0
              &  $t^5$   &  $t^4$  &  0  & $t^6$  &   0   &  0  \\
$(1,1^2,-)$   &  $t^{10} + t^7 + t^4$   &  $t^{10} + t^7$  &  $t^4$
              &   $t^7$   &   $t^5$  &   0  &  0   &  $t^6$  &    0  &  0  \\
$(1,-,2)$     &  $t^{11} + t^8 + t^5$  &  0   &  $t^8 + t^5$   &   $t^2$  
              &   0   &   $t^5$    &  $t^4$  &  0   &   $t^3$  &  0  \\
$(-,2,1)$     &  $t^{8} + t^5 + t^2$  &  $t^8 + t^5$  &  $t^5 + t^2$  &   0
              &  $t^6 + t^3$   &  $t^2$  &   0   &  $t^4$   &  0  &  $t^5$  \\   
$(1^2,1,-)$   &  $t^{11} + t^8 + t^5$   &  $t^{8}$   &  $t^5$
              &  $t^8 + t^5$    &   $t^6$  &  0  &   0  &  0  &  $t^3$  &  0   \\ 
$(-,3,-)$     &  1   &  $t^3$   &   1  &   0  &  $t$  &  1  
              &   0   &  $t^2$  &  0  &   $t^3$  \\ 
$(1,1,1)$     &  $t^{12} + 2t^9 + 2t^6 + t^3$  &   $t^9 + t^6$   &  $2t^6 + t^3$
              &  $t^6 + t^3$   &  $t^7 + t^4$   &   $t^3$  &  $t^5$  &  $t^5$  
              &  $t^4$  &   0  \\
$(21,-,-)$    &  $t^9 + t^6$   &   0  &   $t^6$  &  $t^6 + t^3$  &  0  &   0
              &  $t^5$   &   0   &  $t^4$  &   0  \\  
$(1,2,-)$     &  $t^7 + t^4 + t$    &   $t^7 + t^4$  &   $t^4 + t$  &  $t^4$ 
              &  $t^5 + t^2$   &   $t$   &  0   &   $t^3$  &  $t^2$  &  $t^4$  \\
$(2,-,1)$     &  $t^{10} + t^7 + t^4$   &   0  &   $t^7 + t^4$  &   $t^4 + t$
              &  0   &   $t^4$ &  $t^6 + t^3$  &   0  &  $t^2$ &  0  \\
$(2,1,-)$     &  $t^8 + t^5 + t^2$   &  $t^5$   &   $t^5 + t^2$  &  $t^5 + t^2$ 
              &  $t^3$   &   $t^2$   &   $t^4$  &   $t^4$  &   $t^3$  &  0  \\
$(3,-,-)$     &  $t^3$  &   0   &   $t^3$  &   1   &  0  &  $t^3$  &  $t^2$  &  0  
              &  $t$   &  0  \\       
\hline
\end{tabular}
\end{sideways}

\newpage
\normalsize
\begin{center}
$P^+$ for $n = 3, r= 3$ (continued) 
\end{center}
\medskip
\footnotesize 
\par
\hspace*{2cm}
\begin{sideways}
\begin{tabular}{|c|cccccccccccc|}
\hline
\phantom{$\displaystyle\prod$}   
      &  $(-,1,2)$   & $(1,1^2,-)$ & $(1,-,2)$   & $(-,2,1)$   & $(1^2,1,-)$
      &  $(-,3,-)$   & $(1,1,1)$   & $(21,-,-)$  & $(1,2,-)$   & $(2,-,1)$
      &  $(2,1,-)$   & $(3,-,-)$  \\   
\hline
$(-,-,1^3)$   & \phantom{$\displaystyle\prod$}  &&&&&&&&&&&    \\  
$(-,1^3,-)$   &                                 &&&&&&&&&&&     \\
$(-,-,21)$   &&&&&&&&&&&&   \\
$(1^3,-,-)$  &&&&&&&&&&&&   \\
$(-,1,1^2)$  &&&&&&&&&&&&   \\
$(-,-,3)$    &&&&&&&&&&&&   \\
$(1,-,1^2)$  &&&&&&&&&&&&   \\
$(-,1^2,1)$  &&&&&&&&&&&&   \\
$(1^2,-,1)$  &&&&&&&&&&&&   \\
$(-,21,-)$   &&&&&&&&&&&&   \\
$(-,1,2)$    &  $t^5$  &&&&&&&&&&&  \\
$(1,1^2,-)$  &  0   &  $t^5$  &&&&&&&&&&  \\
$(1,-,2)$    &  0   &  0  &  $t^4$  &&&&&&&&&  \\
$(-,2,1)$    &  $t^3$  &  0  &  0  &  $t^4$  &&&&&&&&  \\
$(1^2,1,-)$  &  0  &  $t^3$  &  0  &  0  &  $t^4$  &&&&&&&  \\
$(-,3,-)$    &  $t$  &  0  &  0  &  $t^2$  &  0  &  $t^3$  &&&&&&  \\
$(1,1,1)$    &  $t^4$  &  $t^4$  &  0 &  0  &  0  
             &   0  &  $t^3$  &&&&&  \\
$(21,-,-)$   &   0  &  0  &  0  &  0  &  0  &  0  
             &  0  &  $t^3$  &&&&  \\
$(1,2,-)$    &  $t^2$  &  $t^2$  &  0  &  $t^3$  
             &  $t^3$  &  0   &   0  &  0  &  $t^2$  &&&  \\
$(2,-,1)$    &  0  &  0  &  $t^3$  &  0  &  0  &  0
             &  0  &  $t$  &  0  &  $t^2$  &&  \\
$(2,1,-)$    &  $t^3$  &  $t^3$  &  0 &  0  &  0  &  0 
             &  $t^2$  &   $t^2$   &  0   &  0  &  $t$  &  \\
$(3,-,-)$    &   0   &   0  &  $t^2$  &  0  &  0  &  0  
             &   0   &   1  &  0  &  $t$  &  0  &  1  \\
\hline   
\end{tabular}
\end{sideways}

%%%%

\par\vspace{1cm}
\noindent
T. Shoji \\
Department of Mathematics, Tongji University \\ 
1239 Siping Road, Shanghai 200092, P. R. China  \\
E-mail: \verb|shoji@tongji.edu.cn|

\end{document}